\documentclass[11pt,reqno]{amsart}
\usepackage[utf8]{inputenc}
\usepackage{amsmath,amssymb,amsthm,mathrsfs, color}
\usepackage{a4wide}
\usepackage{graphics,subfigure}
\usepackage{epsfig}

\usepackage[colorlinks=true]{hyperref}
\hypersetup{%
    ,urlcolor=blue
    ,citecolor=red
    ,linkcolor=blue
    }

\newtheorem{theorem}{Theorem}[section]
\newtheorem{prop}[theorem]{Proposition}
\newtheorem{lemma}[theorem]{Lemma}
\newtheorem{corollary}[theorem]{Corollary}

\newtheorem{remark}[theorem]{Remark}

\newcommand{\cW}{\mathcal{W}}
\newcommand{\cS}{\mathcal{S}}

\newcommand{\D}{\mathcal{D}}
\newcommand{\R}{{\mathbb R}}
\newcommand{\N}{{\mathbb N}}
\newcommand{\E}{{\mathbb E}}

\newcommand{\cE}{{\mathcal E}}

\newcommand{\xcn}{{\check X^n}}
\newcommand{\lcn}{{\check L^n}}
\newcommand{\covX}{{\mathcal{C}^X}}
\newcommand{\Var}{{\rm Var}}
\newcommand{\vn}{{\bf v}_n(\alpha)}
\newcommand{\rb}{{\bf r}}
\newcommand{\Del}{\Delta^\circ}

\def\cY{\mathcal Y}

\def\cov{\textup{cov}}

\makeatletter
 \@addtoreset{equation}{section}
 \makeatother
 
\allowdisplaybreaks
\usepackage{appendix}
\usepackage{hyperref}

\title[Weak error for rough and Gaussian volatility models]{Weak error approximation for rough and Gaussian mean-reverting stochastic volatility models}
\date{\today}
\author{Aurélien Alfonsi}
\address{CERMICS, ENPC, Institut Polytechnique de Paris, CNRS, Marne-la-Vallée, France \& MathRisk team-project, Inria Paris, France. }
\email{aurelien.alfonsi@enpc.fr}

\author{Ahmed Kebaier}
\thanks{The authors  benefited from the support of the ``chaire Risques financiers'', Fondation du Risque.}
\address{LaMME, CNRS, UMR 8071, Université \'Evry Paris Saclay, 91037, \'Evry, France.}
\email{ahmed.kebaier@univ-evry.fr}


\begin{document}

%%%%%%%%%%%%%%%%%%%%%%%%%%%%%%%%%%%%%%%%%%%%%%%%%%%%%%%%%%%
%%%%%%%%%%%%%%%%%%%%%%%%%%%%%%%%%%%%%%%%%%%%%%%%%%%%%%%%%%%

\begin{abstract}
For a class of stochastic models with Gaussian and rough mean-reverting volatility that embeds the genuine rough Stein-Stein model, we study the weak approximation rate when using a Euler type scheme with integrated kernels. Our first result is a weak convergence rate for the discretised rough Ornstein-Uhlenbeck process, that is essentially in $\min(3\alpha-1,1)$, where $\frac{t^{\alpha-1}}{\Gamma(\alpha)} $ is the fractional convolution kernel with $\alpha \in (1/2,1)$. Then, our main result is to obtain the same convergence rate for the corresponding stochastic rough volatility model with polynomial test functions. 
\end{abstract}

\maketitle
%%%%%%%%%%%%%%%%%%%%%%%%%%%%%%%%%%%%%%%%%%%%%%%%%%%%%%%%%%%

\section{Introduction}

In recent years, the modelling of financial markets has evolved significantly beyond classical stochastic differential equations (SDEs) to account for more realistic dynamics, including rough volatility and long memory effects. Traditional models, such as the Heston model or the Black-Scholes framework, often fail to capture the persistent roughness observed in empirical volatility time series Gatheral et al. \cite{GatJaiRos}. To address these limitations, stochastic Volterra equations (SVEs) have emerged as a powerful tool, allowing for non-Markovian dynamics and flexible autocorrelation structures.

Among the various approaches within this framework, the Stein-Stein stochastic volatility model~\cite{Stein-Stein,SchZhu} has been revisited in the context of rough volatility, providing a tractable yet realistic description of asset price dynamics. Recent studies by Bayer et al. \cite{BaFrGa},  El Euch and Rosenbaum \cite{EleRos} have demonstrated that rough volatility models, particularly those driven by fractional Brownian motion, outperform traditional models in capturing the observed roughness and persistence in volatility. The extension of the mean-reverting Stein-Stein model to a stochastic Volterra (see e.g. Abi Jaber \cite{AJ} and very recently Abi Jaber et al.~\cite{AJHM,AJGu})  setting offers a promising direction, combining analytical tractability with empirical accuracy.
This paper investigates the weak error of Euler-type approximation schemes for a generalized rough stochastic volatility version of the Stein-Stein model.

More precisely, we focus on the weak error for asset log-price models $L$ with a Gaussian mean-reverting rough volatility process.  Namely, we are interested in the following model:
\begin{align}
\begin{cases}     L_t &= L_0+ \int_0^t b(X_s) ds+ \int_0^t f(X_s)dB_s, \\
    X_t&=x_0+\int_0^t \frac{(t-s)^{\alpha-1}}{\Gamma(\alpha)} (\kappa_1 +\kappa_2 X_s)ds +\sigma \int_0^t \frac{(t-s)^{\alpha-1}}{\Gamma(\alpha)} dW_s, 
\end{cases}\label{SVE_OU}
\end{align}  
where $x_0,  \kappa_1, \kappa_2 ,L_0 \in \R$, $\sigma>0$, $\rho \in [-1,1]$, $\alpha \in (1/2,1]$,  $B_t=(\rho W_t +\sqrt{1-\rho^2} W^\bot_t)$ with $W$ and $W^\bot$ independent real Brownian motions on a given filtered probability space $(\Omega,\mathcal{F},(\mathcal{F}_t)_{t\ge 0},\mathbb{P})$ and $b,f:\R \to \R_+$ are  smooth functions. The case $b=-\frac 12 f^2$ corresponds to a stochastic volatility model. In particular, when $f(x)=x$ this model encompasses the rough Stein-Stein model, see~\cite{AJ}. Here, without loss of generality, we assume a zero interest rate. 
We consider a time horizon~$T>0$ and a regular time grid $t_k=\frac{kT}{n}$ for $n\ge 1$ and $k\in \{0,\dots,n\}$. For $t \in [0,T]$, we denote $\eta(t)=t_k$ if $t_k\le t<t_{k+1}$ and consider the following scheme for $t\in [0,T]$
\begin{align}\label{eq:scheme_price}%\label{euler_KI}
\begin{cases}
 \check L^n_t&=L_0 + \int_0^t b(\xcn_{\eta(s)})ds+\int_0^t f(\xcn_{\eta(s)})dB_s, \\
  \check X^n_t&=x_0+\int_0^t \frac{(t-s)^{\alpha-1}}{\Gamma(\alpha)} (\kappa_1+\kappa_2 \xcn_{\eta(s)})ds+ \sigma \int_0^t \frac{(t-s)^{\alpha-1}}{\Gamma(\alpha)}  dW_s.
\end{cases} 
  \end{align}
The analysis of numerical approximation methods for rough SDEs is known to be a difficult issue, which is a current hot topic of research.  Fukasawa and Ugai~\cite{FuUg} have studied, for more general coefficients, the limit distribution of $n^{\alpha-1/2}(\check X^n_T-X_T)$. Jourdain and Pagès~\cite{JoPa} have obtained the strong error rate for more general kernels and coefficients, which is in $O(n^{1/2-\alpha})$ for the rough kernel with time-homogeneous coefficients. These results give  a straightforward upper bound for the weak error on the volatility process 
$$\E[\Psi(\check X^n_T)]-\E[\Psi(X_T)],$$
for Lipschitz test functions $\Psi:\R \to \R$. In the literature, the weak error has been mostly studied for the asset price component of the model, namely $\E[\Phi(\check L^n_T)]-\E[\Phi(L_T)]$, for $\Phi:\R \to \R$. However, up to our knowledge, these studies bring only for the particular case $\kappa_1=\kappa_2=0$, where there is no approximation error on the volatility process but only on the process $L$. These works also assume that $b=0$.  More precisely, Bayer et al.~\cite{BHT} have obtained a weak error rate of $O(n^{-\alpha})$ for $f(x)=x$ and $\rho=1$. Bayer et al.~\cite{BFN} have further shown by using Malliavin calculus a rate of $O(n^{1-2\alpha})$, but for a larger family of functions $f$ (namely $\mathcal{C}^2$ with a polynomial growth assumption).  A decisive step has then been reached by Gassiat~\cite{Gassiat} who has shown a sharper weak error rate of $O({\bf v}_n(\alpha))$ for the cubic test function, where
\begin{equation}\label{vitesse}
  {\bf v}_n(\alpha)=\begin{cases} \frac 1 n \text{ if } \alpha \in (2/3,1) \\
  \frac{\log(n)}n  \text{ if } \alpha =2/3 \\
\frac 1 {n^{3\alpha-1}} \text{ if } \alpha \in (1/2,2/3).\end{cases}
\end{equation}
More recently, Friz et al.~\cite{FSW} have generalized his analysis to encompass polynomial test functions~$\Phi$, smooth functions $f$ with derivatives having a sub-exponential growth, and $\rho \in [-1,1]$. Besides,  Bonesini et al.~\cite{BJP} obtain the same rate for the same model but for more general test functions~$\Phi$, by using a different analysis based on path-dependent PDE. A nice summary of these different rates of convergence is given in a table of~\cite[Section 1.2]{BJP}, with the Hurst exponent  $H=\alpha-1/2 \in (0,1/2)$.

In contrast with these studies, we consider coefficients $\kappa_1,\kappa_2\in \R$ and $b$ possibly different from $0$, which allows to reach the genuine rough Stein-Stein model~\cite[Eq.~(1.5) and (1.6)]{AJ}. It allows to have a stationary distribution for the volatility~$X$ when $\kappa_2<0$, see Corollary~\ref{cor_stationarity}, which is an interesting feature for modelling purposes. Note that when $\kappa_2\not=0$, the scheme~\eqref{eq:scheme_price} introduces a discretisation error on the volatility that we handle in this paper.   Our first contribution in this paper is to analyse the weak convergence of the rough Volterra Ornstein-Uhlenbeck process that represents the volatility in~\eqref{SVE_OU}. We prove that for any $\mathcal{C}^2$ test functions~$\Psi$ such that $\Psi$, $\Psi'$ and $\Psi''$  have exponential growth, $\E[\Psi(\check X^n_T)]-\E[\Psi(X_T)]=O(\vn)$, see Theorem~\ref{thm_covariance_error}.   This is the main result of Section~\ref{Sec_rOU}, that also presents key tools to analyse the weak approximation error on the process~$L$. Then, Section~\ref{Sec_cubic} presents the study of the weak error for $\E[(\check L^n_T)^3]-\E[(L_T)^3]$ and $b=0$. It is noticed in~\cite[Proposition 6.1]{FSW} that the cubic case allows to get familiar with the structure of the weak error rate, avoiding too much heavy notation, and prepares for tackling the general polynomial case. We prove in Theorem~\ref{thm_cubic} that $\E[(\check L^n_T)^3]-\E[(L_T)^3]=O(\vn)$ when the coefficient $f$ is $\mathcal{C}^3$, with $f^{(i)}$ having exponential growth for $0\le i\le 3$. Section~\ref{Sec_polynomial} is dedicated to our main result (Theorem~\ref{thm_main}): $\E[\Phi(\check L^n_T)]-\E[\Phi(L_T)]=O(\vn)$ for any polynomial test function $\Phi$, when the coefficients $b$ and $f$ are $\mathcal{C}^\infty$ with all derivatives of exponential growth. For this last result, we extend the framework proposed by Friz et al.~\cite{FSW} to our setting. Finally, the two appendix sections are devoted  to useful results and  technical proofs.

\section{Convergence for the rough Volterra Ornstein-Uhlenbeck process}\label{Sec_rOU}

We focus in this section on the main properties of the rough volatility process. The first subsection is devoted to the process $X$ itself and calculates its expectation and covariance functions as well as its Malliavin derivative. We also prove a crucial estimate for the analysis of the weak error, see Theorem~\ref{thm_technique}. The second subsection shows a weak error results on the rough Ornstein-Uhlenbeck process. To prove it, we establish key estimates on the error between the process and its approximation~$\xcn$ for the expectation and the covariance. These estimates will be essential also in the next sections to estimate the weak error in the rough Stein and Stein model.

We recall few classical results concerning the special functions that we will use through the paper. Let $\Gamma(\alpha)=\int_0^\infty t^{\alpha-1} e^{-t}dt$ denote the Euler function. For $\alpha,\beta>0$, we have $\frac{\Gamma(\alpha+\beta)} {\Gamma(\alpha)\Gamma(\beta)}\int_0^1 x^{\alpha-1}(1-x)^{\beta-1}dx=1$ and therefore for any $0\le s< t$,
\begin{equation}\label{Beta_identity}
  \int_s^t \frac{(u-s)^{\alpha-1}}{\Gamma(\alpha)}\frac{(t-u)^{\beta-1}}{\Gamma(\beta)} du=\frac{t^{\alpha+\beta-1}}{{\Gamma(\alpha+\beta)}},
\end{equation}
by using the change of variable $u=s+x(t-s)$. We also  recall the confluent hypergeometric function~\cite[Eq. 15.6.1 and 15.1.2]{Daalhuis} that   is defined by
\begin{equation}\label{def_2F1}
  _2 F_1(a,b,c;z)=\frac{\Gamma(c)}{\Gamma(b)\Gamma(c-b)} \int_0^1 \frac{t^{b-1}(1-t)^{c-b-1}}{(1-zt)^a}dt
\end{equation}
for $z\in[0,1)$, $c>b>0$. For $z=1$, the function can be extended continuously provided that $c-a-b>0$.

\subsection{The rough Ornstein-Uhlenbeck process}

\begin{prop}\label{prop_esp}
  Let $X$ be the solution of the Stochastic Volterra Equation~\eqref{SVE_OU}.
Then, we have  
\begin{equation}\label{esp_VOU}
  \E[X_t]= \sum_{i=0}^\infty \kappa_2^{i} \left(x_0 \frac{t^{\alpha i}}{\Gamma(\alpha i+1)}+\kappa_1 \frac{t^{\alpha(i+1)}}{\Gamma(\alpha(i+1)+1)}\right)=x_0 E_{\alpha}(\kappa_2 t^\alpha)+ \frac{\kappa_1}{\kappa_2} (E_{\alpha}(\kappa_2 t^\alpha)-1),
\end{equation}
where $E_\alpha(z)=\sum_{i=0}^\infty \frac{z^i}{\Gamma(\alpha i +1)}$ is the Mittag-Leffler function.
\end{prop}
\noindent Note that $\lim_{z\to 0} \frac{E_\alpha (z)-1}{z}= \frac{1}{\Gamma(\alpha+1)}$. Therefore, there is no singularity for $\kappa_2=0$, and we have $\E[X_t]=x_0+\frac{\kappa_1 t^\alpha}{\Gamma(\alpha+1)}$.

\begin{proof}
  By Zhang~\cite[Theorem 3.1]{Zhang}, there exists a unique strong solution of the SVE such that $\sup_{t\in[0,T]} \E[X_t^2]<\infty$ for all $T>0$. We can thus take the expectation and get
  \begin{align*}
    \E[X_t]&=x_0+\int_0^t \frac{(t-s)^{\alpha-1}}{\Gamma(\alpha)} (\kappa_1+\kappa_2 \E[X_s])ds\\
    &=x_0+ \kappa_1 \frac{t^\alpha}{\Gamma(\alpha+1)} +\kappa_2 \int_0^t \frac{(t-s)^{\alpha-1}}{\Gamma(\alpha)} \E[X_s] ds,
  \end{align*} 
and therefore $\E[X_t]$ is the unique solution of this fractional linear differential equation. From the Beta distribution identity~\eqref{Beta_identity}, we have $\int_0^t \frac{(t-s)^{\alpha-1}}{\Gamma(\alpha)} \frac{s^{\alpha i}}{\Gamma(\alpha i+1)}  ds = \frac{t^{\alpha(i+1)}}{\Gamma(\alpha(i+1)+1)}$ for $i\ge 0$, and we get 
\begin{align*}
  &\int_0^t \frac{(t-s)^{\alpha-1}}{\Gamma(\alpha)}  \sum_{i=0}^\infty \kappa_2^{i} \left(x_0 \frac{s^{\alpha i}}{\Gamma(\alpha i+1)}+\kappa_1 \frac{s^{\alpha(i+1)}}{\Gamma(\alpha(i+1)+1)}\right) ds\\
  &= \sum_{i=1}^\infty \kappa_2^{i} \left(x_0 \frac{t^{\alpha i}}{\Gamma(\alpha i+1)}+\kappa_1 \frac{t^{\alpha(i+1)}}{\Gamma(\alpha(i+1)+1)}\right), 
\end{align*}
which gives that~\eqref{esp_VOU} is the desired solution.
\end{proof}

\begin{prop}\label{prop_OUexplicit} Let $t \in \R_+$ and $T\ge t$. Let $X$ be the solution of~\eqref{SVE_OU} and $\mathcal{Y}_t=X_t-\E[X_t]$ for $t\ge 0$.
  We have 
  \begin{align}\label{Y_cal}
    \mathcal{Y}_t=\sigma \sum_{i=1}^\infty \kappa_2^{i-1} \int_0^t \frac{(t-s)^{i\alpha -1}}{\Gamma(i \alpha)} dW_s.
  \end{align} In particular, $(X_t,t\ge 0)$ is a Gaussian process with covariance function $\cov(X_t,X_T)=\covX(t,T)$ with
  \begin{align}\label{cov_formula}
    \covX(t,T)= \sigma^2 \sum_{i_1,i_2=1}^\infty \kappa_2^{i_1+i_2-2} \frac{t^{i_1 \alpha} T^{i_2 \alpha-1}}{\Gamma(i_1 \alpha+1)\Gamma(i_2 \alpha)}\ _2F_1(1-i_2 \alpha,1;i_1 \alpha+1,\frac{t}{T}).
  \end{align}
   Besides, $X_t \in \mathbb{D}^{1,2}$ (i.e. $\|X_t\|^2_{1,2}:=\E[X_t^2]+\int_0^t \E[\D_sX_t^2] ds <\infty$ see~\cite[p.~27]{Nualart}) and the Malliavin derivative $
\D_sX_t$ is  given by
  \begin{equation}\label{der_Malliavin_X}
\D_sX_t= \mathbf{1}_{s<t}    \sigma \sum_{i=1}^\infty \kappa_2^{i-1}  \frac{(t-s)^{i\alpha -1}}{\Gamma(i \alpha)}.
  \end{equation}
\end{prop}
\begin{remark}
  From~\eqref{der_Malliavin_X}, we easily get that for any $T>0$, 
\begin{equation} \label{eq_def_majCT}
  C_T=\sup_{s \in [0,T]} \int_0^s|\D_rX_s|dr< \infty
\end{equation}
and there exists a constant $C'_T$ such that for any $s,t \in [0,T]$ with $s<t$,
\begin{equation*}
  |\D_sX_t|\le C'_T(1+(t-s)^{\alpha-1}).
\end{equation*}

\end{remark}
\begin{remark}
From Propositions~\ref{prop_esp} and~\ref{prop_OUexplicit}, the distribution of the Gaussian vector $(X_{t_k})_{0\le k\le n}$ and even $(W_{t_k},X_{t_k})_{0\le k\le n}$ is known so that it can be sampled exactly, up to the truncation of the series~\eqref{cov_formula}. Then, we can simulate exactly the volatility process on a given time grid. 
In this work, we consider instead an approximation of $X$, which adds new terms in the error analysis and allows to understand the effect of discretising the volatility process. Beyond the theoretical interest, this extension is also useful in practice, for example to calculate the parameter sensitivity by Monte Carlo or for a multi-asset model. Suppose that we  approximate $\frac{d}{d\kappa_2}\E[\Phi(L_T)]\approx \E[\frac{\Phi(L^{\kappa_2+\varepsilon}_T)-\Phi(L^{\kappa_2-\varepsilon}_T)}{2\varepsilon}]$ with $0<\varepsilon \ll 1$: for variance purposes, it is crucial to sample $X$ with the parameters $\kappa_2-\varepsilon$ and $\kappa_2+\varepsilon$ with the same driving Brownian motion. If we use the exact simulation, this gives a nearly singular covariance matrix that may raise numerical issues, while the two approximation schemes can be sampled easily from $(\int_0^{t_k} \frac{(t_k-s)^{\alpha-1}}{\Gamma(\alpha)}dW_s)_{k\le n}$.        
\end{remark}

\begin{proof}
  We have $X_T=x_0+\int_0^T \frac{(T-s)^{\alpha-1}}{\Gamma(\alpha)} (\kappa_1+\kappa_2 X_s)ds + \sigma \int_0^T  \frac{(T-s)^{\alpha-1}}{\Gamma(\alpha)} dW_s$ and thus $\E[X_T]=x_0+\int_0^T \frac{(T-s)^{\alpha-1}}{\Gamma(\alpha)} (\kappa_1+\kappa_2 \E[X_s])ds$. This gives
  $$ \cY_T= \int_0^T \frac{(T-s)^{\alpha-1}}{\Gamma(\alpha)}  \kappa_2 \cY_s ds + \sigma \int_0^T  \frac{(T-s)^{\alpha-1}}{\Gamma(\alpha)}dW_s,  T\ge 0.$$
We then obtain 
$$ \cY_T=\int_0^T \frac{(T-t)^{\alpha-1}}{\Gamma(\alpha)}  \kappa_2 \left(\int_0^t \frac{(t-s)^{\alpha-1}}{\Gamma(\alpha)}  \kappa_2 \cY_s ds + \sigma \int_0^t  \frac{(t-s)^{\alpha-1}}{\Gamma(\alpha)} dW_s\right) dt + \sigma \int_0^T  \frac{(T-t)^{\alpha-1}}{\Gamma(\alpha)}dW_t.$$
By using Fubini's theorem and then~\eqref{Beta_identity}, we get 
$$\int_0^T \frac{(T-t)^{\alpha-1}}{\Gamma(\alpha)}  \int_0^t \frac{(t-s)^{\alpha-1}}{\Gamma(\alpha)}  \cY_s ds dt = \int_0^T \cY_s \frac{(T-s)^{2\alpha-1}}{\Gamma(2\alpha)} ds, $$
Similarly, the stochastic Fubini theorem (see e.g.~\cite[Theorem IV.65]{Protter}) gives
$$\int_0^T \frac{(T-t)^{\alpha-1}}{\Gamma(\alpha)}  \left(\int_0^t \frac{(t-s)^{\alpha-1}}{\Gamma(\alpha)}   dW_s\right) dt = \int_0^T  \frac{(T-s)^{2\alpha-1}}{\Gamma(2\alpha)} dW_s, $$
and we get 
$$\cY_T= \int_0^T \frac{(T-s)^{2\alpha-1}}{\Gamma(2\alpha)}  \kappa_2^2 \cY_s ds+  \sigma \sum_{i=1}^2 \kappa_2^{i-1} \int_0^T \frac{(T-s)^{i\alpha -1}}{\Gamma(i \alpha)} dW_s.$$
By induction, we get using again the same Fubini argument 
$$\cY_T= \int_0^T \frac{(T-s)^{n\alpha-1}}{\Gamma(n\alpha)}  \kappa_2^n \cY_s ds+  \sigma \sum_{i=1}^n \kappa_2^{i-1} \int_0^T \frac{(T-s)^{i\alpha -1}}{\Gamma(i \alpha)} dW_s$$
for $n\ge 1$.   The first term and the rest  $\sum_{i=n+1}^\infty \kappa_2^{i-1} \int_0^t \frac{(T-s)^{i\alpha -1}}{\Gamma(i \alpha)} dW_s$ vanish in $L^2(\Omega)$  as $n\to \infty$. We have indeed 
\begin{align*}
  \E\left[\left( \int_0^T \frac{(T-s)^{n\alpha-1}}{\Gamma(n\alpha)}  \kappa_2^n \cY_s ds\right)^2\right]&\le T \sup_{t\in [0,T]} \E[|X_t|^2] \int_0^T \frac{(T-s)^{2n\alpha-2}}{\Gamma(n\alpha)^2}  \kappa_2^{2n} ds \\
  &=T \sup_{t\in [0,T]} \E[|X_t|^2] \frac{T^{2 n\alpha-1}}{(2n\alpha-1) \Gamma(n\alpha)^2}  \kappa_2^{2n} \to_{n\to \infty} 0,
\end{align*} 
by Stirling's formula. Similarly, we have by the Itô isometry  $$ \E^{1/2}\left[\left(\int_0^T \frac{(T-s)^{i\alpha-1}}{\Gamma(i\alpha)}  \kappa_2^i  dW_s\right)^2\right] = \frac{|\kappa_2|^iT^{i\alpha-1/2}}{(2i\alpha-1)^{1/2}\Gamma(i\alpha)},$$ and we obtain  $\sum_{i=n+1}^\infty \E^{1/2} \left[\left( \kappa_2^{i-1} \int_0^t \frac{(T-s)^{i\alpha -1}}{\Gamma(i \alpha)} dW_s\right)^2\right] \to_{n\to \infty} 0$. By uniqueness of the limit in $L^2(\Omega)$, we get $$\mathcal{Y}_T=\sigma \sum_{i=1}^\infty \kappa_2^{i-1} \int_0^T \frac{(T-s)^{i\alpha -1}}{\Gamma(i \alpha)} dW_s \text{ a.s., for }T>0.$$

We now calculate the covariance by Itô isometry for $t\in [0,T]$:
\begin{align*}
\cov(\cY_t,\cY_T)&= \sigma^2 \int_0^t  \sum_{i_1, i_2=1}^\infty \kappa_2^{i_2+i_2-2} \frac{(t-s)^{i_1\alpha -1}}{\Gamma(i_1 \alpha)}  \frac{(T-s)^{i_2\alpha -1}}{\Gamma(i_2 \alpha)} ds \\
&= \sigma^2 \sum_{i_1, i_2=1}^\infty \frac{\kappa_2^{i_2+i_2-2}}{\Gamma(i_1 \alpha) \Gamma(i_2 \alpha)} t^{i_1\alpha} T^{i_2\alpha-1}\int_0^1 (1-u)^{i_1\alpha-1}\left(1-u \frac{t}{T}\right)^{i_2\alpha -1} du,
\end{align*}
with the change of variable $s=ut$. This gives~\eqref{cov_formula} by~\eqref{def_2F1}. Last, we apply the chain rule to~\eqref{Y_cal} (see~\cite[Proposition 1.3.8]{Nualart}) and we get~\eqref{der_Malliavin_X}. We then observe that $\int_0^t (\D_sX_t)^2 ds=\sigma^2 \sum_{i_1, i_2=1}^\infty \frac{\kappa_2^{i_2+i_2-2}}{\Gamma(i_1 \alpha) \Gamma(i_2 \alpha)} \frac{t^{(i_1+i_2)\alpha-1}}{(i_1+i_2)\alpha-1}$ by the same calculation as above with $T=t$, and therefore $\|X_t\|_{1,2}<\infty$. 
\end{proof}

\begin{corollary}\label{cor_stationarity}
  If $\kappa_2<0$, then $X_t$ converges in distribution when $t\to \infty$ to $\mathcal{N}(-\frac{\kappa_1}{\kappa_2},\Sigma_\infty^2)$, with
  $$\Sigma_\infty^2=\sigma^2 \int_0^\infty s^{2\alpha-2}E_{\alpha,\alpha}^2(\kappa_2 s^\alpha)ds,$$ 
  where $E_{\alpha,\beta}(z)=\sum_{i=0}^\infty \frac{z^i}{\Gamma(\alpha i +\beta)}$ for $z\in \R$ is the generalized Mittag-Leffler function.
\end{corollary}
\begin{proof}
  Since $X_t$ follows a normal distribution, it is sufficient to get the convergence of the mean and the variance. By~\cite[Theorem 1.6]{Podlubny}, $E_\alpha(z)=E_{\alpha,1}(z) \to_{z\to -\infty} 0$, and we get $\E[X_t]\to_{t\to +\infty} -\frac{\kappa_1}{\kappa_2}$ from~\eqref{esp_VOU}. From~\eqref{Y_cal}, we get $$\Var(X_t)=\sigma^2 \int_0^t \sum_{i_1,i_2=1}^\infty \kappa_2^{i_1+i_2-2}  \frac{s^{(i_1+i_2)\alpha -2}}{\Gamma(i_1 \alpha)\Gamma(i_2 \alpha)} ds=\sigma^2 \int_0^t s^{2\alpha-2}E_{\alpha,\alpha}^2(\kappa_2 s^\alpha)ds. $$
  Note that $E_{\alpha,\alpha}(0)=\frac{1}{\Gamma(\alpha)}$, and the integral is well defined in $0+$ since $2\alpha-1>0$. Using again~\cite[Theorem 1.6]{Podlubny}, we get $|E_{\alpha,\alpha}(z)|\le \frac{C}{1-z}$ for $z<0$, so that $s^{2\alpha-2}E_{\alpha,\alpha}^2(\kappa_2 s^\alpha)=_{s\to +\infty} O(s^{-2})$ is integrable. This gives the desired result.
\end{proof}

We end this section by a crucial estimate on the process~$X$. The proof is quite technical and is postponed in the appendix section. 
\begin{theorem}\label{thm_technique}
    Let $T>0$, $n\in \N^*$ and $t_k=k\frac Tn$ for $k\in \N$.  Let $\eta(s)=t_k$ for $t_k\le s<t_{k+1}$. There exists a constant $C_T\in \R_+$ such that for any $t \in [0,T]$,
    \begin{equation*}\max_{1\le k\le n}\left|  \int_0^{t_k}\frac{(t_k-s)^{\alpha-1}}{\Gamma(\alpha)} \E[ X_t (X_s-X_{\eta(s)})]ds   \right| \le C_T \left( \left(1+ t^{\alpha -1}  \right) {\bf v}_n (\alpha)+ \frac{t^{2\alpha -2}}{n}\right).\end{equation*}
\end{theorem}

\subsection{Weak rate of convergence}

The goal of this subsection is to analyse the weak rate of convergence of the kernel integrated Euler scheme~\eqref{eq:scheme_price} for $\xcn$ only.  We start by a remark that motivates the choice of this scheme rather than the discretised kernel Euler scheme:
\begin{align*}
X^n_t&=x_0+\int_0^t\frac{(t-\eta(s))^{\alpha-1}}{\Gamma(\alpha)}  (\kappa_1+\kappa_2 X^n_{\eta(s)})ds+ \sigma \int_0^t \frac{(t-\eta(s))^{\alpha-1}}{\Gamma(\alpha)}  dW_s, \; t\in[0,T].
\end{align*}
We recall that $\eta(s)=\frac{kT}{n}$ for $\frac{kT}{n}\le s<\frac{(k+1)T}{n}$. Let us consider the simplest case: $\kappa_1=\kappa_2=0$ and $\sigma=1$, and consider $X^n_t=\int_0^t \frac{(t-\eta(s))^{\alpha-1}}{\Gamma(\alpha)}dW_s$, the approximation of $X_t=\int_0^t \frac{(t-s)^{\alpha-1}}{\Gamma(\alpha)}dW_s=\xcn_t$. Both are Gaussian processes, and we have
$$ X_T\sim \mathcal{N} \left(0, \frac{T^{2\alpha-1}}{(2\alpha-1)\Gamma(\alpha)^2} \right), \  X^n_T\sim \mathcal{N}\left(0, \int_0^T \frac{(T-\eta(s))^{2\alpha-2}}{\Gamma(\alpha)^2} ds\right).$$
Thus, we have the following equation between characteristic functions $\Phi_{X^n_T}(u)=\E[e^{iu X^n_T}]$ and $\Phi_{X_T}(u)=\E[e^{iu X_T}]$:
$$u \in \R,\  \Phi_{X^n_T}(u)=\Phi_{X_T}(u)\exp \left(\frac{u^2}2 \left(\frac{T^{2\alpha-1}}{(2\alpha-1)\Gamma(\alpha)^2}  - \int_0^t \frac{(T-\eta(s))^{2\alpha-2}}{\Gamma(\alpha)^2} ds \right) \right).$$
By McGown and Parks~\cite{McGown}, we have
\begin{align*}
  \frac{T^{2\alpha-1}}{(2\alpha-1)\Gamma(\alpha)^2}  - \int_0^T \frac{(t-\eta(s))^{2\alpha-2}}{\Gamma(\alpha)^2} ds&=\frac{1}{\Gamma(\alpha)^2}\left(\frac{T^2}{2\alpha-1}-\frac T n \sum_{i=1}^n \sum_{i=1}^n \left(\frac{iT}n \right)^{2\alpha-2}\right) \\
  & \sim_{n\to \infty} \frac{1}{\Gamma(\alpha)^2} \zeta(2(1-\alpha))\frac{T^{2\alpha-1}}{n^{2\alpha-1}},
\end{align*}
which gives a weak rate of convergence proportional to $n^{-(2\alpha-1)}$ for the family of test functions $x\mapsto e^{iux}$. We cannot therefore hope to reach the rate of $O({\bf v}_n(\alpha))$ defined by~\eqref{vitesse}.

In what follows, we introduce the functional spaces that will be used through the paper (in particular as test functions), for $\ell \in \N$: 
 \begin{align}
  &\mathcal{C}_{\exp}^\ell(\R, \R):=\{g \in \mathcal{C}^\ell (\R, \R): \forall k \in \{0,\dots,\ell\}, \exists C\in \R_+ \forall x \in \R,\ |g^{(k)}(x)|\le Ce^{C|x|} \}, \label{def_Cell} \\
  &\mathcal{C}_{\exp}^\infty(\R, \R):=\{g \in \mathcal{C}^\infty(\R, \R): \forall k \in \N, \exists C\in \R_+ \forall x \in \R,\ |g^{(k)}(x)|\le Ce^{C|x|} \}. \label{def_Cexpinfty} 
\end{align} 
It is easy to check that $\mathcal{C}_{\exp}^\infty(\R, \R)$ is a vector space such that for $g\in \mathcal{C}_{\exp}^\infty(\R, \R)$, we have  $g' \in \mathcal{C}_{\exp}^\infty(\R, \R)$. Similarly, for $g_1,g_2\in \mathcal{C}_{\exp}^\infty(\R, \R)$, we have  $g_1 g_2\in \mathcal{C}_{\exp}^\infty(\R, \R)$ by Leibniz rule.

The main result of this section is the following theorem.
\begin{theorem}\label{thm_covariance_error}
  There exists a constant $C_T$ that depends on $T$ and $\alpha \in (1/2,1]$ such that for $0< k_1, k_2\le n$,
  \begin{equation}\label{weak_error_esp}
  |\E[X_{t_{k_1}}]-\E[\xcn_{t_{k_1}}]|\le \frac{C_T}{n},
  \end{equation} 
  $$ |\E[X_{t_{k_1}}X_{t_{k_2}}]-\E[\check X^n_{t_{k_1}}\check X^n_{t_{k_2}}]| \le C_T  \left( (1+t_{k_1}^{\alpha-1}+t_{k_2}^{\alpha-1}) {\bf v}_n(\alpha) + \frac{t_{k_1}^{2\alpha-2}+t_{k_2}^{2\alpha-2}}{n}\right) $$
  and 
  \begin{equation}\label{cov_weak} |Cov(X_{t_{k_1}},X_{t_{k_2}})-Cov(\check X^n_{t_{k_1}},\check X^n_{t_{k_2}})| \le C_T  \left( (1+t_{k_1}^{\alpha-1}+t_{k_2}^{\alpha-1}) {\bf v}_n(\alpha) + \frac{t_{k_1}^{2\alpha-2}+t_{k_2}^{2\alpha-2}}{n}\right). \end{equation}
  Besides,  for any $\Psi \in \mathcal{C}_{\exp }^2(\R, \R)$, there exists a constant $C\in \R_ +$ such that 
  $$\forall n\ge 1, 1\le k\le n, \  \ |\E[\Psi(\xcn_{t_k})]-\E[\Psi(X_{t_k})]|\le C_T\left( (1+t_k^{\alpha-1}){\bf v}_n(\alpha)+\frac{t_k^{2\alpha-2}}{n}\right).$$
  In particular, for $k=n$ we get the weak error $ |\E[\Psi(\xcn_{T})]-\E[\Psi(X_{T})]| =O({\bf v}_n(\alpha))$.
\end{theorem}
\begin{remark}%\label{Rk_k1=0}
Note that if $k_1=0$, we have by~\eqref{weak_error_esp} $$|\E[X_{t_{k_1}}X_{t_{k_2}}]-\E[\check X^n_{t_{k_1}}\check X^n_{t_{k_2}}]|= |X_0| |\E[X_{t_{k_2}}-\check X^n_{t_{k_2}} ]| \le \frac{C_T}{n},$$
   and also obviously $Cov(X_{t_{k_1}},X_{t_{k_2}})=Cov(\check X^n_{t_{k_1}},\check X^n_{t_{k_2}})=0$.
\end{remark}
\begin{remark}\label{rk_unifbound}
  From Proposition~\ref{prop_esp} and~\eqref{weak_error_esp}, we get that $\max_{0\le k\le n } |\E[\xcn_{t_k}]|\le C_T$ for a constant that does not depend on~$n$.
  Besides, using~\eqref{cov_formula} and~\eqref{cov_weak}, we also obtain $\max_{0\le k_1,k_2 \le n } |Cov(\check X^n_{t_{k_1}},\check X^n_{t_{k_2}})|\le C_T$ for a constant that does not depend on~$n$ since $t_1^{\alpha-1}{\bf v}_n(\alpha)$ and $\frac{t_1^{2\alpha-2}}{n}$ go to~$0$ as $n\to \infty$ for any $\alpha \in (1/2,1]$.  
\end{remark}

Before proving the theorem, we introduce some notation and an important preliminary result of Gronwall type. 
For this aim, let us note  $\Delta_{t_k}=X_{t_k}-\check X^n_{t_k}$ for $k\in \{ 0,\dots,n\}$, that satisfies  
\begin{align}
  \Delta_{t_k}&=\int_0^{t_k} \frac{(t_k-s)^{\alpha-1}}{\Gamma(\alpha)}\kappa_2 (X_s-X_{\eta(s)}+X_{\eta(s)}-\xcn_{\eta(s)})ds \notag\\&=\int_0^{t_k} \frac{(t_k-s)^{\alpha-1}}{\Gamma(\alpha)}\kappa_2 (X_s-X_{\eta(s)})ds +\sum_{i=2}^k  \Delta_{t_{i-1}} \int_{t_{i-1}}^{t_i} \frac{(t_k-s)^{\alpha-1}}{\Gamma(\alpha)}ds.\label{delta_tk}
\end{align} 

\begin{lemma}\label{lem_LiZeng} Let $\mathcal{V}$ be a square integrable random variable.  Then, there exists $C\in \R_+$ that does not depend on $\mathcal{V}$ and $n$, such that
  $$\max_{1\le k\le n} \left|  \E[\mathcal{V} \Delta_{t_k}] \right| \le C  \max_{1\le k\le n}\left|  \int_0^{t_k}\frac{(t_k-s)^{\alpha-1}}{\Gamma(\alpha)} \E[ \mathcal{V} (X_s-X_{\eta(s)})]ds   \right|$$
\end{lemma}
\begin{proof}
We have from~\eqref{delta_tk}, for $k\in \{1,\dots,n\}$:
$$\mathcal{V} \Delta_{t_k} = \mathcal{V}  \kappa_2\int_0^{t_k}\frac{(t_k-s)^{\alpha-1}}{\Gamma(\alpha)} (X_s-X_{\eta(s)})ds + \kappa_2 \sum_{i=2}^k \mathcal{V} \Delta_{t_{i-1}}\int_{t_{i-1}}^{t_i} \frac{(t_k-s)^{\alpha-1}}{\Gamma(\alpha)}ds.$$
We then take the expectation and get by the triangular inequality
\begin{align*}
   |\E[\mathcal{V} \Delta_{t_k} ]|\le & |\kappa_2| \left| \int_0^{t_k}\frac{(t_k-s)^{\alpha-1}}{\Gamma(\alpha)} \E[\mathcal{V}(X_s-X_{\eta(s)})] ds \right| \\
   & +|\kappa_2|  \sum_{i=2}^k |\E[\mathcal{V} \Delta_{t_{i-1}}]| \frac{(k-i+1)^\alpha-(k-i)^\alpha}{\Gamma(\alpha+1)} \left(\frac{T}{n} \right)^\alpha.
\end{align*}
By~\eqref{eq_LiZeng}, we obtain 
$$ |\E[\mathcal{V} \Delta_{t_k} ]|\le \mathbf{M} +|\kappa_2| C_\alpha \left(\frac{T}{n} \right)^\alpha \sum_{j=1}^{k-1}|\E[\mathcal{V} \Delta_{t_{j}}]| (k-j)^{\alpha-1},$$
with $\mathbf{M}=|\kappa_2| \max_{1\le k\le n} \left| \int_0^{t_k}\frac{(t_k-s)^{\alpha-1}}{\Gamma(\alpha)} \E[\mathcal{V}(X_s-X_{\eta(s)})] ds \right|$. We conclude by the Gronwall type estimate~\cite[Lemma 3.4]{LiZe}.
\end{proof}

\begin{proof}[Proof of Theorem~\ref{thm_covariance_error}] 
  We start by proving the first bound. We apply Lemma~\ref{lem_LiZeng} with $\mathcal{V}=1$ and get 
  $$\max_{1\le k\le n} |\E[X_{t_{k}}]-\E[\xcn_{t_{k}}]|\le C \max_{1\le k\le n} \left|\int_0^{t_k} \frac{(t_k-s)^{\alpha-1}}{\Gamma(\alpha)}(\E[X_s]-\E[X_{\eta(s)}])ds\right|.$$
By Proposition~\ref{prop_esp}, $\E[X_t]=x_0 t^\alpha +z(t)$ where $z$ is a $\mathcal{C}^1$ function on $\R_+$.   We apply respectively \cite[Theorem 2.4 (b) and (a)]{DFF} to get 
\begin{equation}\label{Diethelm_esperance}
  \left|\int_0^{t_k} \frac{(t_k-s)^{\alpha-1}}{\Gamma(\alpha)}(\E[X_s]-\E[X_{\eta(s)}])ds\right| \le \frac{T}{n}\left( C T^{2\alpha-1}+\frac{ \max_{t\in [0,T]}|z'(t)|}{\alpha}T^\alpha \right),
\end{equation} 
 which gives the claim on the expectation. 
  
  We now focus on covariance terms and have 
\begin{align*}
  \E[X_{t_{k_1}}X_{t_{k_2}}]-\E[\check X^n_{t_{k_1}}\check X^n_{t_{k_2}}]&=\E[X_{t_{k_2}}(X_{t_{k_1}}-\check X^n_{t_{k_1}})]+\E[\check X^n_{t_{k_1}}(X_{t_{k_2}}-\check X^n_{t_{k_2}})]\\
 &=\E[X_{t_{k_2}}\Delta_{t_{k_1}}]+\E[\check X^n_{t_{k_1}}\Delta_{t_{k_2}}],
\end{align*} 
with $\Delta_{t_k}$ given by~\eqref{delta_tk}.
We apply Lemma~\ref{lem_LiZeng} twice: first with $\mathcal{V}=X_{t_{k_2}}$ and then with $\mathcal{V}=\check X^n_{t_{k_1}}$. 

We start with $\mathcal{V}=X_{t_{k_2}}$ and get
\begin{equation*}\max_{1\le k\le n}|\E[X_{t_{k_2}}\Delta_{t_k}]|\le C \max_{1\le k\le n}\left|  \int_0^{t_k}\frac{(t_k-u)^{\alpha-1}}{\Gamma(\alpha)} \E[ X_{t_{k_2}} (X_u-X_{\eta(u)})]du   \right|.\end{equation*}
By Theorem~\ref{thm_technique}, we obtain
$$\max_{1\le k\le n}|\E[X_{t_{k_2}}\Delta_{t_k}]|\le C_T \left((1+t_{k_2}^{\alpha-1}){\bf v}_n(\alpha) + \frac{t_{k_2}^{2\alpha-2}}{n} \right). $$
We now apply Lemma~\ref{lem_LiZeng} with $\mathcal{V}=\check X^n_{t_{k_1}}$:
$$\max_{1\le k \le n} |\E[\check X^n_{t_{k_1}}(X_{t_{k}}-\check X^n_{t_k})]|\le C \max_{1\le k \le n} \left| \int_0^{t_k} \frac{(t_k-s)^{\alpha-1}}{\Gamma(\alpha)} \E[\check X^n_{t_{k_1}} (X_s-X_{\eta(s)})] ds\right|.$$
Then, we write 
\begin{align*}
  \E[\check X^n_{t_{k_1}}(X_s-X_{\eta(s)})]=\E[(\check X^n_{t_{k_1}}-X_{t_{k_1}})X_s]+ \E[X_{t_{k_1}}(X_s-X_{\eta(s)})]+ \E[X_{\eta(s)}(X_{t_{k_1}}- \check X^n_{t_{k_1}})],
\end{align*}
and use the triangle inequality. We first notice that the second term  
$$\max_{1\le k\le n}\left| \int_0^{t_k} \frac{(t_k-s)^{\alpha-1}}{\Gamma(\alpha)} \E[X_{t_{k_1}}(X_s-X_{\eta(s)})] ds \right|$$ is again 
upper bounded  by $C_T \left((1+t_{k_1}^{\alpha-1}){\bf v}_n(\alpha) + \frac{t_{k_1}^{2\alpha-2}}{n} \right)$ using Theorem~\ref{thm_technique}. For the two other terms, we have to upper bound
$$\max_{1\le k\le n}\left| \int_0^{t_k} \frac{(t_k-s)^{\alpha-1}}{\Gamma(\alpha)} \E[(\check X^n_{t_{k_1}} -X_{t_{k_1}})X_s] ds \right|  \text{ and }\max_{1\le k\le n}\left| \int_0^{t_k} \frac{(t_k-s)^{\alpha-1}}{\Gamma(\alpha)} \E[(\check X^n_{t_{k_1}} -X_{t_{k_1}})X_{\eta(s)}] ds \right|$$
To do so, we use Lemma~\ref{lem_LiZeng} with $\mathcal{V} = X_s$ and then 
Theorem~\ref{thm_technique}  to get
\begin{align*}
   |\E[(\check X^n_{t_{k_1}} -X_{t_{k_1}})X_s]| &\le C \max_{1\le k \le n} \left|\int_0^{t_k}\frac{(t_k-u)^{\alpha-1}}{\Gamma(\alpha)}\E[X_s(X_u-X_{\eta(u)})]du \right| \\
   & \le  C_T\left((1+s^{\alpha-1}) {\bf v}_n(\alpha) + \frac{s^{2\alpha-2}}{n}\right).\end{align*}
Therefore, we get 
\begin{align*}
  &\left| \int_0^{t_k} \frac{(t_k-s)^{\alpha-1}}{\Gamma(\alpha)} \E[(\check X^n_{t_{k_1}} -X_{t_{k_1}})X_s] ds \right|\\
  \le &  C_T\left( \int_0^{t_k} \frac{(t_k-s)^{\alpha-1}}{\Gamma(\alpha)} (1+s^{\alpha-1}) {\bf v}_n(\alpha) ds + \int_0^{t_k} \frac{(t_k-s)^{\alpha-1}}{\Gamma(\alpha)} \frac{s^{2\alpha-2}}{n} ds  \right) \\
  = &C_T \left( {\bf v}_n(\alpha) \left(\frac{t_k^\alpha}{\Gamma(\alpha+1)} + t_k^{2\alpha-1} \frac{\Gamma(\alpha)}{\Gamma(2 \alpha)} \right)+ t_k^{3\alpha-2} \frac{\Gamma(2\alpha-1)}{\Gamma(3\alpha-1)n}\right)
\end{align*}
If $3\alpha-2\ge 0$, this is upper bounded by $C_T{\bf v}_n(\alpha)$, using that $t_k\le T$. Otherwise, we have $\frac{t_k^{3\alpha-2}}{n} \le  \frac{t_1^{3\alpha-2}}{n} = \frac{T^{3\alpha-2}}{n^{3\alpha-1}}\le C_T  {\bf v}_n(\alpha)$. Therefore, 
$$\max_{1\le k \le n}\left| \int_0^{t_k} \frac{(t_k-s)^{\alpha-1}}{\Gamma(\alpha)} \E[(\check X^n_{t_{k_1}} -X_{t_{k_1}})X_s] ds \right| \le C_T  {\bf v}_n(\alpha) $$ 
in all cases. We now prove that the same estimate holds for
$$\max_{1\le k \le n}\left| \int_0^{t_k} \frac{(t_k-s)^{\alpha-1}}{\Gamma(\alpha)} \E[(\check X^n_{t_{k_1}} -X_{t_{k_1}})X_{\eta(s)}] ds \right|.$$
We use the triangular inequality
\begin{align*}
  &\left| \int_0^{t_k} \frac{(t_k-s)^{\alpha-1}}{\Gamma(\alpha)} \E[(\check X^n_{t_{k_1}} -X_{t_{k_1}})X_{\eta(s)}] ds \right|\\
  \le & |X_0| \int_{0}^{t_1}\frac{(t_k-s)^{\alpha-1}}{\Gamma(\alpha)} |\E[\check X^n_{t_{k_1}} -X_{t_{k_1}}]| ds \\
  &+C_T\int_{t_1}^{t_k} \frac{(t_k-s)^{\alpha-1}}{\Gamma(\alpha)} (1+\eta(s)^{\alpha-1}) {\bf v}_n(\alpha) ds + C_T \int_{t_1}^{t_k} \frac{(t_k-s)^{\alpha-1}}{\Gamma(\alpha)}  \frac{\eta(s)^{2\alpha-2}}{n} ds.
\end{align*}
The first term is bounded by $C_T\frac{T}{n}$  by using~\eqref{weak_error_esp}, and the other terms can be upper bounded as above, using that $\frac{s}2\le \eta(s)$ for $s\ge t_1$. We finally get 
$$\max_{1\le k \le n}\left| \int_0^{t_k} \frac{(t_k-s)^{\alpha-1}}{\Gamma(\alpha)} \E[(\check X^n_{t_{k_1}} -X_{t_{k_1}})X_{\eta(s)}] ds \right| \le C_T {\bf v}_n(\alpha).$$
This gives the first estimate.

Finally, we get the estimate on the covariance~\eqref{cov_weak} as follows
\begin{align*}
  &|Cov(X_{t_{k_1}},X_{t_{k_2}})-Cov(\check X^n_{t_{k_1}},\check X^n_{t_{k_2}})|\\
  \le &|\E[X_{t_{k_1}}X_{t_{k_2}}]-\E[\check X^n_{t_{k_1}}\check X^n_{t_{k_2}}]|  +|\E[X_{t_{k_1}}]| |\E[X_{t_{k_2}}-\xcn_{t_{k_2}}]|+|\E[\xcn_{t_{k_2}}]| |\E[X_{t_{k_1}}-\xcn_{t_{k_1}}]|,
\end{align*} 
noting that $|\E[X_{t_{k_1}}]|$ and $|\E[\xcn_{t_{k_2}}]|$ are bounded (see Remark~\ref{rk_unifbound}) and using that $|\E[X_{t_{k_1}}-\xcn_{t_{k_1}}]|\le \frac{C_T}{n} \le C_T{\bf v}_n(\alpha).$

As $\Psi\in \mathcal{C}_{\exp}^2(\R, \R)$, we apply Lemma~\ref{lem_weak_ito} and get with $m(u)=(1-u)\E[X_{t_k}]+u\E[\xcn_{t_k}]$, $\Sigma(u)=(1-u)\covX(t_k,t_k)+uCov(\xcn_{t_k},\xcn_{t_k})$, $u\in[0,1]$, and $Z_u\sim\mathcal{N}(m(u) ,\Sigma(u))$, 
\begin{align*}
\E[\Psi(\xcn_{t_k})]-\E[\Psi(X_{t_k})]=&\int_0^1  \left( \E[\xcn_{t_k}]- \E[X_{t_k}] \right) \E[\Psi'(Z_u)]du \\&+\int_0^1 \frac 12 \left(Cov(\xcn_{t_k},\xcn_{t_k})-\covX(t_k,t_k) \right) \E[\Psi''(Z_u)]du.
\end{align*} 
By straightforward calculations, $| \E[\Psi'(Z_u)]|+| \E[\Psi''(Z_u)]|\le C e^{C |m(u)|+ \frac{C^2 \Sigma(u)^2}{2}} \le C e^{C\overline{m}+  \frac{C^2 \overline{\Sigma}^2}{2}}$ for some $C \in \R_+$, where $$\overline{m}=\sup_{n\ge 1} \max_{1\le k \le n}\max(|\E(X_{t_k})|,|\E(\xcn_{t_k}|) \text{ and }\overline{\Sigma}=\sup_{n\ge 1} \max_{1\le k \le n}\max(\covX(t_k,t_k), Cov(\xcn_{t_k},\xcn_{t_k}) )$$ are finite by Equations~\eqref{esp_VOU} and \eqref{cov_formula} and Remark~\ref{rk_unifbound}. Using again Equation~\eqref{cov_weak}, we get
$$|\E[\Psi(\xcn_{t_k})]-\E[\Psi(X_{t_k})]| \le  C e^{C\overline{m}+  \frac{C^2 \overline{\Sigma}^2}{2}} C_T  \left( (1+t_k^{\alpha-1}) {\bf v}_n(\alpha) + \frac{t_k^{2\alpha-2}}{n}\right).\qedhere $$
\end{proof}
We see from the proof that the covariance estimation~\eqref{cov_weak} is only needed at time $t_n=T$ to deduce the weak error rate. However, the full estimation~\eqref{cov_weak} at times $t_{k_1}$ and $t_{k_2}$ will be crucial to get the weak error rate of the rough mean-reverting stochastic volatility model. We will also need the following estimate between the Malliavin derivatives of the scheme and the SVE.  
\begin{lemma}\label{lem_err_Malliavin3}  
    Let $1\le k\le n$. We have $\xcn_{t_k} \in \mathbb{D}^{1,2}$. The Malliavin derivative of $\xcn_{t_k}$ is deterministic and satisfies for $s\in(0,t_k)$,
    \begin{align}\D_s\xcn_{t_k}&= \kappa_2 \int_0^{t_k} \D_s \xcn_{\eta(u)} \frac{(t_k-u)^{\alpha-1}}{\Gamma(\alpha)}du + \sigma\frac{(t_k-s)^{\alpha-1}}{\Gamma(\alpha)}. \label{eq_dxcn}
    \end{align}
    Besides, there is a constant $C \in \R_+$ depending on the parameters $\alpha,\kappa_2,\sigma$, $\gamma$ and $T>0$ such that for all $n\ge 1$ and $t_j\le t_k\le T$, $$\left|\int_0^{t_j} \left(\D_s\xcn_{t_k} -\D_sX_{t_k}\right) ds \right| \le C \frac T n \text{ and }\int_0^{t_k}|\D_s\xcn_{t_k} | ds\le C.$$
  \end{lemma}
  Let us note that $\D_s \xcn_{t_k}$ not only blows up when $s\to t_k$, but also from the induction formula~\eqref{eq_dxcn} when $s\to t_j$ for $j\le k$, since $\D_s \xcn_{t_k}=\sum_{j=1}^{k-1} \int_{t_j}^{t_{j+1}}  \frac{(t_k-u)^{\alpha-1}}{\Gamma(\alpha)}du \D_s \xcn_{t_j}+ \sigma\frac{(t_k-s)^{\alpha-1}}{\Gamma(\alpha)}$. In contrast, $\D_sX_{t_k}$ is continuous for $s<t_k$ and only blows up when $s\to t_k$. Therefore, there is no hope to get a uniform bound on $|\D_s \xcn_{t_k}- \D_sX_{t_k}|$ for $s\in [0,t_k)$. However, Lemma~\ref{lem_err_Malliavin3} shows that the error between the integrals has a nice rate of convergence. Besides,  We will also need the following estimates on some integrals  of  $|\D_s \xcn_{t_k}- \D_sX_{t_k}|$.
\begin{lemma}\label{lem_err_Malliavin_int}
  There exists a constant $C \in \R_+$ that does not depend on~$n\in \N^*$ such that
\begin{equation*}
    \max_{0\le k, \ell \le n}\int_{0}^{t_k}  |\D_s X_{t_\ell}-\D_s\xcn_{t_\ell}| ds \le C n^{-\alpha}.
\end{equation*}
\end{lemma}
  This lemma  will be crucial and enough to prove the weak rate of convergence for the mean-reverting rough stochastic volatility model. 
   Their proofs are postponed in the appendix section.

\section{Weak error approximation with the cubic test function}\label{Sec_cubic}

In this section, we are interested in studying the weak error convergence rate of the scheme~\eqref{eq:scheme_price}. Namely, for $\Phi:\R \to \R$, we want to upper bound $\E[\Phi(\check{L}^n_T)]-\E[\Phi(L_T)]$. When $\Phi$ has a polynomial growth, these expectations are well defined. More precisely, we have the following result.  
\begin{lemma}
    Let $\Phi,f:\R\to \R$ be measurable. We assume that:
    \begin{itemize}
        \item there exists $C,p>0$ such that $|\Phi(x)|\le C(1+|x|^p)$ for all~$x\in \R$,
        \item there exists $C>0$ such that $|f(x)|\le Ce^{C|x|}$.
    \end{itemize}  
    Then, $\E[\Phi(\check{L}^n_T)]$ and $\E[\Phi(L_T)]$ are well  defined and finite.
\end{lemma}
\begin{proof}
Without loss of generality, we assume $p\ge 2$. By Burkholder-Davis-Gundy's inequality, we have 
\begin{align*}
    \E\left[\left|\int_0^T f(X_t) dB_t \right|^p\right] &\le C_p \E\left[\left|\int_0^T f^2(X_t)dt \right|^{p/2}\right] \\
    &\le C_p T^{p/2-1} \int_0^T \E\left[|f(X_t)|^p\right]dt \\
    &\le C_p T^{p/2-1} \int_0^T C^p\E\left[e^{pC|X_t|}\right]dt,
\end{align*}
which is finite since $X_t$ is a Gaussian process with  continuous mean and variance, see Propositions~\ref{prop_esp} and~\ref{prop_OUexplicit}.  The same conclusion holds for~$\xcn$ by using Remark~\ref{rk_unifbound}.
\end{proof}

 For the sake of readability, we start by the case $\Phi(x)=x^3$ as in Gassiat~\cite{Gassiat}, to better emphasize the difficulty coming from the discretization of the rough Volterra SDE and how to manage it. The case $\Phi(x)=x^2$ is too simple due to the Itô isometry as pointed out by Neuenkirch in~\cite{BHT}. We also assume in this section $b\equiv 0$ and $L_0=0$ in~\eqref{SVE_OU} and~\eqref{eq:scheme_price} for the sake of simplicity. We want to analyse the weak error:
$$\mathcal{E}=\E\left[\left(\int_0^T f(X_t) dB_t\right)^3\right]-\E\left[\left(\int_0^T f(\xcn_{\eta(t)}) dB_t\right)^3\right].$$
It is known (see~\cite[Proposition 6.1]{FSW}) to be the prototypical case to get familiar with the structure of the weak error rate, avoiding too much heavy notation. Friz et al.~\cite{FSW} have indeed noticed this and were able to extend to any polynomial test function by an induction procedure with iterated integrals. In Section~\ref{Sec_polynomial}, we show that we can also in our case extend the result to any polynomial functions. However, the iterated integrals coming from arbitrary polynomial functions raise new difficulties due to the discretization step, which does not appear in the work by Friz et al.~\cite{FSW}. 

By Itô calculus, we have 
\begin{align*}
  \left(\int_0^Tf(X_{t})dB_t \right)^3=&3\int_0^T \left(\int_0^tf(X_{s})dB_s \right)^2 f(X_{t})dB_t \\
  &+ 3 \int_0^T \left(\int_0^tf(X_{s})dB_s \right) f(X_{t})^2 dt,
\end{align*}
and
\begin{align*}
  \left(\int_0^Tf(\xcn_{\eta(t)})dB_t \right)^3=&3\int_0^T \left( \int_0^tf(\xcn_{\eta(s)})dB_s \right)^2 f(\xcn_{\eta(t)})dW_t\\&+ 3 \int_0^T \left(\int_0^tf(\xcn_{\eta(s)})dB_s \right) f(\xcn_{\eta(t)})^2 dt.  
\end{align*}
We assume that~$f\in \mathcal{C}_{\exp}^3(\R, \R)$, see~\eqref{def_Cell}.
This gives in particular that the Itô integrals are true martingales and have a zero expectation. We then use the Clark-Ocone formula (see e.g.~\cite[Proposition 1.3.14]{Nualart} that can be applied thanks to Proposition~\ref{prop_OUexplicit} and  Lemma~\ref{lem_err_Malliavin3}) to get
\begin{align*}\E\left[\left(\int_0^Tf(X_{t})dB_t \right)^3 \right]&=3\int_0^T \E\left[\left(\int_0^tf(X_{s})dB_s \right) f(X_{t})^2\right]dt\\& =  6\rho \int_0^T \int_0^{t} \E[f(X_{s}) ff'(X_{t})\D_s X_{t}] ds dt.\end{align*}
and
\begin{align*}\E\left[\left(\int_0^Tf(\xcn_{\eta(t)})dB_t \right)^3 \right]&=3\int_0^T \E\left[\left(\int_0^tf(\xcn_{\eta(s)})dB_s \right) f(\xcn_{\eta(t)})^2\right]dt\\& = 6\rho \int_0^T \int_0^{\eta(t)} \E[f(\xcn_{\eta(s)}) ff'(\xcn_{\eta(t)})\D_s \xcn_{\eta(t)}] ds dt.\end{align*}
Let us recall that from Proposition~\ref{prop_OUexplicit} and  Lemma~\ref{lem_err_Malliavin3}, both $\D_sX_t$ and $\D_s \xcn_{\eta(t)}$ are deterministic. 
We write then this error as $\mathcal{E}=6\rho(\mathcal{E}_1+\mathcal{E}_2)$, with
\begin{align}
  \mathcal{E}_1&=\int_0^T \int_0^{t}\E[f(X_{s}) ff'(X_{t})]\D_sX_t-\E[f(X_{\eta(s)}) ff'(X_{\eta(t)})] \D_s X_{\eta(t)}  ds dt , \label{def_E1}\\
  \mathcal{E}_2&=\int_0^T \int_0^{\eta(t)} \E[f(X_{\eta(s)}) ff'(X_{\eta(t)})] \D_s X_{\eta(t)}-\E[f(\xcn_{\eta(s)}) ff'(\xcn_{\eta(t)})] \D_s \xcn_{\eta(t)}   ds dt,\label{def_E2}
\end{align}
noting that $\D_s X_{\eta(t)}=\D_s \xcn_{\eta(t)}=0$ for $s \in (\eta(t),t)$.
The analysis of~$\mathcal{E}_1$ adapts the arguments of Gassiat~\cite{Gassiat} and Friz et al.~\cite{FSW} for the rough Stochastic Volterra Equation~$X$, since the model considered in these papers corresponds to the case $\kappa_1=\kappa_2=0$. In contrast, the analysis of~$\mathcal{E}_2$ is completely new  and exploits the results obtained in Section~\ref{Sec_rOU} to upper bound the approximation error on the volatility process~$X$. The striking result is that we still get a convergence rate in $O({\bf v}_n(\alpha))$, the same as in~\cite{Gassiat} and~\cite{FSW}. Thus, the additional discretization of~$X$ does not affect the convergence rate. 

\begin{theorem}\label{thm_cubic}
   Let~$L$ be the process defined by~\eqref{SVE_OU} with $L_0=0$,  $b\equiv 0$ and $f\in  \mathcal{C}_{\exp}^3(\R, \R)$. Let $\check{L}^n$ denote the approximation scheme~\eqref{eq:scheme_price} with time step $T/n$.  Then, there exists $C \in \R_+$ such that for $n\ge 1$,  
$$\left|\E[(\check{L}^n_T)^3]-\E[(L_T)^3] \right| \le C \vn.$$
\end{theorem}

\begin{proof} We split the analysis in the two terms~\eqref{def_E1} and~\eqref{def_E2}.

\noindent {\bf $\bullet$ Analysis of $\mathcal{E}_1$.}
The term $\mathcal{E}_1$ can be analysed in the same way as~\cite[Theorem 2.1]{Gassiat}, but with some modifications to get symmetrized estimates on the covariance (see Lemma~\ref{lem:covariance}).  We  write
\begin{align*}
  \mathcal{E}_1&=\mathcal{E}_{1,1}+\mathcal{E}_{1,2}\text{ with}\\
  \mathcal{E}_{1,1}&=\int_0^T \int_0^{t}\E[f(X_{s}) ff'(X_{t})] (\D_s X_{t}-\D_s X_{\eta(t)}) ds dt
  \\
  \mathcal{E}_{1,2}&=\int_0^T \int_0^{\eta(t)} \E\left[\left(f(X_{s}) ff'(X_{t}) - f(X_{\eta(s)}) ff'(X_{\eta(t)})\right)\right] \D_s X_{\eta(t)} ds dt
\end{align*}
where we use for $\mathcal{E}_{1,2}$ that $\D_sX_{\eta(t)}=0$ for $s>\eta(t)$. In fact, we  get from~\eqref{Y_cal}: $$\D_sX_t=\D_s\cY_t=\mathbf{1}_{s<t}\sigma \sum_{i=1}^\infty \kappa_2^{i-1} \frac{(t-s)^{i\alpha -1}}{\Gamma(i \alpha)}, $$
which is deterministic. 

We first focus on $\mathcal{E}_{1,1}$. We note 
\begin{equation}\label{def_phi}
  \phi(s,t)=\E[f(X_s)ff'(X_t)],  \ s,t\ge 0,
\end{equation}
and have 
$\cE_{1,1}=\cE'_{1,1}+\cE''_{1,1}$
with 
\begin{align*}
  \cE'_{1,1}&=\int_0^T \int_0^{t} (\phi(s,t)-\phi(t,t)) (\D_s X_{t}-\D_s X_{\eta(t)}) ds dt \\
  \cE''_{1,1}&=\int_0^T \phi(t,t) \int_0^{t} (\D_s X_{t}-\D_s X_{\eta(t)}) ds dt.
\end{align*}
By Corollary~\ref{cor_phi}~(1), we have
\begin{align*}
  |\phi(t,t)-\phi(s,t)|&\le C |t-s|^{2\alpha-1}  .
\end{align*}
From the formula of $\D_sX_t$, we get by the triangular inequality
\begin{align*}
  |\cE'_{1,1}|\le C \sum_{i=1}^\infty  \frac{|\kappa_2^{i-1}|}{\Gamma(i \alpha)} \int_0^T  \Bigg(& \int_0^{\eta(t)} (t-s)^{2\alpha-1} \left|(t-s)^{i\alpha-1}-(\eta(t)-s)^{i\alpha-1}\right| ds \\&+\int_{\eta(t)}^t (t-s)^{2\alpha-1} (t-s)^{i\alpha-1}  \Bigg)dt .
\end{align*}
For $i\ge 2$, we have $i\alpha-1>0$ and the term in parentheses is equal to
\begin{align*}
  &\int_0^{t} (t-s)^{(2+i)\alpha-2}ds -\int_0^{\eta(t)} (t-s)^{2\alpha-1}(\eta(t)-s)^{i \alpha-1}ds \\
  &\le \frac{1}{(2+i)\alpha-1}(t^{(2+i)\alpha-1}-\eta(t))^{(2+i)\alpha-1}) \le  T^{(2+i)\alpha-2} \frac{T}{n}.
\end{align*} 
For $i=1$, the same term is equal to
\begin{align*}
  &\int_0^{\eta(t)} (t-s)^{2\alpha-1}(\eta(t)-s)^{ \alpha-1}ds -\int_0^{t} (t-s)^{3\alpha-2}ds +2\int_{\eta(t)}^{t} (t-s)^{3\alpha-2}ds \\
  &=t^{2\alpha-1}\left(  \eta(t)^{\alpha}{_2F_1}(1-2\alpha,1,1+\alpha,\frac{\eta(t)}{t})- t^\alpha {_2F_1}(1-2\alpha,1,1+\alpha,1) \right) + 2\int_{\eta(t)}^{t} (t-s)^{3\alpha-2}ds.
\end{align*}
By elementary calculations, we get $\int_0^T\int_{\eta(t)}^{t} (t-s)^{3\alpha-2}ds= O(\frac{1}{n^{3\alpha-1}}) =O(\vn)$.  Now, we write the term in parenthesis as the sum
$$ \eta(t)^{\alpha}\left({_2F_1}(1-2\alpha,1,1+\alpha,\frac{\eta(t)}{t})-{_2F_1}(1-2\alpha,1,1+\alpha,1)\right) +(\eta(t)^\alpha- t^\alpha) {_2F_1}(1-2\alpha,1,1+\alpha,1).$$
We have $\int_0^T t^{2\alpha-1} (t^\alpha- \eta(t)^\alpha)dt\le T^{2\alpha-1} \int_0^T  (t^\alpha- \eta(t)^\alpha)dt \le C \frac{T}{n}$ by~\cite[Theorem 2.4]{DFF}, and we focus now on the first term  
$$\int_0^T t^{2\alpha-1}\eta(t)^\alpha\left({_2F_1}(1-2\alpha,1,1+\alpha,\frac{\eta(t)}{t})-{_2F_1}(1-2\alpha,1,1+\alpha,1)\right) dt,$$
which is positive since $z\mapsto {_2F_1}(1-2\alpha,1,1+\alpha,z)$ is decreasing, as the first argument is negative see~\cite[Eq. 15.6.1]{Daalhuis}. We use then $\eta(t)^\alpha\le t^\alpha$. 
The function $z\mapsto {_2F_1}(1-2\alpha,1,1+\alpha,z)$ is differentiable, and the derivative is given by $\frac{1-2\alpha}{1+\alpha} {_2F_1}(2-2\alpha,2,2+\alpha,z)$, see~\cite[Eq. 15.5.1]{Daalhuis}. It is bounded uniformly for $z\in [0,1]$ if $3\alpha-2> 0$, see Equation~\eqref{max_confluent}, and then the term  is upper bounded by $CT/n$. Otherwise, 
it is upper bounded, up to a constant by $C(1-z)^{3\alpha-2}$ (see~\cite[Eq. 15.8.1]{Daalhuis}), and the same term is then  upper bounded by
$$ \int_0^T t^{3\alpha-1} \int_{\frac{\eta(t)}t}^1C(1-z)^{3\alpha-2} dz  dt=\frac{C}{3\alpha-1}\int_0^T (t^{3\alpha-1}-\eta(t)^{3\alpha-1})dt \le C \frac{T}{n},$$
by~\cite[Theorem 2.4]{DFF}. Thus, in all cases, we have $\cE'_{1,1}\le C {\bf v}_n(\alpha).$

For $\cE''_{1,1}$, we have $\int_0^t \D_s X_t-\D_sX_{\eta(t)}ds=\sigma \sum_{i=1}^\infty  \kappa_2^{i-1} \frac{t^{i\alpha}-\eta(t)^{i\alpha}}{\Gamma(i\alpha+1)}$ and thus using the boundedness of~$\phi$  we get
$$|\cE''_{1,1}|\le C\sigma \sum_{i=1}^\infty  |\kappa_2|^{i-1} \frac{1}{\Gamma(i\alpha+1)} \int_0^T(t^{i\alpha}-\eta(t)^{i\alpha})dt. $$
We have $\int_0^T(t^{i\alpha}-\eta(t)^{i\alpha})dt\le C i\alpha T^{i\alpha} \frac{T}{n}$ by~\cite[Theorem 2.4]{DFF}, which gives that $|\cE''_{1,1}|\le \frac{C}{n}$ since the series converges.

We now analyse $\cE_{1,2}=\int_0^T \int_0^{\eta(t)} (\phi(s,t)-\phi(\eta(s),\eta(t))) \D_s X_{\eta(t)} ds dt$. We proceed as in~\cite[Lemma 2.3]{Gassiat} and split the integration domain as the disjunct union of
\begin{align*}
  \mathbf{D}_1&:=\{(s,t)\in (0,T)^2: s\le \eta(t),\ t-s\le \frac{2T}{n} , \ s\ge \frac{2T}{n} \}\\
  \mathbf{D}_2&:=\{(s,t)\in (0,T)^2: s\le \eta(t),\ t-s> \frac{2T}{n} , \ s\ge \frac{2T}{n} \}\\
  \mathbf{D}_3&:=\{(s,t)\in (0,T)^2: s\le \eta(t), \ s< \frac{2T}{n} \},
\end{align*}
and we note $\cE_{1,2,1}$, $\cE_{1,2,2}$ and $\cE_{1,2,3}$ the corresponding terms. From Corollary~\ref{cor_phi}~(1) and~\eqref{der_Malliavin_X}, we get
\begin{align*}
  |\cE_{1,2,1}|&\le C  \left(\frac Tn \right)^{2\alpha-1}\int_0^T \int_{\frac{2T}{n}\vee\left(t-\frac{2T}{n}\right)}^{\eta(t)} \sigma \sum_{i=1}^\infty |\kappa_2|^{i-1}  \frac{(\eta(t)-s)^{i\alpha -1}}{\Gamma(i \alpha)} ds dt \\
  &\le C  \left(\frac Tn \right)^{2\alpha-1} \sigma \sum_{i=1}^\infty |\kappa_2|^{i-1}   \int_0^T \frac{(\eta(t)-(t-2T/n))^{i\alpha}}{\Gamma(i \alpha+1)} dt \\
  &\le C  \left(\frac Tn \right)^{2\alpha-1} \sigma \sum_{i=1}^\infty |\kappa_2|^{i-1}   T \frac{(2T/n)^{i\alpha}}{\Gamma(i \alpha+1)} = O\left( \left(\frac Tn \right)^{3\alpha-1}\right),
\end{align*}
where we used $0\le \eta(t)-(t-2T/n)\le 2T/n$ for the last inequality.

The second term is equal to
\begin{align*}
  \cE_{1,2,2}=\int_0^T \int_{\frac{2T}{n}}^{t-\frac{2T}{n}} \D_sX_{\eta(t)} \left( \int_{\eta(s)}^s \partial_1 \phi(u,t) du +\int_{\eta(t)}^t \partial_2 \phi(\eta(s),v) dv\right)dsdt. 
\end{align*}
By Corollary~\ref{cor_phi}~(2) and~(3), we get since $s\mapsto s^{2\alpha-2}$ is decreasing
\begin{align*}
  |\cE_{1,2,2}|&\le C \frac{T}{n} \int_0^T \int_{\frac{2T}{n}}^{t-\frac{2T}{n}} |\D_sX_{\eta(t)}| \left( \eta(s)^{2\alpha-2}+(t-s)^ {2\alpha-2}+\eta(t)^{2\alpha-2}+(\eta(t)-\eta(s))^{2\alpha-2} \right)dsdt \\
  &\le C \frac{T}{n} \int_0^T \int_{\frac{2T}{n}}^{t-\frac{2T}{n}} |\D_sX_{\eta(t)}|  \left(s^{2\alpha-2}+(t-s)^ {2\alpha-2}\right)dsdt.
\end{align*}
For the last inequality, we used that $\eta(t)-\eta(s)\ge (t-s)/2$ and $\eta(s)\ge s/2$ since $t-s\ge \frac{2T}{n}$ and $s\ge \frac{2T}{n}$. From~\eqref{der_Malliavin_X}, we have $|\D_sX_\eta(t)|\le C(t-s)^{\alpha-1} +C$ since $(\eta(t)-s)\ge (t-s)/2$. Therefore,
\begin{align*}
  |\cE_{1,2,2}|&\le C \frac T n \int_0^T \left( t^{2\alpha-1}+ \int_0^{t} s^{2\alpha-2} (t-s)^{\alpha-1} ds + \int_{\frac{2T}{n}}^{t-\frac{2T}{n}} s^{3\alpha-3}  ds \right)dt.
\end{align*}
Notice that $\int_0^{t} s^{2\alpha-2} (t-s)^{\alpha-1} ds =\frac{\Gamma(2\alpha-1)\Gamma(\alpha)}{\Gamma(3\alpha-1)} t^{3\alpha-1}$. Thus the two first terms are in $O(T/n)$, and the third one is of order $\frac{T}{n}\int_{\frac{2T}{n}}^{t-\frac{2T}{n}} s^{3\alpha-3}  ds =O({\bf v}_n(\alpha))$ using standard calculations.

Using again Corollary~\ref{cor_phi}~(1), we bound the last term by
\begin{align*}
  |\cE_{1,2,3}|&\le C \left(\frac T n \right)^{2\alpha -1}\int_{\frac Tn}^T \int_0^{\eta(t)\wedge \frac{2T}{n}} (\eta(t)-s)^{\alpha-1} + 1 ds dt \\
  &\le C  \left(\frac T n \right)^{2\alpha -1} \left( \frac{2T^2}{n} + \int_{\frac Tn}^{\frac{2T}{n}} \left(\frac T n \right)^{\alpha} dt +  \int_{\frac{2T}{n}}^T \eta(t)^\alpha-(\eta(t)-\frac{2T}{n})^{\alpha}dt \right)=O\left( \left(\frac Tn \right)^{3\alpha-1}\right),
\end{align*}
where we used the $\alpha$-H\"older property of $t\mapsto t^\alpha$. Gathering all the terms, we get
\begin{equation*}
  \mathcal{E}_{1,2}=\mathcal{E}_{1,2,1}+\mathcal{E}_{1,2,2}+\mathcal{E}_{1,2,3}=O({\bf v}_n(\alpha)).
\end{equation*}

\noindent {\bf $\bullet$ Analysis of $\mathcal{E}_2$.} From~\eqref{def_E2}, we have $\mathcal{E}_2=\mathcal{E}_{2,1}+\mathcal{E}_{2,2}$ with 
\begin{align*} 
  \mathcal{E}_{2,1}&= \int_0^T \int_0^{\eta(t)} \left( \E[f(X_{\eta(s)}) ff'(X_{\eta(t)})] -\E[f(\xcn_{\eta(s)}) ff'(\xcn_{\eta(t)})] \right)\D_s \xcn_{\eta(t)}   ds dt\\
  \mathcal{E}_{2,2}&=\int_0^T \int_0^{\eta(t)} \E[f(X_{\eta(s)}) ff'(X_{\eta(t)})] (\D_s X_{\eta(t)}- \D_s \xcn_{\eta(t)})   ds dt.
\end{align*}

We first focus on $\mathcal{E}_{2,1}$. We have by using Corollary~\ref{cor_weak_ito}:%with Proposition~\ref{prop_OUexplicit} and Remark~\ref{rk_unifbound} for the uniform bound on the two first moments,
\begin{align*}
 | \mathcal{E}_{2,1} |\le &C \int_0^T \int_0^{\eta(t)} \Bigg(|\E[X_{\eta(s)}-\xcn_{\eta(s)}]|+|\E[X_{\eta(t)}-\xcn_{\eta(t)}]|+|Var(X_{\eta(s)})-Var(\xcn_{\eta(s)})| \\
  &+|Var(X_{\eta(t)})-Var(\xcn_{\eta(t)})| + |Cov(X_{\eta(s)},X_{\eta(t)})-Cov(\xcn_{\eta(s)},\xcn_{\eta(t)})|\Bigg)\D_s \xcn_{\eta(t)}ds dt.
\end{align*}
By using Theorem~\ref{thm_covariance_error} and using that $\E[X_{t_0}-\xcn_{t_0}]=Var(X_{t_0})=Var(\xcn_{t_0})=Cov(X_{t_0},X_{t_k})=Cov(\xcn_{t_0},\xcn_{t_k})=0$, we get
\begin{align}
  |\mathcal{E}_{2,1}| &\le 
   C \int_{t_1}^T \int_{t_1}^{\eta(t)} \Bigg((1+\eta(s)^{\alpha-1}+\eta(t)^{\alpha-1}){\bf v}_n(\alpha) + \frac{\eta(s)^{2\alpha-2}+\eta(t)^{2\alpha-2}}{n} \Bigg)\D_s \xcn_{\eta(t)}ds dt \notag \\
   &\le 
   C \int_{t_1}^T \int_{t_1}^{\eta(t)} \Bigg((1+\eta(s)^{\alpha-1}){\bf v}_n(\alpha) + \frac{\eta(s)^{2\alpha-2}}{n} \Bigg)\D_s \xcn_{\eta(t)}ds dt \label{eq:boundE21},
\end{align}
since $\alpha-1<0$. By~\eqref{eq_dxcn}, we get for $k>1$ and $\beta\in(-1,0]$
$$\int_{t_1}^{t_k} \eta(s)^\beta |\D_s \xcn_{t_k}|ds\le |\kappa_2|  \sum_{i=1}^{k-1} c_{i,k} \int_{t_1}^{t_k} \eta(s)^\beta |\D_s \xcn_{t_i}|ds + \sigma \int_{t_1}^{t_k} \eta(s)^\beta \frac{(t_k-s)^{\alpha-1}}{\Gamma(\alpha)} ds$$
with $c_{i,k}=\int_{t_i}^{t_{i+1}}\frac{(t_k-u)^{\alpha-1}}{\Gamma(\alpha)}du$. We observe that for $s\ge t_1$, $\eta(s)\ge s/2$ and therefore the last integral is upper bounded by
$$ 2^{-\beta}\int_0^{t_k}s^{\beta}\frac{(t_k-s)^{\alpha-1}}{\Gamma(\alpha)} ds=2^{-\beta} \frac{\Gamma(\beta+1)}{\Gamma(\alpha+\beta+1)} t_k^{\alpha+\beta}\le C_{\alpha,\beta}\left(\mathbf{1}_{\alpha+\beta\ge 0} T^{\alpha+\beta}+\mathbf{1}_{\alpha+\beta< 0} t_1^{\alpha+\beta} \right). $$
We now apply the Gronwall type lemma~\cite[Lemma 3.4]{LiZe} to get 
\begin{equation}\label{majo_der_malliavin_sch}
  \int_{t_1}^{t_k} \eta(s)^\beta |\D_s \xcn_{t_k}|ds\le C \left(\mathbf{1}_{\alpha+\beta\ge 0} T^{\alpha+\beta}+\mathbf{1}_{\alpha+\beta< 0} t_1^{\alpha+\beta} \right), \ 1\le k\le n. 
\end{equation}
We apply this bound in~\eqref{eq:boundE21} with $\beta=0$, $\beta=\alpha-1$ and $\beta=2\alpha-2$ and get
$$  |\mathcal{E}_{2,1} | \le 
C T  \left( \left(C(T^\beta+T^{2\alpha-1})\right){\bf v}_n(\alpha) +  \frac{\mathbf{1}_{\alpha\ge 2/3}T^{3\alpha-2}+\mathbf{1}_{\alpha< 2/3}t_1^{3\alpha-2} }{n}\right).  $$
 Note that $\frac {t_1^{3\alpha-2}}{n}=\frac{T^{3\alpha-2}}{n^{3\alpha-1} }\le  T^{3\alpha-2} {\bf v}_n(\alpha)$, which gives 
 $$ |\mathcal{E}_{2,1}|\le C \vn .$$
% Therefore, 
% \begin{align*}
%   \mathcal{E}_{2,1} &\le C \int_{t_1}^T \left((1+\eta(t)^{\alpha-1}){\bf v}_n(\alpha) +\frac{\eta(t)^{2\alpha-2}}{n}\right) dt \\
% &\le C \int_{0}^T \left((1+t^{\alpha-1}){\bf v}_n(\alpha) +\frac{t^{2\alpha-2}}{n}\right) dt= C\left(T+\frac{T^\alpha}{\alpha}\right){\bf v}_n(\alpha) + C\frac{T^{2\alpha-1}}{(2\alpha-1)n} ,
% \end{align*}
% where we have used that $\eta(t)\ge t/2$ on $t\ge t_1$ for the last inequality. 

Last, we focus on  $\mathcal{E}_{2,2}$.
 We have 
$$\mathcal{E}_{2,2}=\frac{T}{n} \sum_{k=1}^{n-1} \int_0^{t_k} \E[f(X_{\eta(s)}) ff'(X_{t_k})](\D_sX_{t_k}-\D_s \xcn_{t_k})ds,$$
and we focus on each integral. Let $\phi$ be defined by~\eqref{def_phi}. By Corollary~\ref{cor_phi}, it is continuously differentiable and we write $\phi(\eta(s),t_k)=\phi(t_k,t_k)-\int_{\eta(s)}^{t_k} \partial_1 \phi(u,t_k)du$. This gives
\begin{align*}
  &\int_0^{t_k} \E[f(X_{\eta(s)}) ff'(X_{t_k})] (\D_sX_{t_k}-\D_s \xcn_{t_k}) ds\\
&=\phi(t_k,t_k)\int_0^{t_k} (\D_s X_{t_k}-\D_s \xcn_{t_k})ds -  \int_0^{t_k} \int_0^{t_k}  \mathbf{1}_{\eta(s) \le u} \partial_1 \phi(u,t_k) (\D_s X_{t_k}-\D_s \xcn_{t_k}) ds du\\
&=\phi(t_k,t_k)\int_0^{t_k} (\D_s X_{t_k}-\D_s \xcn_{t_k})ds -  \int_0^{t_k} \partial_1 \phi(u,t_k) \left( \int_0^{\bar{\eta}(u)}   (\D_s X_{t_k}-\D_s \xcn_{t_k}) ds \right) du,
\end{align*}
by using Fubini theorem. 
By using the triangle inequality, $\phi$ bounded, Lemma~\ref{lem_err_Malliavin3} and Corollary~\ref{cor_phi}, we get
\begin{align*}
  &\left|\int_0^{t_k} \E[f(X_{\eta(s)}) ff'(X_{t_k})] (\D_sX_{t_k}-\D_s \xcn_{t_k}) ds\right|\\
&\le C \left( \frac{T}{n} + \int_0^{t_k} \left(u^{2\alpha-2}+(t_k-u)^{2\alpha-2}\right) \left| \int_0^{\bar{\eta}(u)}   (\D_s X_{t_k}-\D_s \xcn_{t_k}) ds \right|  du \right)\\
&\le C \frac{T}{n}  \left(1  +  t_k^{2\alpha-1} \right) \le C \frac{T}{n}  \left(1  +  T^{2\alpha-1} \right), 
\end{align*}
since $2\alpha-1>0$.  Therefore, we get $|\mathcal{E}_{2,2}|\le  CT(1+T^{2\alpha-1})\frac{T}{n}$, which completes the proof.
\end{proof}

\section{Weak error approximation with polynomial test functions}\label{Sec_polynomial}
 
 We now state our main result on the weak convergence rate for any polynomial test function. To get this result, we build up on the framework developed in Friz et al.~\cite{FSW} that we extend to our setting. In contrast with the cubic case with $b=0$ considered in Section~\ref{Sec_cubic}, it is worth noting that the analysis of some terms is more challenging for this general context. In particular the term $\mathcal{E}_{2,2}$ requires a much more involved analysis when studied with general polynomial test functions.       

\begin{theorem}\label{thm_main}
   Let~$L$ be the process defined by~\eqref{SVE_OU} with $b,f\in  \mathcal{C}_{\exp}^\infty(\R, \R)$.
    Let $\check{L}^n$ denote the approximation scheme~\eqref{eq:scheme_price} with time step $T/n$ with $n\ge 1$.  Then, for any polynomial function $\Phi:\R \to \R$, there exists $C \in \R_+$ such that for all $n$,   
$$\left|\E[\Phi(\check{L}^n_T)]-\E[\Phi(L_T)] \right| \le C \vn .$$
\end{theorem}

The proof of this theorem is the goal of Section~\ref{Sec_polynomial}. It is split into two main steps. The first one consists in representing the weak error by the mean of a sum of iterated integrals coming from the Malliavin duality formula. This is made in Subsection~\ref{Subsec_Repr}. In Subsection~\ref{Subsec_weakerror}, we analyse the weak error itself.  Guided by the cubic case, we decompose the error in a similar way and provide the analysis of each term. It turns out that the analysis of some terms is much involved than for the cubic case since we have, due to the combinatorics generated by higher polynomial degrees, more different terms to analyse.

\subsection{Representation of the weak error}\label{Subsec_Repr} We adopt and extend the framework introduced by Friz et al.~\cite{FSW} to analyse the weak error for polynomial test functions. Without loss of generality as $x\mapsto (L_0+x)^N$ is still a polynomial function, we assume that $L_0=0$ in Section~\ref{Sec_polynomial}.
We introduce, for $m\in \N$, 
$$\Del_m:= \{(r_1,\dots,r_m): 0<r_m<\dots<r_1<T \},$$
the open simplex of order~$m$. 
We recall that $b,f\in \mathcal{C}_{\exp}^\infty(\R, \R)$. Let $\mathcal{Q}^m$  be the set of functions $F:\R^m \times \Del_m \to \R$ such that for any $\rb=(r_1,\dots,r_m)\in \Del_m$, $F(\cdot,\rb)$ is a smooth function on~$\R^m$ with derivatives of all orders with exponential growth. For $N\ge 1$, we define the three following operators $\mathcal{I}^N,\mathcal{J}^N,\mathcal{K}^N: \mathcal{Q}^m\to \mathcal{Q}^{m+1}$ by
\begin{align}
  &(\mathcal{I}^N F)(x_1,\dots,x_m,y,r_1,\dots,r_m,s)=\rho N f(y) \sum_{j=1}^m \partial_{x_j}F(x_1,\dots,x_m,r_1,\dots,r_m) D_s X_{r_j}, \label{def_I}\\
  &(\mathcal{J}^N F)(x_1,\dots,x_m,y,r_1,\dots,r_m,s)=\frac{N(N-1)}2 f^2(y)F(x_1,\dots,x_m,r_1,\dots,r_m) \label{def_J}\\
  &(\mathcal{K}^N F)(x_1,\dots,x_m,y,r_1,\dots,r_m,s)=N b(y)F(x_1,\dots,x_m,r_1,\dots,r_m) \label{def_K},
\end{align}
where $\rb=(r_1,\dots,r_m)\in \Del_m$, $x\in \R^m$, $s\in(0,r_m)$ and $y\in \R$. Recall that $D_sX_{r_j}$ is the (deterministic) Malliavin derivative given by Proposition~\ref{prop_OUexplicit}.

\begin{lemma}\label{lem_recXN}
  Let $N\ge 1$. Let $m\ge 1$, $F\in \mathcal{Q}^m$, $\rb \in \Del_m$, $t\in[0,r_m)$ and denote ${\bf X}_{\rb}=(X_{r_1},\dots,X_{r_m})$. We have, for $N\ge 1$
\begin{align*}
  \E\left[ (L_t)^N F({\bf X}_{\rb},\rb)\right]=&\int_0^t  \E\left[ (L_s)^{N-1} (\mathcal{I}^N F)({\bf X}_{\rb},X_s,\rb,s) \right] ds  \\
  &+\int_0^t  \E\left[ (L_s)^{N-1} (\mathcal{K}^N F)({\bf X}_{\rb},X_s,\rb,s) \right] ds \\
  & +\int_0^t  \E\left[ (L_s)^{N-2} (\mathcal{J}^N F)({\bf X}_{\rb},X_s,\rb,s) \right] ds,
\end{align*}
where the last integral is meant to be $0$ for $N=1$.
\end{lemma}
\begin{proof}
First, we apply Itô formula to get 
$$L_t^N=\int_0^t NL_s^{N-1} f(X_s) dB_s +\int_0^t NL_s^{N-1} b(X_s) +\frac{N(N-1)}{2}L_s^{N-2} f^2(X_s) ds.$$  
From the Clark-Ocone formula, we have 
$$F({\bf X}_{\rb},\rb)=\E[F({\bf X}_{\rb},\rb)]+\sum_{j=1}^m \int_0^{r_j} \E[\partial_{x_j}F({\bf X}_{\rb},\rb)|\mathcal{F}_s] D_sX_{r_j} dW_s, $$
where $(\mathcal{F}_s,s\ge 0)$ denotes the natural filtration of $(B,W)$. This gives 
\begin{align*}
\E\left[ \left(\int_0^t NL_s^{N-1} f(X_s) dB_s\right)F({\bf X}_{\rb},\rb) \right]&=\rho \sum_{j=1}^m \int_0^{t} N \E[(L_s)^{N-1} f(X_s)\partial_{x_j}F({\bf X}_{\rb},\rb)]D_sX_{r_j} ds\\
&=\int_0^t\E\left[ (L_s)^{N-1} (\mathcal{I}^N F)({\bf X}_{\rb},X_s,\rb,s) \right] ds.
\end{align*}
The two other terms follow immediately.
\end{proof}

Lemma~\ref{lem_recXN} gives a recursive way to obtain~$\E[(L_t)^N]$. We now want to remove this recursion and get an explicit form. To do so, we use the same procedure as Friz et al.~\cite{FSW}. Let $\mathcal{W}$ be the set of all words with letters $\{I,J,K\}$. We denote by $|\cdot|$ the length of a word, i.e. $|w|=m$ for $w=w_1\dots w_m\in \mathcal{W}$. We define also $\ell:\mathcal{W}\to \N$ by 
$$\ell(w)=\sum_{j=1}^{|w|} \ell(w_j), \text{ for } w=w_1\dots w_m,$$
with $\ell(I)=\ell(K)=1$ and $\ell(J)=2$. Heuristically, $\ell$ counts the decrease of the degree of the monomial part of the expectation. To any word $w\in \cW$, we associate an application $\iota$ defined as follows:
\begin{align}\label{def_iota}
  \iota(w)=\begin{cases}
    \iota (w_1\dots w_{m-1}) \circ \mathcal{I}^{\ell(w)} \quad \text{ if } w_m=I, \\
    \iota (w_1\dots w_{m-1}) \circ \mathcal{J}^{\ell(w)} \quad \text{ if } w_m=J, \\
    \iota (w_1\dots w_{m-1}) \circ \mathcal{K}^{\ell(w)} \quad \text{ if } w_m=K, \\
  \end{cases}
\end{align}
where by convention  $\iota(w_1\dots w_{m-1})$ is the identity for $m=1$ and the operators $\mathcal{I}^{\ell(w)}$, $\mathcal{J}^{\ell(w)}$ and $\mathcal{K}^{\ell(w)}$ are respectively defined by~\eqref{def_I}, \eqref{def_J} and \eqref{def_K}. 
Since the operator $\mathcal{I}$  involves only derivatives, we get that for the constant function~$1$, $\iota(w)(1)=0$ for any word $w$ with last letter~$I$.

\begin{theorem} \label{thm_moments1}
We have, for any $N\ge 1$,
\begin{align*}
  &\E\left[ (L_t)^N \right] =\sum_{w\in \cW , \ell(w)=N} \int_0^t \dots \int_0^{r_{|w|-1}} \E\left[(\iota(w) 1)(X_{r_1},\dots,X_{r_{|w|}},  r_1,\dots, r_{|w|}) \right] dr_{|w|}\dots dr_1.
\end{align*}
\end{theorem}
\begin{proof} We follow~\cite{FSW} and prove the following more general result: for  $m\ge 1$, ${\bf u}\in \Del_m$, $t<u_m$ and $F\in \mathcal{Q}^m$, we have
  \begin{align*}
    &\E\left[ (L_t)^N F({\bf X}_{\bf u},{\bf u})\right] \\&=\sum_{w\in \cW , \ell(w)=N} \int_0^t \dots \int_0^{r_{|w|-1}} \E\left[(\iota(w) F)({\bf X}_{\bf u},X_{r_1},\dots,X_{r_{|w|}}, {\bf u}, r_1,\dots, r_{|w|}) \right] dr_{|w|}\dots dr_1.
  \end{align*}
 We show this result by induction on~$N$. For $N=1$, this is given by Lemma~\ref{lem_recXN}. For $N\ge 2$, we apply again Lemma~\ref{lem_recXN} and the induction hypothesis as in the proof of~\cite[Proposition 3.8]{FSW}, using that for $w\in \mathcal{W}$ such that $\ell(w)=N$, either $\iota(w)F=\iota(w_1\dots w_{|w|-1}) \mathcal{I}^N F$ or $\iota(w)F=\iota(w_1\dots w_{|w|-1}) \mathcal{J}^N F$ or $\iota(w)F=\iota(w_1\dots w_{|w|-1}) \mathcal{K}^N F$. We conclude by taking~$F \equiv 1$.
\end{proof}

We now want a similar formula for the approximation scheme~\eqref{eq:scheme_price}. For $N\ge 1$, we define the following operator $\check{\mathcal{I}}^N: \mathcal{Q}^m\to \mathcal{Q}^{m+1}$ by
\begin{align}
  &(\check{\mathcal{I}}^N F)(x_1,\dots,x_m,y,r_1,\dots,r_m,s)=\rho N f(y) \sum_{j=1}^m \partial_{x_j}F(x_1,\dots,x_m,r_1,\dots,r_m) D_s \xcn_{\eta(r_j)}, \label{def_Icheck}
  \end{align}
where $\rb=(r_1,\dots,r_m)\in \Del_m$, $x\in \R^m$, $s\in(0,r_m)$ and $y\in \R$. Recall that $D_s \xcn_{r_j}$ is the (deterministic) Malliavin derivative given by Lemma~\ref{lem_err_Malliavin3}. We then define for any word $w\in \mathcal{W}$ with $w=w_1\dots w_m$ the application 
\begin{align*}%\label{def_iotacheck}
  \check{\iota}(w)=\begin{cases}
    \check{\iota}(w_1\dots w_{m-1}) \circ \check{\mathcal{I}}^{\ell(w)} \quad \text{ if } w_m=I, \\
    \check{\iota} (w_1\dots w_{m-1}) \circ \mathcal{J}^{\ell(w)} \quad \text{ if } w_m=J, \\
    \check{\iota} (w_1\dots w_{m-1}) \circ \mathcal{K}^{\ell(w)} \quad \text{ if } w_m=K, \\
  \end{cases}
\end{align*}
where again by convention  $\check{\iota}(w_1\dots w_{m-1})$ is the identity for $m=1$ and the operators $\check{\mathcal{I}}^{\ell(w)}$, $\mathcal{J}^{\ell(w)}$ and $\mathcal{K}^{\ell(w)}$ are respectively defined by~\eqref{def_Icheck}, \eqref{def_J} and \eqref{def_K}. 
\begin{theorem} \label{thm_moments2}
  We have, for any $N\ge 1$,
  \begin{align*}
    &\E\left[ (\lcn_t)^N \right] =\sum_{w\in \cW , \ell(w)=N} \int_0^t \dots \int_0^{r_{|w|-1}} \E\left[(\check{\iota}(w) 1)(\check{X}_{\eta(r_1)},\dots,\check{X}_{\eta(r_{|w|})},  r_1,\dots, r_{|w|}) \right] dr_{|w|}\dots dr_1.
  \end{align*}
\end{theorem}
\begin{proof}
  The proof is similar to the one of Theorem~\ref{thm_moments1}. More precisely, we first show for $N\ge 1$, $m\ge 1$, $F\in \mathcal{Q}^m$, $\rb \in \Del_m$, and $t\in[0,r_m)$
\begin{align*}
  \E\left[ (\lcn_t)^N F(\check{\mathbf{X}}^n_{\rb},\rb)\right]=&\int_0^t  \E\left[ (\lcn_s)^{N-1} (\check{\mathcal{I}}^N F)(\check{\mathbf{X}}^n_{\rb},\xcn_{\eta(s)},\rb,s) \right] ds  \\
  &+\int_0^t  \E\left[ (\lcn_s)^{N-1} ({\mathcal{K}}^N F)(\check{\mathbf{X}}^n_{\rb},\xcn_{\eta(s)},\rb,s) \right] ds \\
  & +\int_0^t  \E\left[ (\lcn_s)^{N-2} ({\mathcal{J}}^N F)(\check{\mathbf{X}}^n_{\rb},\xcn_{\eta(s)},\rb,s) \right] ds,
\end{align*}
 with $\check{\mathbf{X}}^n_{\rb}=(\xcn_{r_1},\dots,\xcn_{r_m})$, as in Lemma~\ref{lem_recXN}. We then conclude by induction. \end{proof}

% \begin{proof}
% First, we apply Itô formula to get 
% $$(\lcn_t)^N=\int_0^t N(\lcn_s)^{N-1} f(\xcn_{\eta(s)}) dB_s +\int_0^t N(\lcn_s)^{N-1} b(\xcn_{\eta(s)}) +\frac{N(N-1)}{2}(\lcn_s)^{N-2} f^2(\xcn_{\eta(s)}) ds.$$  
% From the Clark-Ocone formula, we have 
% $$F(\check{\bf X}^n_{\rb},\rb)=\E[F(\check{\bf X}^n_{\rb},\rb)]+\sum_{j=1}^m \int_0^{r_j} \E[\partial_{x_j}F(\check{\bf X}^n_{\rb},\rb)|\mathcal{F}_s] D_s \xcn_{r_j} dW_s, $$
% where $(\mathcal{F}_s,s\ge 0)$ denotes the natural filtration of $(B,W)$. This gives 
% \begin{align*}
% \E\left[ \left(\int_0^t NL_s^{N-1} f(X_s) dB_s\right)F({\bf X}_{\rb},\rb) \right]&=\rho \sum_{j=1}^m \int_0^{t} N \E[(L_s)^{N-1} f(X_s)\partial_{x_j}F({\bf X}_{\rb},\rb)]D_sX_{r_j} ds\\
% &=\int_0^t\E\left[ (L_s)^{N-1} (\mathcal{I}^N F)({\bf X}_{\rb},X_s,\rb,s) \right] ds.
% \end{align*}
% The two other terms follow immediately.
% \end{proof}

Theorems~\ref{thm_moments1} and~\ref{thm_moments2} give the moments explicitly by means of the the functions $\iota(w)1$ and $\check\iota(w)1$. However, these functions are defined by induction, and the goal of the next proposition is to get rid of it. 
\begin{prop}\label{prop_iota} (Calculation of $\iota(w)$ defined by~\eqref{def_iota}.)
Let $q\ge 0$, $\psi_1,\dots,\psi_q \in  \mathcal{C}_{\exp}^\infty(\R, \R)$ and  $F\in \mathcal{Q}^q$ such that $F(x_1,\dots,x_q,r_1,\dots,r_q)=\prod_{i=1}^q \psi_i(x_i)$  (convention $F \equiv 1$ if $q=0$). Let  $w \in \mathcal{W}$ with $w=w_1\dots w_{m}$ and $m=|w|\ge 1$. We define $N_I^w=\{j : w_j=I\}$, $N_J^w=\{j : w_j=J\}$ and $N_K^w=\{j : w_j=K\}$. 
\begin{itemize}
\item  If $N_I^w=\emptyset$ (i.e $w\in \mathcal{W}$ does not contain the letter~$I$), then we have 
$$(\iota(w)F)(x_1,\dots,x_{q+m},r_1,\dots,r_{q+m})=C_w \prod_{i=1}^q \psi_i(x_i) \prod_{l \in N^w_J } f^2(x_{m+q-l+1}) \prod_{l \in N^w_K } b(x_{m+q-l+1}).$$  
\item Otherwise, let $k=\#N_I^w\le m$ and $1\le j_1<\dots<j_k\le m$ such that $N_I^w=\{j_1,\dots,j_k\}$. Then, we have
\begin{align*}
  &(\iota(w)F)(x_1,\dots,x_q,\dots,x_{q+m},r_1,\dots,r_q,\dots,r_{q+m})\\
  &=C_{w} \sum_{l_1=1}^{m+q-j_1} \dots  \sum_{l_k=1}^{m+q-j_k} \D_{r_{m+q-j_1+1}}X_{r_{l_1}} \dots \D_{r_{m+q-j_k+1}}X_{r_{l_k}}\\
  & \times \partial_{x_{l_1}} \dots \partial_{x_{l_k}}\left( \prod_{i=1}^q \psi_i(x_i) \prod_{l \in N^w_J } f^2(x_{m+q-l+1}) \prod_{l \in N^w_K } b(x_{m+q-l+1}) \prod_{l \in N^w_I } f(x_{m+q-l+1})\right),
\end{align*}
\end{itemize}
with $C_w= 2^{-\#N_J^w}\rho^{\#N_I^w}\times \ell(w)! $. 

The same formulas hold for $\check{\iota}(w)F$, with $\D_{r_{m+q-j_\nu+1}} \xcn_{\eta(r_{l_\nu})}$ instead of $\D_{r_{m+q-j_\nu+1}}X_{r_{l_\nu}}$ for $\nu \in \{1,\dots,k\}$.
\end{prop}
%{\bf J'ai réindicé les indices muets l1...lk de FSW}
\begin{remark}
  This proposition slightly extends (in our framework) \cite[Proposition 3.13]{FSW} that calculates $\iota(w)1$. Here, we use a different induction argument that allows to better keep track of the constant~$C_w$. 
\end{remark}
\begin{remark}\label{rk_iota1}
  Proposition~\ref{prop_iota} gives $\iota(w)1$ simply by taking formally $q=0$ in the above formula. This is clear when $w$ has a single letter (i.e. $|w|=1$). We indeed have $\iota(I)1=0$, $(\iota(J)1)(x,r)=f^2(x)$ and $(\iota(K)1)(x,r)=b(x)$. For $|w|\ge 2$, we just use the induction $w=w'w_m$ with $w_m\in\{I,J,K\}$ and apply Proposition~\ref{prop_iota} for $w'$, $q=1$ and  $F(x,r)=(\iota(K)1)(w_m,r) \in \mathcal{Q}^1$.   
\end{remark}

\begin{proof}
  We prove this result by induction on the length $|w|=m$. For $m=1$, if $w=J$ (resp. $w=K$) then $N_J^w=\{1\}$ (resp. $N_K^w=\{1\}$) so that 
  $$(\iota(J)F)(x_1,\dots,x_{q+1},r_1,\dots,r_{q+1})=f^2(x_{q+1}) \prod_{i=1}^q \psi_i(x_i)$$
   (resp. $(\iota(K)F)(x_1,\dots,x_{q+1},r_1,\dots,r_{q+1})=b(x_{q+1}) \prod_{i=1}^q \psi_i(x_i)$). If $w=I$, we have $j_1=1$ and 
$$(\iota(I)F)(x_1,\dots,x_{q+1},r_1,\dots,r_{q+1})= \rho f(x_{q+1}) \sum_{l=1}^q \partial_{x_l}\left( \prod_{i=1}^q \psi_i(x_i)\right) \D_{r_{q+1}}X_{r_l}.$$

We now prove the induction step and consider $m\ge 2$. We write $w=w'w_m$, where $w' \in \mathcal{W}$ with $|w'|=m-1$. Let us first consider the case $w=w'J$ so that $\iota(w)=\iota(w')\circ \mathcal{J}^{\ell(w)}$. We have 
$$ \mathcal{J}^{\ell(w)}F(x_1,\dots,x_{q+1},r_1,\dots,r_{q+1})= \frac{\ell(w)(\ell(w)-1)}{2}f^2(x_{q+1}) \prod_{i=1}^q \psi_i(x_i),$$
and we apply the induction hypothesis to this function. If $w'$ does not contain the letter~$I$, we get 
\begin{align*}
  &(\iota(w')\mathcal{J}^{\ell(w)}F)(x_1,\dots,x_{q+m},r_1,\dots,r_{q+m})\\
  &=C_{w'} \frac{\ell(w)(\ell(w)-1)}{2}f^2(x_{q+1}) \prod_{i=1}^q \psi_i(x_i)  \prod_{l \in N^{w'}_J } f^2(x_{m+q-l+1}) \prod_{l \in N^{w'}_K } b(x_{m+q-l+1}) \\
  &=C_w  \prod_{i=1}^q \psi_i(x_i)  \prod_{l \in N^{w}_J } f^2(x_{m-q-l+1}) \prod_{l \in N^{w}_K } b(x_{m+q-l+1}),
\end{align*}
with $C_{w}=\frac{\ell(w)(\ell(w)-1)}{2}C_{w'}$ since $N^{w}_J=N^{w'}_J\cup{\{m\}}$ and $N^{w}_K=N^{w'}_K$. If $w'$ contains the letter~$I$, we get 
\begin{align*}
  &(\iota(w')\mathcal{J}^{\ell(w)}F)(x_1,\dots,x_{q+m},r_1,\dots,r_{q+m})=C_{w'} \frac{\ell(w)(\ell(w)-1)}{2}\\
  &\times
  \sum_{l_1=1}^{m+q-j'_1}\dots \sum_{l_k=1}^{m+q-j'_k}  \D_{r_{m+q-j'_1+1}}X_{r_{l_1}}\dots \D_{r_{m+q-j'_k+1}}X_{r_{l_k}}\\
  &  \times  \partial_{x_{l_1}} \dots \partial_{x_{l_k}}\left( \prod_{i=1}^q \psi_i(x_i) f^2(x_{q+1}) \prod_{l \in N^{w'}_J } f^2(x_{m+q-l+1}) \prod_{l \in N^{w'}_K } b(x_{m+q-l+1}) \prod_{l \in N^{w'}_I } f(x_{m+q-l+1})\right)\\
  &=C_{w} \sum_{l_1=1}^{m+q-j_1}\dots \sum_{l_k=1}^{m+q-j_k}  \D_{r_{m+q-j_1+1}}X_{r_{l_1}}\dots \D_{r_{m+q-j_k+1}}X_{r_{l_k}} \\
  & \times \partial_{x_{l_1}} \dots \partial_{x_{l_k}}\left( \prod_{i=1}^q \psi_i(x_i) \prod_{l \in N^w_J } f^2(x_{m+q-l+1}) \prod_{l \in N^w_K } b(x_{m+q-l+1}) \prod_{l \in N^w_I } f(x_{m+q-l+1})\right),
\end{align*}
with $C_{w}=\frac{\ell(w)(\ell(w)-1)}{2}C_{w'}$, using again $N^{w}_J=N^{w'}_J\cup{\{m\}}$, $N^{w}_K=N^{w'}_K$, $N^{w}_I=N^{w'}_I$ which amounts to have $j'_1=j_1,\dots,j'_k=j_k$.

The proof is similar if $w=w'K$ with $C_w=C_{w'} \ell(w)$. For $w=w'I$, we have $$ \mathcal{I}^{\ell(w)}F(x_1,\dots,x_{q+1},r_1,\dots,r_{q+1})= \rho \ell(w) f(x_{q+1}) \sum_{l=1}^q \partial_{x_l}\left( \prod_{i=1}^q \psi_i(x_i)\right) \D_{r_{q+1}}X_{r_l}.$$
Using the linearity of the operator $\iota(w')$ and the fact that $ \partial_{x_l}\left( \prod_{i=1}^q \psi_i(x_i)\right)  =\psi_l'(x_l) \prod_{i\not = l} \psi_i(x_i)$ is a product of smooth functions, we get by induction (in the case where $w'$ contains the letter~$I$, the other case can be analysed in the same way applying the first assertion of Proposition~\ref{prop_iota})
\begin{align*}
  &(\iota(w')\mathcal{I}^{\ell(w)}F)(x_1,\dots,x_{q+m},r_1,\dots,r_{q+m})=\rho C_{w'} \ell(w) \times \sum_{l=1}^q \sum_{l_1=1}^{m+q-j'_1}\dots \sum_{l_k=1}^{m+q-j'_k} \bigg\{\\
  & \partial_{x_{l_1}} \dots \partial_{x_{l_k}}\left( \partial_{x_{l}}\left(\prod_{i=1}^q \psi_i(x_i) f(x_{q+1}) \right)\prod_{\bar{l} \in N^{w'}_J } f^2(x_{m+q-\bar{l}+1}) \prod_{\bar{l} \in N^{w'}_K } b(x_{m+q-\bar{l}+1}) \prod_{\bar{l} \in N^{w'}_I } f(x_{m+q-\bar{l}+1})\right)\\
  &  \times  \D_{r_{m+q-j'_1+1}}X_{r_{l_1}} \dots \D_{r_{m+q-j'_k+1}}X_{r_{l_k}}\times D_{r_{q+1}}X_{r_l} \bigg\},
\end{align*}
where $k=\# N_I^{w'}$ and $N_I^{w'}=\{j'_1,\dots,j'_k\}$.
We notice that $j'_1=j_1,\dots,j'_k=j_k$ and $j_{k+1}=m$ so that $m+q-j_{k+1}+1=q+1$. Besides, $m+q-\bar{l}+1 \ge q+2$ since $\bar{l}\le m-1$, so that we can put the product of functions inside the derivation with respect to $x_l$, for $l\le q$. Since $N_I^w=N_I^{w'}\cup \{m \}$, $N_J^w=N_J^{w'}$ and $N_K^w=N_K^{w'}$, this gives 
\begin{align*}
  &(\iota(w')\mathcal{I}^{\ell(w)}F)(x_1,\dots,x_{q+m},r_1,\dots,r_{q+m})=\\
  &=C_w \sum_{l_1=1}^{m+q-j'_1}\dots \sum_{l_{k+1}=1}^{m+q-j_{k+1}} \D_{r_{m+q-j'_1+1}}X_{r_{l_1}} \dots \D_{r_{m+q-j_{k+1}+1}}X_{r_{l_{k+1}}}\\
  &  \times  \partial_{x_{l_1}} \dots \partial_{x_{l_{k+1}}}\left( \prod_{i=1}^q \psi_i(x_i)  \prod_{l \in N^{w}_J } f^2(x_{m+q-l+1}) \prod_{l \in N^{w}_K } b(x_{m+q-l+1}) \prod_{l \in N^{w}_I } f(x_{m+q-l+1})\right),
\end{align*}
with $C_w=\rho C_{w'} \ell(w)$.
\end{proof}

\subsection{Weak error analysis}\label{Subsec_weakerror}
We are now in position to analyse the error 
\begin{align*}
   &\E[(L_T)^N]-\E[(\lcn_T)^N]\\
   &=  \sum_{w\in \cW , \ell(w)=N} \int_0^T \dots \int_0^{r_{|w|-1}} \E\bigg[(\iota(w) 1)(X_{r_1},\dots,X_{r_{|w|}},  r_1,\dots, r_{|w|}) \\
   &\quad \quad -(\check{\iota}(w) 1)(\xcn_{\eta(r_1)},\dots,\xcn_{\eta(r_{|w|})},  r_1,\dots, r_{|w|}) \bigg] dr_{|w|}\dots dr_1.
\end{align*}
By Proposition~\ref{prop_iota} and Remark~\ref{rk_iota1}, the expectation can be written as a finite sum of the following terms  
\begin{align*}
&\E\left[\Psi^w_{l_1,\dots,l_k}(X_{r_1},\dots,X_{r_{m}}) \right] \prod_{\nu=1}^k \D_{r_{m-j_\nu+1}}X_{r_{l_\nu}} \\&  -\E\left[\Psi^w_{l_1,\dots,l_k}(\xcn_{\eta(r_1)},\dots,\xcn_{\eta(r_{m})}) \right] \prod_{\nu=1}^k \D_{r_{m-j_\nu+1}} \xcn_{\eta(r_{l_\nu})} ,
\end{align*}
with $m=|w|$, $k=\#N_I^w$, $\{j_1,\dots,j_k\}=N_I^w$ and $\Psi^w_{l_1,\dots,l_k}:\R^m \to \R^m$ is a product of $m$ functions i.e.~$\Psi^w_{l_1,\dots,l_k}(x_1,\dots,x_m)=\prod_{m'=1}^m g_{m'}(x_{m'})$ with $g_1,\dots,g_m \in \mathcal{C}_{\exp}^\infty(\R, \R)$, see Equation~\eqref{def_Cexpinfty}. Note that every word~$w$ finishing by the letter~$I$ gives $\iota(w)1=0$. Thus, we may assume in what follows that $j_k<m$ and $k<m$.  

In the following, we drop the superscript and subscript of $\Psi$ to have a simpler notation, and we analyse the error term
\begin{align*}
\mathcal{E}=\int_0^T \dots \int_0^{r_{m-1}} &\E\left[\Psi(X_{r_1},\dots,X_{r_{m}}) \right] \prod_{\nu=1}^k \D_{r_{m-j_\nu+1}}X_{r_{l_\nu}}  \\
& -\E\left[\Psi(\xcn_{\eta(r_1)},\dots,\xcn_{\eta(r_{m})}) \right] \prod_{\nu=1}^k \D_{r_{m-j_\nu+1}} \xcn_{\eta(r_{l_\nu})} \ dr_{m}\dots dr_1,
\end{align*}
when $\Psi(x_1,\dots,x_m)=\prod_{i=1}^mg_i(x_i)$ for some functions $g_1,\dots,g_m \in \mathcal{C}_{\exp}^\infty(\R, \R)$. In particular, $\Psi$ satisfies
\begin{equation*}%\label{test_func}
 \forall k_1,\dots,k_m \in \N, \exists C\in \R_+ \forall x \in \R^m,\ |\partial_1^{k_1}\dots \partial_{m}^{k_m}\Psi(x)|\le Ce^{C\sum_{i=1}^m|x_i|}.
\end{equation*}

As in Section~\ref{Sec_cubic} for the cubic case, we split $\mathcal{E}$ as follows
$$\mathcal{E}=\mathcal{E}_{1,1}+\mathcal{E}_{1,2}+\mathcal{E}_{2,1}+\mathcal{E}_{2,2},$$
with {\small
\begin{align*}
  \mathcal{E}_{1,1}&=\int_0^T \dots \int_0^{r_{m-1}} \E\left[\Psi(X_{r_1},\dots,X_{r_{m}}) \right] \left( \prod_{\nu=1}^k \D_{r_{m-j_\nu+1}}X_{r_{l_\nu}}- \prod_{\nu=1}^k \D_{r_{m-j_\nu+1}}X_{\eta(r_{l_\nu})}  \right) dr_{m}\dots dr_1,\\
    \mathcal{E}_{1,2}&=\int_0^T \dots \int_0^{r_{m-1}} \E\left[\Psi(X_{r_1},\dots,X_{r_{m}})-\Psi(X_{\eta(r_1)},\dots,X_{\eta(r_{m})}) \right]  \prod_{\nu=1}^k \D_{r_{m-j_\nu+1}}X_{\eta(r_{l_\nu})}   dr_{m}\dots dr_1,\\
    \mathcal{E}_{2,1}&=\int_0^T \dots \int_0^{r_{m-1}} \E\left[\Psi(X_{\eta(r_1)},\dots,X_{\eta(r_{m})})-\Psi(\xcn_{\eta(r_1)},\dots,\xcn_{\eta(r_{m})}) \right] \prod_{\nu=1}^k \D_{r_{m-j_\nu+1}}\xcn_{\eta(r_{l_\nu})}     dr_{m}\dots dr_1,\\
    \mathcal{E}_{2,2}&=\int_0^T \dots \int_0^{r_{m-1}} \E\left[\Psi(X_{\eta(r_1)},\dots,X_{\eta(r_{m})}) \right] \left( \prod_{\nu=1}^k \D_{r_{m-j_\nu+1}}X_{\eta(r_{l_\nu})} - \prod_{\nu=1}^k \D_{r_{m-j_\nu+1}}\xcn_{\eta(r_{l_\nu})}  \right)   dr_{m}\dots dr_1.\end{align*}
}

The two first terms do not involve the discretization process $\xcn$ and are essentially already analysed in~\cite{FSW}. Namely, $\mathcal{E}_{1,2}$ corresponds to~\cite[Equation~(25)]{FSW} and $\mathcal{E}_{1,1}$ corresponds to a sum of terms~\cite[Equation (26)]{FSW}, with the Malliavin derivatives replacing the kernels\footnote{Note that  one should of course either read in~\cite[Equation (26)]{FSW}  $\widehat{W}_{\bf t}$ instead of $\widehat{W}_{\eta(\bf t)}$ or in~\cite[Equation (25)]{FSW} $K$ instead of $\tilde{K}$, so that the sum of~\cite[Equation (25) and (26)]{FSW} gives the weak error~\cite[Equation (24)]{FSW}.}. Thus, we will be brief on the analysis of these two terms and rather focus on the two other terms which are specific to our new setting and arise from the discretization of the dynamics of the rough process~$X$. For that purpose, we provide additional technical analysis, see in particular Proposition~\ref{prop_E22}.

\subsubsection{Analysis of $\mathcal{E}_{1,1}$ and $\mathcal{E}_{1,2}$}
Following~\cite{FSW}, we rewrite $$ \prod_{\nu=1}^k \D_{r_{m-j_\nu+1}}X_{r_{l_\nu}}= \prod_{i=2}^m  \D_{r_{i}}X_{r_{\alpha(i)}},$$ with $\alpha:\{1,\dots m\} \to \mathbb{N}$ such that $\alpha(i)=l_\nu$ if $i=m-j_\nu+1$ for some $\nu \in \{1,\dots,k\}$, and  $\alpha(i)=0$ otherwise. We use as in~\cite{FSW} the convention that $\D_{r_{i}}X_{r_{\alpha(i)}}=1$ if $\alpha(i)=0$. Note that $\alpha(i)< i$, so that $r_i< r_{\alpha(i)}$. Note also that $\alpha(1)=0$ since $j_k<m$ (we recall that words~$w$ finishing by the letter $I$ give $\iota(w)I=0$, and that $j_1,\dots,j_k$ are the positions of the letter~$I$ in the word $w$), which explains why the above product starts from~$2$.  This new expression is used to have all the indices $i\in\{1,\dots,m\}$ which makes the following error decomposition clearer. Namely, using that $\prod_{i=2}^m a_i -\prod_{i=2}^m b_i= \sum_{j=2}^m (\prod_{i<j} a_i) (a_j-b_j)(\prod_{i>j} b_i)$, we get
\begin{align*}
  \mathcal{E}_{1,1}=&\sum_{j=2}^m  \int_0^T \dots \int_0^{r_{m-1}} \E\left[\Psi(X_{r_1},\dots,X_{r_{m}}) \right] \prod_{i<j} \D_{r_{i}}X_{r_{\alpha(i)}} \\
  & \times \left( \D_{r_{j}}X_{r_{\alpha(j)}} -\D_{r_{j}}X_{\eta(r_{\alpha(j)})}  \right)\prod_{i>j} \D_{r_{i}}X_{\eta(r_{\alpha(i)})} dr_{m}\dots dr_1 \\
  =:&\sum_{j=2}^m \mathcal{E}_{1,1,j}.
\end{align*}
The term $\mathcal{E}_{1,1,j}$ corresponds in our framework to~\cite[Equation (26)]{FSW}, with the Malliavin derivatives replacing the kernels.
From~\eqref{der_Malliavin_X}, the Malliavin derivative is a series whose the first term is precisely the fractional kernel considered in~\cite{FSW}.  Following their analysis (namely \cite[Lemmas~4.4 and 4.7]{FSW}) by distinguishing the first term in~\eqref{der_Malliavin_X} and the other ones as in the analysis of $\mathcal{E}_{1,1}$ in Section~\ref{Sec_cubic},  we get that $|\mathcal{E}_{1,1,j}|\le C \vn $ and then $|\mathcal{E}_{1,1}|\le C \vn$.
In the same way, adapting~\cite[Lemma 4.3]{FSW} to our framework allows to get $|\mathcal{E}_{1,2}|\le C \vn$.

%\subsubsection{Analysis of $\mathcal{E}_{1,2}$}

\subsubsection{Analysis of $\mathcal{E}_{2,1}$}

We use the same notation as for $\mathcal{E}_{1,1}$ and write
\begin{align*}
   \mathcal{E}_{2,1}&=\int_0^T \dots \int_0^{r_{m-1}} \E\left[\Psi(X_{\eta(r_1)},\dots,X_{\eta(r_{m})})-\Psi(\xcn_{\eta(r_1)},\dots,\xcn_{\eta(r_{m})}) \right] \prod_{i=2}^m \D_{r_{i}}\xcn_{\eta(r_{\alpha(i)})}     dr_{m}\dots dr_1.
\end{align*}
By Corollary~\ref{cor_weak_ito}, we get 
\begin{align*}
   |\mathcal{E}_{2,1}|&\le C \int_0^T \dots \int_0^{r_{m-1}} \bigg(  \sum_{1\le k\le l \le m} |Cov(X_{\eta(r_k)},X_{\eta(r_l)})- Cov(\xcn_{\eta(r_k)},\xcn_{\eta(r_l)}) | \\&+\sum_{l=1}^m |\E[X_{\eta(r_l)}]-\E[\xcn_{\eta(r_l)}]| \bigg) \prod_{i=2}^m | \D_{r_{i}}\xcn_{\eta(r_{\alpha(i)})}|     dr_{m}\dots dr_1.
\end{align*}
By Theorem~\ref{thm_covariance_error}, and noting that the term in parenthesis vanishes when $r_l<t_1$, we get 
\begin{align*}
   |\mathcal{E}_{2,1}|&\le C \sum_{l=1}^m  \int_0^T \dots \int_0^{r_{m-1}} \mathbf{1}_{r_l>t_1}  \bigg( (1+\eta(r_l)^{\alpha-1}){\bf v}_n(\alpha) + \frac{\eta(r_l)^{2\alpha-2}}{n}  \bigg) \prod_{i=2}^m | \D_{r_{i}}\xcn_{\eta(r_{\alpha(i)})}|     dr_{m}\dots dr_1,
\end{align*}
since $r_l\le r_k$. By Lemma~\ref{lem_err_Malliavin3}, we get 
\begin{align*}
   |\mathcal{E}_{2,1}|&\le  C \sum_{l=1}^m \int_0^T \dots \int_0^{r_{l-1}} \mathbf{1}_{r_l>t_1} \bigg( (1+\eta(r_l)^{\alpha-1}){\bf v}_n(\alpha) + \frac{\eta(r_l)^{2\alpha-2}}{n}  \bigg) \prod_{i=2}^l | \D_{r_{i}}\xcn_{\eta(r_{\alpha(i)})}|     dr_{l}\dots dr_1,
\end{align*}
by integrating with respect to $r_m,\dots,r_{l+1}$ (as usual, the constant $C$ changes from line to line). Now, we proceed as in the cubic case: by~\eqref{majo_der_malliavin_sch}, we get 
$$ \int_{t_1}^{r_{l-1}} \eta(r_l)^\beta |\D_{r_l} \xcn_{\eta(r_{\alpha(l)})}|dr_l\le C \left(\mathbf{1}_{\alpha+\beta\ge 0} T^{\alpha+\beta}+\mathbf{1}_{\alpha+\beta< 0} t_1^{\alpha+\beta} \right),$$
for $\beta \in (-1,0]$. This gives with $\beta\in \{\alpha-1,2\alpha-2\}$, as in the cubic case, 
$$ \int_0^{r_{l-1}} \mathbf{1}_{r_l>t_1} C \bigg( (1+\eta(r_l)^{\alpha-1}){\bf v}_n(\alpha) + \frac{\eta(r_l)^{2\alpha-2}}{n}  \bigg) | \D_{r_{l}}\xcn_{\eta(r_{\alpha(l)})}|     dr_{l}\le C \vn .$$
Finally, we use again Lemma~\ref{lem_err_Malliavin3} to get $|\mathcal{E}_{2,1}|\le C \vn$.

\subsubsection{Analysis of $\mathcal{E}_{2,2}$}
We use the usual decomposition for the difference of the products
\begin{equation}\label{product_expansion}
  \prod_{i \in I } a_i - \prod_{i \in I } b_i  =\sum_{j \in I} \left(\prod_{i \in I, i<j} b_i \right)(a_j-b_j) \left(\prod_{i \in I, i>j} a_i\right)
\end{equation}
 to get with $I=\{2\dots,m\}$
\begin{align*}
 &\prod_{i=2}^m \D_{r_{i}}X_{\eta(r_{\alpha(i)})} - \prod_{i=2}^m \D_{r_{i}}\xcn_{\eta(r_{\alpha(i)})}\\&=\sum_{j=2}^m \prod_{2\le i<j} \D_{r_{i}}\xcn_{\eta(r_{\alpha(i)})} 
   \times \left( \D_{r_{j}}X_{\eta(r_{\alpha(j)})} -\D_{r_{j}}\xcn_{\eta(r_{\alpha(j)})}  \right)\prod_{m\ge i>j} \D_{r_{i}}X_{\eta(r_{\alpha(i)})}.
\end{align*}
Here, we use the standard convention $\prod_{\emptyset}=1$.
Note that in~\eqref{product_expansion} $a$ and $b$ play symmetric role up to the minus sign, but it is important to keep the Malliavin derivatives of the original process for $i>j$ in the next analysis to be able to show the integration by parts~\eqref{IPP}. 
We then write 
\begin{align}
  \mathcal{E}_{2,2}&=\sum_{j=2}^m  \int_0^T \dots \int_0^{r_{m-1}} \E\left[\Psi(X_{r_1},\dots,X_{r_{m}}) \right] \prod_{i<j} \D_{r_{i}}\xcn_{\eta(r_{\alpha(i)})} \label{decompo_E22}  \\
    &\times \left( \D_{r_{j}}X_{\eta(r_{\alpha(j)})} -\D_{r_{j}}\xcn_{\eta(r_{\alpha(j)})}  \right)\prod_{i>j} \D_{r_{i}}X_{r_{\alpha(i)}} 
dr_{m}\dots dr_1  + \mathcal{R}^1_{2,2}+ \mathcal{R}^2_{2,2} \text{ with } \notag\\
\mathcal{R}^1_{2,2} &= \sum_{j=2}^m  \int_0^T \dots \int_0^{r_{m-1}} \E\left[\Psi(X_{r_1},\dots,X_{r_{m}}) \right] \prod_{i<j} \D_{r_{i}}\xcn_{\eta(r_{\alpha(i)})}  \notag\\
    &\times \left( \D_{r_{j}}X_{\eta(r_{\alpha(j)})} -\D_{r_{j}}\xcn_{\eta(r_{\alpha(j)})}  \right) \left( \prod_{i>j} \D_{r_{i}}X_{\eta(r_{\alpha(i)})} -\prod_{i>j} \D_{r_{i}}X_{r_{\alpha(i)}} \right) 
dr_{m}\dots dr_1 \notag\\
\mathcal{R}^2_{2,2} &= \sum_{j=2}^m  \int_0^T \dots \int_0^{r_{m-1}} \left( \E\left[\Psi(X_{\eta(r_1)},\dots,X_{\eta(r_{m})}) \right] - \E\left[\Psi(X_{r_1},\dots,X_{r_{m}}) \right] \right)\prod_{i<j} \D_{r_{i}}\xcn_{\eta(r_{\alpha(i)})} \notag \\
 &   \times \left( \D_{r_{j}}X_{\eta(r_{\alpha(j)})} -\D_{r_{j}}\xcn_{\eta(r_{\alpha(j)})}  \right)\prod_{i>j} \D_{r_{i}}X_{\eta(r_{\alpha(i)})} 
dr_{m}\dots dr_1. \notag
\end{align}
Note that $\mathcal{R}^2_{2,2}$ ``replaces'' $\E\left[\Psi(X_{\eta(r_1)},\dots,X_{\eta(r_{m})}) \right]$ by $\E\left[\Psi(X_{r_1},\dots,X_{r_{m}}) \right] $ and $\mathcal{R}^1_{2,2}$ ``replaces'' $\prod_{i>j} \D_{r_{i}}X_{\eta(r_{\alpha(i)})}$ by $\prod_{i>j} \D_{r_{i}}X_{r_{\alpha(i)}} $.

We first analyse $\mathcal{R}^1_{2,2}$. To do so, we use an upper bound of $|\E\left[\Psi(X_{r_1},\dots,X_{r_{m}}) \right]|$ and the triangular inequality, so that we have to upper bound the integrals
\begin{align*} \int_0^T \dots \int_0^{r_{m-1}} \prod_{i<j} |\D_{r_{i}}\xcn_{\eta(r_{\alpha(i)})}| & \left| \D_{r_{j}}X_{\eta(r_{\alpha(j)})} -\D_{r_{j}}\xcn_{\eta(r_{\alpha(j)})}  \right| \\& \times \left| \prod_{i>j} \D_{r_{i}}X_{\eta(r_{\alpha(i)})} -\prod_{i>j} \D_{r_{i}}X_{r_{\alpha(i)}} \right| 
dr_{m}\dots dr_1 .\end{align*}
We use again~\eqref{product_expansion} and get
$$ \prod_{i>j} \D_{r_{i}}X_{\eta(r_{\alpha(i)})} -\prod_{i>j} \D_{r_{i}}X_{r_{\alpha(i)}} =\sum _{i>j}  \prod_{j<k<i} \D_{r_{k}}X_{r_{\alpha(k)}} \left(\D_{r_{i}}X_{\eta(r_{\alpha(i)})}  -\D_{r_{i}}X_{r_{\alpha(i)}} \right) \prod_{k>i}\D_{r_{k}}X_{\eta(r_{\alpha(k)})}.$$
Let us consider the term~$i$ of this sum, and integrate first  from the $dr_m$ to $dr_i$. We have
\begin{align*}
  &\int_0^{r_{i-1}} \dots \int_0^{r_{m-1}}  \left|\D_{r_{i}}X_{\eta(r_{\alpha(i)})}  -\D_{r_{i}}X_{r_{\alpha(i)}} \right| \prod_{k>i}|\D_{r_{k}}X_{\eta(r_{\alpha(k)})} | dr_m \dots dr_i \\& \le  C_T^{m-i-1}\int_0^{r_{i-1}}\left|\D_{r_{i}}X_{\eta(r_{\alpha(i)})}  -\D_{r_{i}}X_{r_{\alpha(i)}} \right| dr_i,
\end{align*} 
by using~\eqref{eq_def_majCT}. Now, using~\eqref{bound_integ}, we obtain that this integral is $O(n^{-\alpha})$. Integrating then from $dr_i$ to $dr_{j+1}$, we get
$$\mathcal{R}^1_{2,2}\le C n^{-\alpha} \sum_{j=2}^m  \int_0^T \dots \int_0^{r_{j-1}}  \prod_{i<j} |\D_{r_{i}}\xcn_{\eta(r_{\alpha(i)})}| \left| \D_{r_{j}}X_{\eta(r_{\alpha(j)})} -\D_{r_{j}}\xcn_{\eta(r_{\alpha(j)})}  \right|dr_j\dots dr_1.  $$
We now use Lemma~\ref{lem_err_Malliavin_int} for the integration with respect to $dr_j$ and then Lemma~\ref{lem_err_Malliavin3} for the integration from $dr_{j-1}$ to $dr_1$, and we get $\mathcal{R}^1_{2,2}\le C n^{-2\alpha} =O(\vn)$ since $2\alpha>1$.

We now analyse $\mathcal{R}^2_{2,2}$. By Proposition~\ref{prop_phi}~(1), we have 
$$\left| \E\left[\Psi(X_{\eta(r_1)},\dots,X_{\eta(r_{m})}) \right] - \E\left[\Psi(X_{r_1},\dots,X_{r_{m}}) \right] \right|\le C \left(\frac{T}n\right) ^{2\alpha-1}.$$
Then, we integrate from $dr_m$ to $dr_1$ as follows: we use~\eqref{eq_def_majCT} for $dr_{i}$ with $i>j$, Lemma~\ref{lem_err_Malliavin_int} for the integral with respect to $dr_{j}$ and finally Lemma~\ref{lem_err_Malliavin3} for $dr_{i}$ with $i<j$. This leads to 
$$|\mathcal{R}^2_{2,2}|\le  C \left(\frac{T}n\right) ^{3\alpha-1} =O(\vn).$$ 

We then go back to Equation~\eqref{decompo_E22} and  define
\begin{align*}
  &\mathcal{E}_{2,2,j}= \int_0^T \dots \int_0^{r_{j-1}} G_j(r_1,\dots,r_j) \prod_{i<j} \D_{r_{i}}\xcn_{\eta(r_{\alpha(i)})}  \left( \D_{r_{j}}X_{\eta(r_{\alpha(j)})} -\D_{r_{j}}\xcn_{\eta(r_{\alpha(j)})}  \right)
dr_{j}\dots dr_1 , \\
&G_j(r_1,\dots,r_j)=\int_0^{r_j}\int_0^{r_{j+1}}\dots\int_0^{r_{m-1}}  \E\left[\Psi({ X}_{{\bf r}})\right] \Pi_j({\bf r})dr_{j+1}\dots dr_m,
\end{align*}
with ${\bf r}=(r_1,\dots,r_m)$,  $$\Pi_j({\bf r}):=\prod_{i>j} \D_{r_{i}}X_{r_{\alpha(i)}} $$ and $X_{{\bf r}}=(X_{r_1},\dots,X_{r_m})$. Note that in the product, we distinguish the indices strictly greater than $j$ that are frozen on the time grid, while the indices smaller or equal to~$j$ are not frozen, so that the function $G_j$ is continuous with respect to $(r_1,\dots,r_j)$. The analysis of this term mainly relies on the following proposition, whose proof is quite technical and is postponed in the appendix. 
\begin{prop}\label{prop_E22}
  Let  $1\le j\le m$, $\ell_{j+1},\dots,\ell_{m}\ge 1$. Then the function
\begin{align*}
  H_j(r_1,\dots,r_j)=\int_0^{r_j}\int_0^{r_{j+1}}\dots\int_0^{r_{m-1}}  \phi({\bf r})\prod_{i>j} (r_{\alpha(i)}-r_i)^{\ell_i \alpha-1} dr_m\dots dr_{j+1}
\end{align*}
with ${\bf r}=(r_1,\dots,r_m)$ and $\phi({\bf r})=\E[\Psi(X_{\bf r})]$ is differentiable with respect to $r_{j}$ and 
$$\int_0^{r_{j-1}}|\partial_j H (r_1,\dots,r_j) | dr_j \le C_{m,\phi} T^{ 2\alpha \sum_{i=j+1}^m  \ell_i}. $$
\end{prop}

 We now make an integration by part with respect to $r_j$ that we justify afterwards:
 \begin{align}
 &\int_0^{r_{j-1}} G_j(r_1,\dots,r_j) \left( \D_{r_{j}}X_{\eta(r_{\alpha(j)})} -\D_{r_{j}}\xcn_{\eta(r_{\alpha(j)})}  \right) dr_j \notag\\
 &= G_j(r_1,\dots,r_{j-1},r_{j-1}) \int_0^{r_{j-1}} \left( \D_{s}X_{\eta(r_{\alpha(j)})} -\D_{s}\xcn_{\eta(r_{\alpha(j)})}  \right) ds  \label{IPP} \\
 &\ \ -\int_0^{r_{j-1}} \partial_{j} G_j(r_1,\dots,r_j)  \int_0^{r_j} \left( \D_{s}X_{\eta(r_{\alpha(j)})} -\D_{s}\xcn_{\eta(r_{\alpha(j)})}  \right) ds dr_j. \notag
 \end{align}
In fact, we can write by~\eqref{der_Malliavin_X} $G_j$ as an infinite sum of terms that can be written as in Proposition~\ref{prop_E22}, and using this Proposition, we have $\int_0^{r_{j-1}} |\partial_{j} G_j(r_1,\dots,r_j) |dr_j \le C$ for a constant $C$ that depends on $\phi$ but does not depend on $n$. Of course, this justifies the integration by parts.  We observe also that $G_j$ is bounded. Therefore, we get from~\eqref{IPP} and Lemma~\ref{lem_err_Malliavin3}
 \begin{align*}
 &\left|\int_0^{r_{j-1}} G_j(r_1,\dots,r_j) \left( \D_{r_{j}}X_{\eta(r_{\alpha(j)})} -\D_{r_{j}}\xcn_{\eta(r_{\alpha(j)})}  \right) dr_j \right| \le C \frac{T}{n}.
 \end{align*}
 We conclude that $|\mathcal{E}_{2,2,j}|\le   C \frac{T}{n}$ by using Lemma~\ref{lem_err_Malliavin3} for the integration from $dr_{j-1}$ to $dr_1$. This leads to $|\mathcal{E}_{2,2}|\le   C \vn$. This concludes the proof of Theorem~\ref{thm_main}

\appendix 
\section{Technical results}

We first recall that there exists $C_\alpha>0$ such that
\begin{equation}\label{eq_LiZeng}
  (i+1)^\alpha - i^\alpha \le C_\alpha i^{\alpha-1}, \ (i+1)^\alpha - i^\alpha \le C_\alpha (i+1)^{\alpha-1}, \ i\in \N,
\end{equation}
and refer to~\cite[Lemma 3.1]{LiZe} for a proof (note that $1\le \frac{i+1}{i} \le 2$ for $i\ge 2$, so that the first inequality implies the second one, and the second one is trivially satisfied for $i=0$). We present a useful lemma, which is an adaptation to our context of~\cite[Lemma 3.2]{LiZe}.

\begin{lemma}\label{lemma_beta_disc2}
  Let $\alpha,\beta \in (0,1)$. Then, we have for all $k\ge 1$
$$\sum_{i=1}^k (k-(i-1))^{\beta-1}i^{\alpha-1}\le   \frac{\Gamma(\alpha)\Gamma(\beta)}{\Gamma(\alpha+\beta)}  k^{\alpha + \beta-1}.$$
\end{lemma}
\begin{proof}
We write
$$\sum_{i=1}^k (k-(i-1))^{\beta-1}i^{\alpha-1}=k^{\alpha+\beta-1} \frac 1 k \sum_{i=1}^k \left(1-\frac{i-1}{k} \right)^{\beta-1}\left(\frac{i}{k}\right)^{\alpha-1}, $$ 
and observe that $\int_{\frac{i-1}{k}}^{\frac{i}{k}} (1-x)^{\beta-1} x^{\alpha-1}dx \ge \frac 1 k \left(1-\frac{i-1}{k} \right)^{\beta-1}\left(\frac{i}{k}\right)^{\alpha-1} $ since $x\mapsto x^{\alpha-1}$ and $x\mapsto x^{\beta-1}$ are positive decreasing functions. Therefore,
$$ \sum_{i=1}^k (k-(i-1))^{\beta-1}i^{\alpha-1}\le k^{\alpha+\beta-1} \int_0^1 (1-x)^{\beta-1} x^{\alpha-1}dx =\frac{\Gamma(\alpha)\Gamma(\beta)}{\Gamma(\alpha+\beta)}  k^{\alpha + \beta-1}. \qedhere$$
\end{proof}

\begin{lemma}\label{lemma_beta_disc}
  Let $T>0$, $n\ge 1$. Let $t_k=kT/n$ for $k\in \{0,\dots, n\}$. Let $\alpha,\beta \in (0,1)$.
Then, there exists a constant~$C_{\alpha,\beta}$ that only depends on $\alpha$ and $\beta$ such that for any $k\le n$, 
\begin{equation*}
  \sum_{i=1}^k \left((t_{k}-t_{i-1})^\beta - (t_{k}-t_{i})^\beta \right)\left( t_i^\alpha-t_{i-1}^\alpha \right)\le C_{\alpha,\beta} \left( \frac T n \right)^{\alpha+\beta} k^{\alpha+\beta-1}.
\end{equation*}
\end{lemma}
\begin{proof}
We have by using~\eqref{eq_LiZeng}
\begin{align*}
  &\sum_{i=1}^k \left((t_{k}-t_{i-1})^\beta - (t_{k}-t_{i})^\beta \right)\left( t_i^\alpha-t_{i-1}^\alpha \right) \\
  =&\left( \frac T n \right)^{\alpha+\beta} \sum_{i=1}^k \left((k-(i-1))^\beta - (k-i)^\beta \right)\left( i^\alpha-(i-1)^\alpha \right) \\
  \le& C_\alpha C_\beta \left( \frac T n \right)^{\alpha+\beta} \sum_{i=1}^k (k-(i-1))^{\beta-1}i^{\alpha-1}\le C_\alpha C_\beta \frac{\Gamma(\alpha)\Gamma(\beta)}{\Gamma(\alpha+\beta)} \left( \frac T n \right)^{\alpha+\beta} k^{\alpha+\beta-1},
\end{align*}
by Lemma~\ref{lemma_beta_disc2}.
\end{proof}

We now present a lemma that is the analogous of~\cite[Proposition 4.1]{Gassiat} in our framework.

\begin{lemma}\label{lem:covariance}
  Let $\alpha \in (1/2,1)$ and  $T>0$. There exists a constant $C$ depending on $T$, $\sigma$, $\kappa_2$ and $\alpha$ such that
\begin{equation}\label{cov_increment}
  \forall s,t \in [0,T], \ |\covX(s,t)-\covX(t,t)| \le C |t-s|^{2\alpha-1}.  
\end{equation}
  Moreover, the function $t\mapsto \covX(t,s)$ is differentiable on $(0,T)$ for $t\not= s$, and there exists a constant $C$ depending on $T$, $\sigma$, $\kappa_2$ and $\alpha$ such that
  \begin{equation}\label{majo_all_s}
    \forall s,t \in (0,T), s\not=t, \ |\partial_1 \covX(t,s)|\le C \left(s^{\alpha-1}t^{\alpha-1}+  \left|s-t \right|^{2 \alpha-2}\right).
  \end{equation}
\end{lemma}
Note that by symmetry we get from~\eqref{majo_all_s}
\begin{equation*}%\label{majo_all_t}
  \forall s,t \in (0,T), s\not=t, \ |\partial_2 \covX(t,s)|\le C \left(s^{\alpha-1}t^{\alpha-1}+  \left|s-t \right|^{2 \alpha-2}\right).
\end{equation*}  

\begin{proof}
The first part of the claim is an easy consequence of the second. Indeed, we have
\begin{align*}
  |\covX(s,t)-\covX(t,t)|&\le  \int_{(s,t)} |\partial_1 \covX(u,t)| du \\
  &\le C \int_{(s,t)} \left(u^{\alpha-1}t^{\alpha-1}+  \left|u-t \right|^{2 \alpha-2}\right) du \\
  &= C \left(\frac{t^{\alpha-1}}{\alpha}|s^\alpha-t^\alpha|+ \frac{|s-t|^{2\alpha-1}}{2\alpha-1}\right) \le C|s-t|^{2\alpha-1} ,
\end{align*}   
since $|s^\alpha-t^\alpha|\le |t-s|^\alpha \le  T^{1-\alpha} |t-s|^{2\alpha-1} $. This gives~\eqref{cov_increment}.

  From~\eqref{cov_formula}, we have
\begin{align*}
  \covX(t,s)&=\sigma^2 \sum_{i_1,i_2=1}^\infty \kappa_2^{i_1+i_2-2} \frac{t^{i_1\alpha} s^{i_2 \alpha-1}}{\Gamma(i_1 \alpha+1)\Gamma(i_2 \alpha)}  {_2F_1}(1-i_2 \alpha,1;i_1 \alpha+1,\frac{t}{s}), \ t\le s,\\
\covX(t,s)&=\sigma^2 \sum_{i_1,i_2=1}^\infty \kappa_2^{i_1+i_2-2} \frac{s^{i_1\alpha } t^{i_2\alpha-1}}{\Gamma(i_1 \alpha+1)\Gamma(i_2 \alpha)} {_2F_1}(1-i_2 \alpha,1;i_1 \alpha+1,\frac{s}{t}), \ s \le t.
\end{align*}
% Let us note that for $z \in [0,1]$, ${_2F_1}(1-i_2 \alpha,1;i_1 \alpha+1,z)=\frac{\Gamma(i_1 \alpha+1)}{\Gamma(i_1\alpha )} \int_0^1(1-t)^{i_1 \alpha-1}(1-zt)^{i_2\alpha-1} dt$ by~\cite[15.6.1 and 15.1.2]{Daalhuis}. If $i_2\ge 2$, we have $i_2\alpha-1\ge 1$ and thus  ${_2F_1}(1-i_2 \alpha,1;i_1 \alpha+1,z)\le
% \int_0^1 i_1 \alpha(1-t)^{i_1 \alpha-1} dt =1$. If $i_2=1$, then ${_2F_1}(1-i_2 \alpha,1;i_1 \alpha+1,z)\le
% \int_0^1 i_1 \alpha(1-t)^{(i_1+1) \alpha-2} dt =\frac{i_1\alpha}{(i_1+1) \alpha-1} \le \frac{\alpha}{2\alpha -1}$ for $i_1\ge 1$. Therefore,
% $$\sup_{i_1,i_2 \ge 1}{_2F_1}(1-i_2 \alpha,1;i_1 \alpha+1,z) \le \frac{\alpha}{2\alpha -1}.$$
We can differentiate and get for $0<t < s$ by using~\cite[15.5.1]{Daalhuis}:
\begin{align*}
\partial_1\covX(t,s)=&\sigma^2 \sum_{i_1,i_2=1}^\infty \kappa_2^{i_1+i_2-2}  \frac{t^{i_1\alpha-1} s^{i_2 \alpha-1}}{\Gamma(i_1 \alpha)\Gamma(i_2 \alpha)}  {_2F_1}(1-i_2 \alpha,1;i_1 \alpha+1,\frac{t}{s}) \\
  &+\sigma^2 \sum_{i_1,i_2=1}^\infty \kappa_2^{i_1+i_2-2} \frac{t^{i_1\alpha} s^{i_2 \alpha-2}}{\Gamma(i_1 \alpha+1)\Gamma(i_2 \alpha)}   \frac{1-i_2 \alpha}{i_1 \alpha +1} {_2F_1}(2-i_2 \alpha,2;i_1 \alpha+2,\frac{t}{s}).
\end{align*}
By~\eqref{def_2F1}, the maximum of $_2F_1(a,b;c,z)$ on $z\in[0,1]$ is reached by $z=0$ if $a\le 0$ and $z=1$ if $a>0$.  We get by~\cite[15.4.20]{Daalhuis}:
\begin{equation}\label{max_confluent}\max_{z\in [0,1]}\ _2F_1(a,b;c,z)=1_{a\le 0}+1_{a> 0} \frac{\Gamma(c)\Gamma(c-a-b)}{\Gamma(c-a)\Gamma(c-b)}.
\end{equation}  
Therefore, we have
\begin{align*}
  &\left|\sigma^2 \sum_{i_1,i_2=1}^\infty \kappa_2^{i_1+i_2-2} \frac{t^{i_1\alpha-1} s^{i_2 \alpha-1}}{\Gamma(i_1 \alpha)\Gamma(i_2 \alpha)}  {_2F_1}(1-i_2 \alpha,1;i_1 \alpha+1,\frac{t}{s})\right| \\
  & \le s^{\alpha-1}t^{\alpha-1} \sigma^2 \sum_{i_1,i_2=1}^\infty |\kappa_2|^{i_1+i_2-2} \frac{t^{(i_1-1)\alpha} s^{(i_2-1) \alpha}}{\Gamma(i_1 \alpha)\Gamma(i_2 \alpha)} \max_{z\in [0,1]}\ _2F_1(1-i_2 \alpha,1;i_1 \alpha+1,z)  \\
  &\le s^{\alpha-1}t^{\alpha-1}  \sigma^2 \sum_{i_1,i_2=1}^\infty |\kappa_2|^{i_1+i_2-2}  \frac{T^{(i_1+i_2-2) \alpha}}{\Gamma(i_1 \alpha)\Gamma(i_2 \alpha)} \max_{z\in [0,1]}\ _2F_1(1-i_2 \alpha,1;i_1 \alpha+1,z) \\
  &=: C^1_T s^{\alpha-1}t^{\alpha-1}.
\end{align*}
For the second series, we first observe that $_2F_1(2-i_2 \alpha,2;i_1 \alpha+2,1)<\infty$ if $(i_1+i_2) \alpha-2>0$ (which holds if $i_1+i_2\ge 4$) by~\eqref{max_confluent} and  have 
\begin{align*}
  &\left|\sigma^2 \sum_{i_1+i_2\ge 4}^\infty \kappa_2^{i_1+i_2-2} \frac{t^{i_1\alpha} s^{i_2 \alpha-2}}{\Gamma(i_1 \alpha+1)\Gamma(i_2 \alpha)}   \frac{1-i_2 \alpha}{i_1 \alpha +1} {_2F_1}(2-i_2 \alpha,2;i_1 \alpha+2,\frac{t}{s})\right|
  \\
  &\le \sigma^2 s^{\alpha-1}t^{\alpha-1} \sum_{i_1+i_2\ge 4}^\infty |\kappa_2|^{i_1+i_2-2} \frac{t^{(i_1-1)\alpha+1} s^{(i_2-1) \alpha-1}}{\Gamma(i_1 \alpha+1)\Gamma(i_2 \alpha)}   \frac{|1-i_2 \alpha|}{i_1 \alpha +1} {_2F_1}(2-i_2 \alpha,2;i_1 \alpha+2,\frac{t}{s})\\
  &\le \sigma^2 s^{\alpha-1}t^{\alpha-1} \sum_{i_1+i_2\ge 4}^\infty |\kappa_2|^{i_1+i_2-2} \frac{ T^{(i_1+i_2-2)\alpha}}{\Gamma(i_1 \alpha+1)\Gamma(i_2 \alpha)}  \frac{|1-i_2 \alpha|}{i_1 \alpha +1} \max_{z\in [0,1]}\  {_2F_1}(2-i_2 \alpha,2;i_1 \alpha+2,z)\\
  &=:C^2_T s^{\alpha-1}t^{\alpha-1},
\end{align*}
 using $t\le s\le T$ for the last inequality.

We now focus on the three other terms of the second series, i.e. $i_1=i_2=1$; $i_1=1$ and $i_2=2$; $i_2=1$ and $i_1=2$. 
To analyse these terms, we use \cite[Formula 15.8.1]{Daalhuis}
$_2F_1(2-i_2 \alpha,2;i_1 \alpha+2,\frac{t}{s})=(1-\frac{t}{s})^{(i_1+i_2) \alpha-2}\ _2F_1((i_1+i_2) \alpha,i_1 \alpha;i_1 \alpha+2,\frac{t}{s})$. 

\noindent $\bullet$ The term obtained with $i_1=i_2=1$ is:
$$\frac{\sigma^2 (1-\alpha)}{1+\alpha} \frac{t^{\alpha} s^{\alpha-2}}{\Gamma(1+\alpha)\Gamma(\alpha)} \left(1-\frac{t}{s}\right)^{2 \alpha-2} \ _2F_1(2 \alpha, \alpha;\alpha+2,\frac{t}{s}).$$
$\bullet$ The term obtained with $i_1=1$, $i_2=2$ is:
$$\frac{\sigma^2 \kappa_2 (1-2\alpha)}{1+\alpha} \frac{t^\alpha s^{2\alpha-2}}{\Gamma(1+\alpha)\Gamma(2\alpha)} \left(1-\frac{t}{s}\right)^{3 \alpha-2} \ _2F_1(3 \alpha, \alpha;\alpha+2,\frac{t}{s}).$$
$\bullet$ The term obtained with $i_1=2$, $i_2=1$ is:
$$\frac{\sigma^2 \kappa_2(1-\alpha)}{1+2\alpha} \frac{t^{2\alpha}s^{\alpha-2}}{\Gamma(1+2\alpha)\Gamma(\alpha)} \left(1-\frac{t}{s}\right)^{3 \alpha-2} \ _2F_1(3 \alpha, 2\alpha;2\alpha+2,\frac{t}{s}).$$
By~\eqref{max_confluent}, the sum of these three terms can be upper bounded for all $0< t< s \le T$, up to a multiplicative constant (depending on $\alpha$, $\sigma$ and $\kappa_2$), by 
$$ (s-t)^{2\alpha-2}+(s-t)^{3\alpha-2}\le (1+T^\alpha)(s-t)^{2\alpha-2}.$$
Finally, there exists a constant $C_T$ such that
\begin{equation}\label{majo_s_smaller}
  \forall 0<t< s \le T, \  |\partial_1\covX(t,s)|\le C_T \left(s^{2\alpha-2} +\left(s-t \right)^{2 \alpha-2}\right).
\end{equation}

Now, we turn to the case $0<s<t\le T$   and we get using~\cite[15.5.1]{Daalhuis}
\begin{align*}
  \partial_1\covX(t,s)=&\sigma^2 \sum_{i_1,i_2=1}^\infty \kappa_2^{i_1+i_2-2} (i_2 \alpha-1) \frac{s^{i_1\alpha} {t}^{i_2 \alpha -2}}{\Gamma(i_1 \alpha+1)\Gamma(i_2 \alpha)}  {_2F_1}(1-i_2 \alpha,1;i_1 \alpha+1,\frac{s}{t}) \\
  &+\sigma^2 \sum_{i_1,i_2=1}^\infty \kappa_2^{i_1+i_2-2}  \frac{i_2 \alpha-1}{i_1 \alpha +1} \frac{s^{i_1\alpha+1} {t}^{i_2 \alpha  -3}}{\Gamma(i_1 \alpha+1)\Gamma(i_2 \alpha)}  {_2F_1}(2-i_2 \alpha,2;i_1 \alpha+2,\frac{s}{t}).
\end{align*}
For the first series, we have
\begin{align*}
&\left|\sigma^2 \sum_{i_1,i_2=1}^\infty \kappa_2^{i_1+i_2-2} (i_2\alpha-1) \frac{s^{i_1\alpha} {t}^{i_2 \alpha -2 }}{\Gamma(i_1 \alpha+1)\Gamma(i_2 \alpha)}  {_2F_1}(1-i_2 \alpha,1;i_1 \alpha+1,\frac{s}{t})\right|\\&\le 
\sigma^2 \sum_{i_1,i_2=1}^\infty |\kappa_2|^{i_1+i_2-2} |i_2\alpha-1| \frac{s^{i_1\alpha} {t}^{i_2 \alpha -2 }}{\Gamma(i_1 \alpha+1)\Gamma(i_2 \alpha)} \max_{z\in [0,1]}\ {_2F_1}(1-i_2 \alpha,1;i_1 \alpha+1,z)\\&
\le 
\sigma^2 s^{\alpha-1}t^{\alpha-1} \sum_{i_1,i_2=1}^\infty |\kappa_2|^{i_1+i_2-2} |i_2\alpha-1| \frac{T^{(i_1+i_2-2)\alpha}}{\Gamma(i_1 \alpha+1)\Gamma(i_2 \alpha)} \max_{z\in [0,1]}\ {_2F_1}(1-i_2 \alpha,1;i_1 \alpha+1,z) \\&
=:\tilde C^1_T s^{\alpha-1}t^{\alpha-1}.
\end{align*}
For the second series,  using the same arguments as for the first part of the proof, we have
\begin{align*}
 & \sigma^2\left|\sum_{i_1+i_2\ge 4}^\infty 
 \kappa_2^{i_1+i_2-2}  \frac{i_2 \alpha-1}{i_1 \alpha +1} \frac{s^{i_1\alpha+1} {t}^{i_2 \alpha  -3}}{\Gamma(i_1 \alpha+1)\Gamma(i_2 \alpha)}  {_2F_1}(2-i_2 \alpha,2;i_1 \alpha+2,\frac{s}{t})
\right| \\
&\le 
\sigma^2 s^{\alpha-1}t^{\alpha-1}  \sum_{i_1+i_2\ge 4}^\infty |\kappa_2|^{i_1+i_2-2}  \frac{|i_2 \alpha-1|}{i_1 \alpha +1} \frac{s^{(i_1-1)\alpha+2} t^{(i_2-1) \alpha  -2}}{\Gamma(i_1 \alpha+1)\Gamma(i_2 \alpha)}  \max_{z\in [0,1]}\  {_2F_1}(2-i_2 \alpha,2;i_1 \alpha+2,z)\\
&\le 
\sigma^2  s^{\alpha-1}t^{\alpha-1} \sum_{i_1+i_2\ge 4}^\infty \kappa_2^{i_1+i_2-2}  \frac{|i_2 \alpha-1|}{i_1 \alpha +1} \frac{T^{(i_1+i_2-2)\alpha} }{\Gamma(i_1 \alpha+1)\Gamma(i_2 \alpha)}  \max_{z\in [0,1]}\  {_2F_1}(2-i_2 \alpha,2;i_1 \alpha+2,z)\\
  &=:\tilde C^2_T s^{\alpha-1}t^{\alpha-1} .
\end{align*}
Now the remaining terms are given by~\cite[Formula 15.8.1]{Daalhuis}\\
$\bullet$ for $i_1=i_2=1$ is 
$$
\sigma^2   \frac{ \alpha-1}{\alpha +1} \frac{s^{\alpha+1} {t}^{\alpha-3}}{\Gamma(\alpha+1)\Gamma(\alpha)} \left(1-\frac s{t}\right)^{2\alpha-2} {_2F_1}(2\alpha,\alpha; \alpha+2,\frac{s}{t})
$$
$\bullet$ for $i_1=1$ $i_2=2$ is
$$
\kappa_2  \frac{2 \alpha-1}{ \alpha +1} \frac{s^{\alpha+1} {t}^{2\alpha-3}}{\Gamma( \alpha+1)\Gamma(2\alpha)}  \left(1-\frac s{t}\right)^{3\alpha-2} {_2F_1}(3\alpha,\alpha; \alpha+2,\frac{s}{t})
$$
$\bullet$ for $i_1=2$ $i_2=1$ is
$$
 \kappa_2  \frac{\alpha-1}{2 \alpha +1} \frac{s^{2\alpha+1} {t}^{\alpha-3}}{\Gamma(2 \alpha+1)\Gamma(\alpha)}   \left(1-\frac st\right)^{3\alpha-2} {_2F_1}(3\alpha,2\alpha; 2\alpha+2,\frac{s}{t}).
 $$
 The sum of these three terms can be upper bounded for $0<s<t\le T$, up to a multiplicative constant (depending on $\alpha$, $\sigma$ and $\kappa_2$), by 
$$
(s-t)^{2\alpha-2}+(s-t)^{3\alpha-2}\le (1+T^\alpha)(s-t)^{2\alpha-2}.
$$
Combining this with~\eqref{majo_s_smaller}, we get~\eqref{majo_all_s}.
\end{proof}

\begin{prop}\label{prop_phi}
  Let $\Psi \in \mathcal{C}^2(\R^d, \R)$ such that   \begin{equation}\label{Bound_Psi}\exists C\in \R_+, \ \forall x \in \R^d,\ |\Psi(x)|+\sum_{k=1}^d |\partial_k \Psi(x)|+\sum_{k_1,k_2=1}^d |\partial_{k_1}\partial_{k_2} \Psi(x)| \le Ce^{C\sum_{i=1}^d|x_i|}.\end{equation}
Let $\phi$ be defined by
$$ \phi(r_1,\dots,r_d)=\E[\Psi(X_{r_1},\dots,X_{r_d})], \ r_1,\dots,r_d \in [0,T].$$  
   Then, $\phi$ is continuous and continuously differentiable on $\{(r_1,\dots,r_d) \in [0,T] :  r_1>\dots>r_d \}$. Moreover, there exists $C\in \R_+$ that depends on $T$ and $\Psi$ such that  $|\phi(r_1,\dots,r_d)|\le C$ and 
  \begin{enumerate}
      \item\label{item_1} $$|\phi(r_1,\dots,r_d)-\phi(\mathring{r}_1,\dots,\mathring{r}_d)|\le C\left(\sum_{i=1}^d |r_i-\mathring{r}_i|^{2\alpha-1} \right),$$
      \item       for $2\le i\le d$, $$|\partial_i \phi(r_1,\dots,r_d)|\le C \left((r_{i-1}-r_i)^{2\alpha-2}+(r_i-r_{i+1})^{2\alpha-2} \right),$$

      \item $$|\partial_1 \phi(r_1,\dots,r_d)|\le C (r_1-r_{2})^{2\alpha-2}, \ |\partial_d \phi(r_1,\dots,r_d)|\le C \left(r_d^{2\alpha-2}+(r_{d-1}-r_{d})^{2\alpha-2} \right).$$
  \end{enumerate}
\end{prop}
\begin{proof} The boundedness of $\phi$ is obvious since $|\Psi(x)|\le C e^{C \sum_{i=1}^d|x_i|}$ and $(X_{r_1},\dots,X_{r_d})$ is normally distributed with a continuous mean and covariance from Propositions~\ref{prop_esp} and~\ref{prop_OUexplicit}.

  Let $m$ and $\Gamma$ (resp. $\mathring{m}$ and $\mathring{\Gamma}$) be the mean and the covariance of $(X_{r_1},\dots,X_{r_d})$ (resp. $(X_{\mathring{r}_1},X_{r_2},\dots,X_{r_d})$). We define, for $u\in[0,1]$,

  $$Z_u \sim \mathcal{N}_d\left((1-u)m+u\mathring{m}, (1-u)\Gamma +u\mathring{\Gamma}\right).$$ 
  From~\eqref{Bound_Psi}, $ \E[\partial_1 \Psi(Z_u)]$ and $ \E[\partial_1 \partial_i \Psi(Z_u)]$, $1\le i\le d$ are bounded for $u\in [0,1]$ and we get by Lemma~\ref{lem_weak_ito}.
  \begin{align*}
  &|\phi(\mathring{r}_1,r_2,\dots,r_m)-\phi(r_1,r_2,\dots,r_m)|\\
  &\le C \left( |\E[X_{\mathring{r}_1}]-\E[X_{r_1}]|+ \sum_{i=2}^d |\covX(\mathring{r}_1,r_i)-\covX(r_1,r_i)| +|\covX(\mathring{r}_1,\mathring{r}_1)-\covX(r_1,r_1)| \right)\\
  &\le C \left( |\E[X_{\mathring{r}_1}]-\E[X_{r_1}]|+ \sum_{i=1}^d |\covX(\mathring{r}_1,r_i)-\covX(r_1,r_i)|+ |\covX(\mathring{r}_1,\mathring{r}_1)-\covX(\mathring{r}_1,r_1)|\right),
  \end{align*}
using the triangular inequality.
  By Proposition~\ref{prop_esp}, $t\mapsto \E[X_t]$ is $\alpha$-Hölder (and thus $(2\alpha-1)$-Hölder) on $[0,T]$. We conclude by using Lemma~\ref{lem:covariance}, Equation~\eqref{cov_increment}.  By symmetry, and with the triangular inequality, we deduce~\eqref{item_1}.

Let $T\ge r_1>\dots>r_d>0$ be fixed.
Let $Z_r$ be the distribution of $(X_{r},X_{r_2},\dots,X_{r_d})$, $r \in (r_1,T]$. By using again Lemma~\ref{lem_weak_ito}, we obtain 
\begin{align*}
    &\partial_1 \phi(r_1,\dots,r_d)=\frac{d}{dr}[\E[X_r]]_{ \big|_{r=r_1}} \E[\partial_1 \Psi(X_{r_1},X_{r_2},\dots,X_{r_d})]+\\
   & \ \ \frac 12 \frac{d}{dr} [\covX(r,r)]\big|_{r=r_1} \E[\partial_1^2 \Psi(X_{r_1},X_{r_2},\dots,X_{r_d})]+\sum_{j=2}^d \partial_1  \covX(r_1,r_j) \E[\partial_1 \partial_j \Psi(X_{r_1},X_{r_2},\dots,X_{r_d})].
  \end{align*}
  Using the sub-exponential growth property~\eqref{Bound_Psi}, the above expectations are bounded. By Proposition~\ref{prop_esp},  $\frac{d}{dr}\E[X_r]\le C r^{\alpha-1}\le C r^{2\alpha-2}$ for $r \in (0,T]$ and by Lemma~\ref{prop_OUexplicit}, $\frac{d}{dr}\covX(r,r)\le C r^{2\alpha-2}$ for $r \in (0,T]$. Then, using Lemma~\ref{lem:covariance}, we obtain 
  $$|\partial_1 \phi(r_1,\dots,r_d)|\le C \left(r_1^{2\alpha-2} +\sum_{i=2}^d (r_1-r_i)^{2\alpha-2}\right) \le C d (r_1-r_2)^{2\alpha-2}.$$

For $i\in \{2,\dots,d-1 \}$, we note $Z_r$ be the distribution of $(X_{r_1},\dots,X_{r_{i-1}},X_r,X_{r_{i+1}},\dots,X_{r_d})$, $r \in (r_{i-1},r_{i+2})$. We have \begin{align*}
    &\partial_i \phi(r_1,\dots,r_d)=\frac{d}{dr}[\E[X_r]]_{ \big|_{r=r_i}} \E[\partial_i \Psi(X_{r_1},\dots,X_{r_d})]+\\
   & \ \ \frac 12 \frac{d}{dr} [\covX(r,r)]\big|_{r=r_i} \E[\partial_i^2 \Psi(X_{r_1}, \dots,X_{r_d})]+\sum_{j=1,j \not = i}^d \partial_1  \covX(r_i,r_j) \E[\partial_i \partial_j \Psi(X_{r_1},\dots,X_{r_d})],
  \end{align*}
  which is upper bounded by
  $$C \left(r_i^{2\alpha-2} + \sum_{j \not = i} |r_j-r_i|^{2\alpha-2}\right)\le C d( (r_{i-1}-r_i)^{2\alpha-2}+(r_{i}-r_{i+1})^{2\alpha-2} ),$$
  using that $|r_j-r_i|^{2\alpha-2}\le (r_{i-1}-r_i)^{2\alpha-2}$ if $j<i$ and $|r_j-r_i|^{2\alpha-2}\le (r_{i}-r_{i+1})^{2\alpha-2}$ if $j>i$. The derivative with respect to $r_d$ is handled in the same way.
\end{proof}

We obtain easily then the following corollary.
\begin{corollary}\label{cor_phi}
  Let $\phi$ be defined by~\eqref{def_phi} and assume that $f\in \mathcal{C}_{\exp}^2(\R, \R)$. Then, $\phi(s,t)$ is continuous and continuously differentiable on $\{(s,t) \in (\R_+^*)^2, s\not = t \}$. Moreover, there exists $C\in \R_+$ such that we have for $0\le s,t\le T$, $|\phi(s,t)|\le C$ and the following estimates:
  \begin{enumerate}
      \item $$|\phi(s,t)-\phi(t,t)|\le C|t-s|^{2\alpha-1},$$
      \item $$|\partial_1 \phi(s,t)|\le C \left((s \wedge t)^{2\alpha-2}+|t-s|^{2\alpha-2}\right),$$
      \item $$|\partial_2 \phi(s,t)|\le C\left(t^{\alpha-1}(s \wedge t)^{2\alpha-2}+|t-s|^{2\alpha-2}\right).$$
  \end{enumerate}
\end{corollary}
 \begin{proof} 
  The proof is the same as for Proposition~\ref{prop_phi} with $\Psi(x,y)=f(x)ff'(y)$. However, since $f\in \mathcal{C}_{\exp}^2(\R, \R)$, $\Psi$ may not be $\mathcal{C}^2$ and we explain why it still works.  
  
   For the H\"older property, we introduce
    $$Z_u \sim \mathcal{N}\left(\left[\begin{array}{c} (1-u)\E[X_s]+r \E[X_t]\\ \E[X_t]
       \end{array} \right], (1-u) \left[\begin{array}{cc} \covX(s,s) & \covX(s,t)\\ \covX(s,t) & \covX(t,t)
       \end{array} \right] +u \left[\begin{array}{cc} \covX(t,t) & \covX(t,t)\\ \covX(t,t) & \covX(t,t)
       \end{array} \right] \right),$$
       with $u\in [0,1]$. We stress that the second diagonal term of the covariance is constant. Since $\E[\partial_1 \Psi (Z_u)]$, $\E[\partial^2_{1} \Psi (Z_u)]$, $\E[\partial_{1}\partial_2 \Psi (Z_u)]$  are bounded, we can thus apply the second part of Lemma~\ref{lem_weak_ito}  (without requiring the boundedness of $\E[\partial^2_{2} \Psi (Z_u)]$) and get
   \begin{align*}
     |\phi(t,t)-\phi(s,t)|&\le C \left(|\E[X_t]-\E[X_s]| +|\covX(s,t)-\covX(t,t)|+ |\covX(s,s)-\covX(t,t)|\right)\\
     & \le C \left(|\E[X_t]-\E[X_s]| +2 |\covX(s,t)-\covX(t,t)|+ |\covX(s,s)-\covX(s,t)|\right),
   \end{align*}
   which gives the $2\alpha-1$-Hölder property as in Proposition~\ref{prop_phi}.
   The same remark allows to get also the estimates for $\partial_1\phi$ and $\partial_2\phi$.
 \end{proof}

  We state a lemma that slightly extends~\cite[Lemma 2.2]{FSW} to the non-centered case, and can be seen as a direct application of~\cite[Theorem 1.9]{HPRY}, see also~\cite[Theorem 2.1]{HRY}.
  \begin{lemma}\label{lem_weak_ito}
    Let $a<b$, $d\in \N^*$, $m:[a,b] \to \R^d $ and $\Sigma:[a,b] \to \R^{d\times d}$ be continuously differentiable, such that $\Sigma(r)$ is positive semidefinite for all $r \in [a,b]$.
    Let $g:\R^d\to \R$ be a $\mathcal{C}^2$ function such that there exists $C>0$, 
    $$\forall x \in \R^d, |g(x)|+\sum_{k=1}^d|\partial_k g(x)|+\sum_{k,l=1}^d|\partial^2_{k,l}  g(x)|\le Ce^{C|x|}.$$ 
    Then, $\varphi:[a,b]\to \R$ defined by $\varphi(r)=\E[g(Z_r)]$ with $Z_r\sim \mathcal{N}_d(m(r),\Sigma(r))$ is $\mathcal{C}^1$ and satisfies 
    $$ \varphi(r)=\varphi(a)+\int_a^r \sum_{k=1}^d m'_k(s)\E[\partial_k g(Z_s)] + \frac 12 \sum_{k,l=1}^d \Sigma'_{k,l}(s)\E[\partial^2_{k,l} g(Z_s)] ds, \ r \in[a,b].$$
    Suppose in addition that $\Sigma(t)_{k,l}$ is constant for $(k,l) \not \in \mathcal{I} \subset \{1,\dots,d\}^2$. Then, the same result holds for $g\in \mathcal{C}^1$  such that for all $(k,l)\in \mathcal{I}$, $\partial_kg$ is continuously differentiable with respect to the $l$-th coordinate and satisfies 
    $$\forall x \in \R^d, |g(x)|+\sum_{k=1}^d|\partial_k g(x)|+\sum_{(k,l) \in \mathcal{I}}|\partial^2_{k,l}  g(x)|\le Ce^{C|x|}.$$ 
  \end{lemma}
  \begin{proof}
  Let $F(r,x)=g(m(r)+x)$. We have $\partial_r F(r,x)=\sum_{k=1}^d m'_k(r)\partial_k g(m(r)+x)$, and thus $|\partial_r F(r,x)|\le C e^{C|x|}$ for some constant $C>0$, since $m$ is $\mathcal{C}^1$ on $[a,b]$. We apply then the weak Itô formula~\cite[Theorem 1.9]{HPRY}.  

  The second statement can be easily deduced by adapting the proof of~\cite[Theorem 1.9]{HPRY}, since the partial derivatives $\partial_{k,l}^2g (x)$ for $(k,l) \not \in \mathcal{I}$ no longer appear in the calculations.  
  \end{proof}
  \begin{corollary}\label{cor_weak_ito}
    Let $Z_1\sim\mathcal{N}_d(m_1,\Sigma_1)$ and $Z_2\sim\mathcal{N}_d(m_2,\Sigma_2)$ with $m_1,m_2\in \R^d$ and $\Sigma_1,\Sigma_2$ positive semidefinite matrices. Let $g$ satisfy the same assumption as in Lemma~\ref{lem_weak_ito}. Then, for any $M>0$, there is a constant $C\in \R_+$ such that    
     $$|\E[g(Z_2)]-\E[g(Z_1)]|\le C \left(\sum_{k=1}^d |(m_1-m_2)_k|+  \sum_{k,l=1}^d |(\Sigma_1-\Sigma_2)_{k,l}|\right),$$
for all $m_1,m_2,\Sigma_1,\Sigma_2$ such that $\max(|m_1|,|m_2|,|\Sigma_1|,|\Sigma_2|)\le M$.
    \end{corollary}
  \begin{proof}
    Let $\varphi(r)=\E[g(Z_r)]$ with $Z_r\sim \mathcal{N}_d(rm_1+(1-r)m_2,r\Sigma_1+(1-r)\Sigma_2)$ for $r\in[0,1]$. The, we get by Lemma~\ref{lem_weak_ito} 
    $$\varphi(1)-\varphi(0)=\int_0^1 \sum_{k=1}^d (m_1-m_2)_k \E[\partial_k g(Z_r)] +\frac 12 \sum_{k,l=1}^d (\Sigma_1-\Sigma_2)_{k,l} \E[\partial^2_{k,l} g(Z_r)] dr.$$
    From the assumption on~$g$, $\sup_{r\in [0,1]} \sum_{k=1}^d  \E[|\partial_k g(Z_r)|]+ \sum_{k,l=1}^d \E[|\partial^2_{k,l} g(Z_r)|]$ is bounded by a constant depending on~$M$. We then get the claim.   
  \end{proof}

  \section{Technical proofs}

\begin{proof}[Proof of Theorem~\ref{thm_technique}]
    We first observe that $\E[ X_t (X_u-X_{\eta(u)})]=\E[ X_t] \E[ (X_u-X_{\eta(u)})]+ \E[ \mathcal{Y}_t (X_u-X_{\eta(u)})]$. Therefore, 
    \begin{align} \max_{1\le k\le n}\left|  \int_0^{t_k}\frac{(t_k-u)^{\alpha-1}}{\Gamma(\alpha)} \E[ X_t (X_u-X_{\eta(u)})]du   \right| \le &|\E[X_t]| \max_{1\le k\le n}\left|  \int_0^{t_k}\frac{(t_k-u)^{\alpha-1}}{\Gamma(\alpha)} \E[X_u-X_{\eta(u)}]du   \right| \notag\\
      &+\max_{1\le k\le n}\left|  \int_0^{t_k}\frac{(t_k-u)^{\alpha-1}}{\Gamma(\alpha)} \E[ \mathcal{Y}_t (X_u-X_{\eta(u)})]du   \right|.\label{majoration_1}\end{align}
  We use the estimate $ \left|  \int_0^{t_k}\frac{(t_k-u)^{\alpha-1}}{\Gamma(\alpha)}  (\E[X_u]-\E[X_{\eta(u)}])du   \right| \le \frac{C_T}{n}$ (see Equation~\eqref{Diethelm_esperance}) for some $C_T\in \R_+$. By Proposition~\ref{prop_esp}, $|\E[X_t]|$ is bounded uniformly on $t\in[0,T]$, and we get 
  \begin{equation}\label{esperance_terme}
    \forall t \in [0,T], \ |\E[X_t]| \max_{1\le k\le n}\left|  \int_0^{t_k}\frac{(t_k-u)^{\alpha-1}}{\Gamma(\alpha)} \E[X_u-X_{\eta(u)}]du   \right| \le \frac{C_T}{n},
  \end{equation}
  where $C_T$ is a constant depending on $T$ and on $X_0$, $\kappa_1$, $\kappa_2$, $\sigma$ and $\alpha$, that may change from line to line.

  We are thus interested in studying $\E[ \cY_t (X_u-X_{\eta(u)})]=\covX(t,u)-\covX(t,\eta(u))$. For $\xi  \in [0,T^{\alpha}]$, let\footnote{Note that this change of variable allows us to get $t^{\alpha-1}$ instead of $t^{2\alpha-2}$ that appears from a direct application of Lemma~\ref{lem:covariance}.} $g_t(\xi)=\covX(t,\xi^{1/\alpha})$. We have $g'_t(\xi)=\frac{1}{\alpha}\xi^{1/\alpha-1}\partial_2 \covX(t, \xi^{1/\alpha})$, and by Lemma~\ref{lem:covariance}, we get
  $$\forall \xi \in (0,T),\ |g'_t(\xi)|\le C_T(t^{\alpha-1}+|t-\xi^{1/\alpha}|^{2\alpha-2}).$$
  We have
  \begin{align*}
    &\left|\int_0^{t_k}\frac{(t_k-u)^{\alpha-1}}{\Gamma(\alpha)} \E[ \cY_t (X_u-X_{\eta(u)})]du \right|\\
    &\le \int_0^{t_k}\frac{(t_k-u)^{\alpha-1}}{\Gamma(\alpha)} |g_t(u^\alpha)-g_t(\eta(u)^\alpha)|du =\sum_{i=1}^k\int_{t_{i-1}}^{t_i}\frac{(t_k-u)^{\alpha-1}}{\Gamma(\alpha)} |g_t(u^\alpha)-g_t(t_{i-1}^\alpha)|du \\
    &=\sum_{i=1}^k\int_{t_{i-1}}^{t_i}\frac{(t_k-u)^{\alpha-1}}{\Gamma(\alpha)} (u^\alpha-t_{i-1}^\alpha) \left|\int_0^1 g_t'(z u^\alpha+ (1-z)t_{i-1}^\alpha) dz \right|du \\
    &\le \sum_{i=1}^k\int_{t_{i-1}}^{t_i}\frac{(t_k-u)^{\alpha-1}}{\Gamma(\alpha)} (u^\alpha-t_{i-1}^\alpha)\int_0^1 C_T\left(t^{\alpha-1}+\left|t-(z u^\alpha+ (1-z)t_{i-1}^\alpha)^{1/\alpha}\right|^{2\alpha-2}\right)dz du\\
   & =: \mathcal{A}+\mathcal{B}.
  \end{align*}
  The first term $\mathcal{A}:=C_T t^{\alpha-1} \sum_{i=1}^k\int_{t_{i-1}}^{t_i}\frac{(t_k-u)^{\alpha-1}}{\Gamma(\alpha)} (u^\alpha-t_{i-1}^\alpha)du =C_T t^{\alpha-1} \int_0^{t_k} \frac{(t_k-u)^{\alpha-1}}{\Gamma(\alpha)} (u^\alpha -\eta(u)^\alpha) du \le C_T t^{\alpha-1} \frac{t_k^{2\alpha-1}}{n} \le\frac{C_T t^{\alpha-1}}{n}$ by using~\cite[Theorem 2.4 (b)]{DFF}.
  The second term is equal to:
  \begin{align}
    \mathcal{B}:=&\sum_{i=1}^k\int_{t_{i-1}}^{t_i}\frac{(t_k-u)^{\alpha-1}}{\Gamma(\alpha)} (u^\alpha-t_{i-1}^\alpha)\int_0^1 C_T \left|t-(z u^\alpha+ (1-z)t_{i-1}^\alpha)^{1/\alpha}\right|^{2\alpha-2} dz du \notag \\
    &=\sum_{i=1}^k\int_{t_{i-1}}^{t_i}\frac{(t_k-u)^{\alpha-1}}{\Gamma(\alpha)} (u^\alpha-t_{i-1}^\alpha)\int_{t_{i-1}}^u C_T \left|t-v\right|^{2\alpha-2} \frac{dv}{\frac 1 \alpha (u^\alpha -t_{i-1}^\alpha) v^{1-\alpha}} du \notag\\
    &=\sum_{i=1}^k\int_{t_{i-1}}^{t_i}\frac{\alpha (t_k-u)^{\alpha-1}}{\Gamma(\alpha)} \int_{t_{i-1}}^u C_T v^{\alpha-1}\left|t-v\right|^{2\alpha-2} dv  du  \notag\\
    & \le \sum_{i=1}^k\int_{t_{i-1}}^{t_i}\frac{\alpha (t_k-u)^{\alpha-1}}{\Gamma(\alpha)} \int_{t_{i-1}}^{t_i} C_T v^{\alpha-1}\left|t-v\right|^{2\alpha-2} dv  du  \notag\\
    &=\frac{C_T}{\Gamma(\alpha)}\sum_{i=1}^k \left( (t_k-t_{i-1})^\alpha - (t_k-t_{i})^\alpha \right)\int_{t_{i-1}}^{t_i}  v^{\alpha-1}\left|t-v\right|^{2\alpha-2} dv \notag .
  \end{align}
  Next, we use 
  \begin{equation*} v^{\alpha-1}\left|t-v\right|^{2\alpha-2} \le \left( \frac{t}{2}\right)^{\alpha-1}\left|t-v\right|^{2\alpha-2} +\left( \frac{t}{2}\right)^{2\alpha-2}v^{\alpha-1}, 
  \end{equation*}
  since $v^{\alpha-1}\left|t-v\right|^{2\alpha-2} \le \left( \frac{t}{2}\right)^{\alpha-1}\left|t-v\right|^{2\alpha-2} $ for $v\ge t/2$ and  $v^{\alpha-1}\left|t-v\right|^{2\alpha-2} \le \left( \frac{t}{2}\right)^{2\alpha-2}v^{\alpha-1}$ for $v \in (0,t/2)$.
  This gives 
  \begin{equation}
    \mathcal{B} \le \mathcal{B}_1+\mathcal{B}_2,  \label{2ndT}
  \end{equation}
  where 
  \begin{align*}
    \mathcal{B}_1&:=\left( \frac{t}{2}\right)^{2\alpha-2} \frac{C_T}{\Gamma(\alpha)}\sum_{i=1}^k \left( (t_k-t_{i-1})^\alpha - (t_k-t_{i})^\alpha \right)\int_{t_{i-1}}^{t_i}  v^{\alpha-1}  dv \\
    \mathcal{B}_2&:=\left( \frac{t}{2}\right)^{\alpha-1} \frac{C_T}{\Gamma(\alpha)}\sum_{i=1}^k \left( (t_k-t_{i-1})^\alpha - (t_k-t_{i})^\alpha \right)\int_{t_{i-1}}^{t_i}  \left|t-v\right|^{2\alpha-2}  dv.
  \end{align*}
  On the one hand, we have 
  \begin{align}
   \mathcal{B}_1\le  &\left( \frac{t}{2}\right)^{2\alpha-2} \frac{C_T}{\Gamma(\alpha)}\sum_{i=1}^k \left( (t_k-t_{i-1})^\alpha - (t_k-t_{i})^\alpha \right)\int_{t_{i-1}}^{t_i}  v^{\alpha-1}  dv \notag \\
    \le &\left( \frac{t}{2}\right)^{2\alpha-2} \frac{C_T}{\alpha \Gamma(\alpha)}\sum_{i=1}^k \left( (t_k-t_{i-1})^\alpha - (t_k-t_{i})^\alpha \right)\left(t_i^\alpha - t_{i-1}^\alpha \right) \notag\\
  \le &C_T \frac{ t^{2\alpha-2}k^{2\alpha-1}}{n^{2\alpha}} \le C_T \frac{ t^{2\alpha-2}}{n},\label{somme_1} 
  \end{align}
  by Lemma~\ref{lemma_beta_disc}, using that $2\alpha>1$ and $k\le n$.
  
  On the other hand,  we use the following equality for $\mathcal{B}_2$:
  $$\int_{t_{i-1}}^{t_i}  \left|t-v\right|^{2\alpha-2}  dv = \frac 1 {2\alpha-1} \times \begin{cases} (t-t_{i-1})^{2\alpha-1}-(t-t_{i})^{2\alpha-1} \text{ if } t_i \le t\\
    (t_{i}-t)^{2\alpha-1} + (t-t_{i-1})^{2\alpha-1}  \text{ if } t_{i-1} \le t <t_i\\(t_{i}-t)^{2\alpha-1} - (t_{i-1}-t)^{2\alpha-1} \text{ if } t_{i-1} > t.\end{cases}$$
    We then split the sum defining $\mathcal{B}_2$ in three terms:
    \begin{align*}
     \mathcal{B}_2= \sum_{i=1}^k a_{i,k}& = \sum_{i=1, t_i\le t }^k a_{i,k}+\sum_{i=1, t_{i-1}> t }^k a_{i,k} +\mathbf{1}_{k\ge {\bf i}(t)+1} a_{{\bf i}(t)+1,k}\\
    & = \sum_{i=1}^{\min(k,{\bf i}(t))} a_{i,k}+\sum_{i={\bf i}(t)+2}^{k} a_{i,k} +\mathbf{1}_{k\ge {\bf i}(t)+1} a_{{\bf i}(t)+1,k}=:\mathcal{B}_{21}+\mathcal{B}_{22}+\mathcal{B}_{23},
  \end{align*}
  where ${\bf i}(t) \in \{0,\dots,n\}$ is the index such that $\eta(t)=t_{{\bf i}(t)}$ for $t\in[0,T]$ and setting  $a_{i,k}=\left( \frac{t}{2}\right)^{\alpha-1} \frac{C_T}{\Gamma(\alpha)}\left( (t_k-t_{i-1})^\alpha - (t_k-t_{i})^\alpha \right)\int_{t_{i-1}}^{t_i}  \left|t-v\right|^{2\alpha-2}  dv$, with using the standard convention $\sum_{i=i_1}^{i_2}a_{i,k}= 0$ if $i_2<i_1$.

    \noindent {\bf $\bullet$ Term $\mathcal{B}_{21}$:} Since $t\mapsto (t-t_{i-1})^{2\alpha-1}-(t-t_{i})^{2\alpha-1}$ is decreasing for $t\ge t_i$,
  \begin{align*}
    &\sum_{i\le k, t_i\le t} \left( (t_k-t_{i-1})^\alpha - (t_k-t_{i})^\alpha \right) \left( (t-t_{i-1})^{2\alpha-1}-(t-t_{i})^{2\alpha-1}\right) \\
    \le &\sum_{i\le k, t_i\le t} \left( (t_k-t_{i-1})^\alpha - (t_k-t_{i})^\alpha \right) \left( \eta(t)-t_{i-1})^{2\alpha-1}-(\eta(t)-t_{i})^{2\alpha-1}\right)\\
    = & \left( \frac{T}{n}\right)^{3\alpha-1} \sum_{i=1}^{\min(k,{\bf i}(t))}  \left( (k-i+1)^\alpha - (k-{i})^\alpha \right) \left( ({\bf i}(t)-i+1 )^{2\alpha-1}-({\bf i}(t)-{i})^{2\alpha-1}\right)\\
    \le & \frac{C_T}{n^{3\alpha-1}}  \sum_{i=1}^{\min(k,{\bf i}(t))}  (k-i+1)^{\alpha-1}  ({\bf i}(t)-i+1 )^{2\alpha-2} \text{ by~\eqref{eq_LiZeng}}\\
    \le& \frac{C_T}{n^{3\alpha-1}}  \sum_{i=1}^{\min(k,{\bf i}(t))}  (\min(k,{\bf i}(t))-i+1)^{3\alpha-3} = \frac{C_T}{n^{3\alpha-1}}  \sum_{i=1}^{\min(k,{\bf i}(t))}  i^{3\alpha-3}\le \frac{C_T}{n^{3\alpha-1}}  \sum_{i=1}^{n}  i^{3\alpha-3}.
  \end{align*}
  We now observe that $\sum_{i=1}^n i^{3\alpha-3}=O(n^{3\alpha-2})$ for $\alpha>2/3$, $\sum_{i=1}^n i^{3\alpha-3}=O(\log(n))$ for $\alpha=2/3$, and $\sum_{i=1}^n i^{3\alpha-3}=O(1)$ for $1/2<\alpha<2/3$. This gives 
  $$\mathcal{B}_{21}\le C_T t^{\alpha-1} {\bf v}_n(\alpha).$$
  
   \noindent {\bf $\bullet$ Term $\mathcal{B}_{22}$:} Similarly, since $t\mapsto (t_{i}-t)^{2\alpha-1}-(t_{i-1}-t)^{2\alpha-1}$ is increasing for $t< t_{i-1}$, we get
  \begin{align*}
    &\sum_{i\le k,  t< t_{i-1}}   \left( (t_k-t_{i-1})^\alpha - (t_k-t_{i})^\alpha \right) \left( (t_{i}-t)^{2\alpha-1}-(t_{i-1}-t)^{2\alpha-1}\right)\\
    \le &\sum_{i\le k,  t < t_{i-1}}   \left( (t_k-t_{i-1})^\alpha - (t_k-t_{i})^\alpha \right) \left( (t_{i}-t_{{\bf i}(t)+1})^{2\alpha-1}-(t_{i-1}-t_{{\bf i}(t)+1})^{2\alpha-1}\right)\\
    =&\frac{C_T}{n^{3\alpha-1}} \sum_{i=\eta(t)+2}^k \left((k-i+1)^\alpha-(k-i)^\alpha\right)\left((i-{\bf i}(t)-1)^{2\alpha-1}-(i-{\bf i}(t)-2)^{2\alpha-1}\right)\\
    \le & \frac{C_T}{n^{3\alpha-1}} \sum_{  i={\bf i}(t) +2}^{ k} (k-i+1)^{\alpha-1}  (i-{\bf i}(t)-1)^{2\alpha-2} \text{ by~\eqref{eq_LiZeng}}\\=&\frac{C_T}{n^{3\alpha-1}} \sum_{  i=1}^{ k-{\bf i}(t)-1 } (k-{\bf i}(t) -i)^{\alpha-1}  i^{2\alpha-2} \\
    \le& \mathbf{1}_{k\ge {\bf i}(t)+2 }\frac{C_T}{n^{3\alpha-1}} (k-{\bf i}(t)-1)^{3\alpha-2},
  \end{align*}
  where we used Lemma~\ref{lemma_beta_disc2} for the last inequality. 
  Therefore, we get $$\mathcal{B}_{22} \le C_T t^{\alpha-1} {\bf v}_n(\alpha),$$
  by using that $k-{\bf i}(t)-1\le n$ for $3\alpha-2>0$ and $\mathbf{1}_{k\ge {\bf i}(t) +2 }(k-{\bf i}(t)-1)^{3\alpha-2} \le 1$ for $3\alpha-2\le 0$.
  
   \noindent {\bf $\bullet$ Term $\mathcal{B}_{23}$:} We have
   \begin{align*}\mathcal{B}_{23}=&\mathbf{1}_{k\ge {\bf i}(t)+1}\left( \frac{t}{2}\right)^{\alpha-1} \frac{C_T}{(2\alpha-1)\Gamma(\alpha)} \left( (t_k-\eta(t))^\alpha - (t_k-t_{{\bf i}(t)+1})^\alpha \right) \left((t_{{\bf i}(t)+1}-t)^{2\alpha-1} + (t-\eta(t))^{2\alpha-1}\right) \\
    &\le \mathbf{1}_{k\ge {\bf i}(t)+1} C_T t^{\alpha-1} \frac{(t_{{\bf i}(t)+1}-\eta(t))^{\alpha}}{n^{2\alpha-1}}\le C_T t^{\alpha-1} \frac{1}{n^{3\alpha-1}} \le  C_T t^{\alpha-1} {\bf v}_n(\alpha) , 
   \end{align*}
  by using the $\alpha$-H\"older inequality and also that $2\alpha-1>0$ and $\max(t-\eta(t),t_{{\bf i}(t)+1}-t)\le T/n$. Gathering all the terms, we get 
  \begin{equation*}
    \mathcal{B}_{2}=\mathcal{B}_{21}+\mathcal{B}_{22}+\mathcal{B}_{23}\le  C_T t^{\alpha-1} {\bf v}_n(\alpha).
  \end{equation*}
  From~\eqref{majoration_1}, we get the claim by using~\eqref{esperance_terme},~\eqref{2ndT},~\eqref{somme_1} with the last inequality and  $t^{\alpha-1}\le 1+t^{2\alpha-2}$ since $\alpha\le 1$. 
  \end{proof}

\begin{proof}[Proof of Lemma~\ref{lem_err_Malliavin3}]
Let us first check by induction that $\xcn_{t_k}\in \mathbb{D}^{1,2}$  so that $\D_s \xcn_{t_k}$ is well defined. For $k=1$, we have $\xcn_{t_1}=x_0+(\kappa_1+\kappa_2 x_0) \frac{t_1^\alpha}{\Gamma(\alpha+1)}+ \sigma \int_0^{t_1} \frac{(t_1-s)^{\alpha-1}}{\Gamma(\alpha)} dW_s$ and thus $\D_s\xcn_{t_1}=\mathbf{1}_{s<t_1}\frac{(t_1-s)^{\alpha-1}}{\Gamma(\alpha)}$, which is deterministic and satisfies $\int_0^{t_1} (\D_s \xcn_{t_1})^2 ds <\infty$. Next, we proceed by induction on~$k$ from the dynamics~\eqref{eq:scheme_price} on $\xcn$. Then, we necessarily have $\D_s \xcn_{t_k}=\sum_{j=1}^{k-1} \left(\int_{t_j}^{t_{j+1}}  \frac{(t_k-u)^{\alpha-1}}{\Gamma(\alpha)}du \right) \D_s \xcn_{t_j}+ \sigma\frac{(t_k-s)^{\alpha-1}}{\Gamma(\alpha)}$, which gives that $\int_0^{t_k} (\D_s \xcn_{t_k})^2ds<\infty$ by using $\int_0^{t_k} (t_k-s)^{2\alpha-2}ds <\infty$ and the induction hypothesis. Therefore,
$\D_s\xcn_{t_k}$ is well defined, deterministic, and satisfies~\eqref{eq_dxcn}.

From~\eqref{SVE_OU}, we also easily get $$
\D_sX_{t_k}= \kappa_2 \int_0^{t_k} \D_s X_{u} \frac{(t_k-u)^{\alpha-1}}{\Gamma(\alpha)}du + \sigma\frac{(t_k-s)^{\alpha-1}}{\Gamma(\alpha)}.$$
Combined with~\eqref{eq_dxcn}, this gives 
\begin{align}\label{rec_delta_Malliavin}
  \D_sX_{t_k}-\D_s\xcn_{t_k} &=\kappa_2 \int_0^{t_k} (\D_s X_{\eta(u)}-\D_s \xcn_{\eta(u)})  \frac{(t_k-u)^{\alpha-1}}{\Gamma(\alpha)}du+ \kappa_2 \int_0^{t_k} (\D_s X_{u}-\D_s X_{\eta(u)})  \frac{(t_k-u)^{\alpha-1}}{\Gamma(\alpha)}du \notag
  \\
  &= \kappa_2 \sum_{i=1}^{k-1} c_{i,k} (\D_sX_{t_i}-\D_s\xcn_{t_i})+\kappa_2 \int_0^{t_k} (\D_s X_{u}-\D_s X_{\eta(u)})  \frac{(t_k-u)^{\alpha-1}}{\Gamma(\alpha)}du , 
\end{align}
with $c_{i,k}=\int_{t_i}^{t_{i+1}}\frac{(t_k-u)^{\alpha-1}}{\Gamma(\alpha)}du$. We integrate on $[0,t_j]$, set $I\Delta_{j,k}=\int_0^{t_j} \left(\D_s\xcn_{t_k} -\D_sX_{t_k}\right) ds$ and get for $j\le k$,
$$I\Delta_{j,k}=\kappa_2  \sum_{i=1}^{k-1} c_{i,k}  I\Delta_{j,i} +\kappa_2 \int_0^{t_j} \int_0^{t_k} (\D_s X_{u}-\D_s X_{\eta(u)})  \frac{(t_k-u)^{\alpha-1}}{\Gamma(\alpha)}du ds.$$
Then,  we get 
$$|I\Delta_{j,k}|\le |\kappa_2|  \sum_{i=1}^{k-1} c_{i,k}  |I\Delta_{j,i}| +|\kappa_2|  \int_0^{t_k} \frac{(t_k-u)^{\alpha-1}}{\Gamma(\alpha)} \left| \int_0^{t_j}(\D_s X_{u}-\D_s X_{\eta(u)})  ds\right| du.$$
So, by a Gr\"onwall type inequality~\cite[Lemma 3.4]{LiZe}, it is sufficient to check that 
\begin{equation}\label{bound_Sjk}\cS_{j,k}:=\int_0^{t_k} \frac{(t_k-u)^{\alpha-1}}{\Gamma(\alpha)} \left| \int_0^{t_j}(\D_s X_{u}-\D_s X_{\eta(u)})  ds\right| du\le C\frac{T}{n},
\end{equation}
for a constant $C$ that does not depend on $0\le j\le k\le n$.
We have by Lemma~\ref{prop_OUexplicit}
\begin{align*}
  \int_0^{t_j}(\D_s X_{u}&-\D_s X_{\eta(u)})  ds= \begin{cases}
    \sigma \sum_{i=1}^\infty \kappa_2^{i-1}  \frac{u^{i\alpha}-\eta(u)^{i\alpha}}{\Gamma(i \alpha+1)} \text{ for } u\le t_j,  \\
    \sigma \sum_{i=1}^\infty \kappa_2^{i-1}  \frac{u^{i\alpha}-(u-t_j)^{i\alpha}-(\eta(u)^{i\alpha} -(\eta(u)-t_j)^{i\alpha} )}{\Gamma(i \alpha+1)} \text{ for } u\in]t_j,t_k],
  \end{cases}\\
  &=\sigma \sum_{i=1}^\infty \kappa_2^{i-1}  \frac{u^{i\alpha}-\eta(u)^{i\alpha}}{\Gamma(i \alpha+1)} -\mathbf{1}_{u>t_j}\sigma \sum_{i=1}^\infty \kappa_2^{i-1}  \frac{(u-t_j)^{i\alpha}-(\eta(u)-t_j)^{i\alpha} }{\Gamma(i \alpha+1)}.
\end{align*}
Therefore, 
\begin{align*}
  \cS_{j,k}\le &\int_0^{t_k} \frac{(t_k-u)^{\alpha-1}}{\Gamma(\alpha)} \sum_{i=1}^\infty |\kappa_2|^{i-1}  \frac{u^{i\alpha}-\eta(u)^{i\alpha}}{\Gamma(i \alpha+1)}du  \\
  &+ \int_{t_j}^{t_k} \frac{(t_k-u)^{\alpha-1}}{\Gamma(\alpha)} \sum_{i=1}^\infty |\kappa_2|^{i-1}  \frac{(u-t_j)^{i\alpha}-(\eta(u)-t_j)^{i\alpha}}{\Gamma(i \alpha+1)}du \\
  =&\int_0^{t_k} \frac{(t_k-u)^{\alpha-1}}{\Gamma(\alpha)} \sum_{i=1}^\infty |\kappa_2|^{i-1}  \frac{u^{i\alpha}-\eta(u)^{i\alpha}}{\Gamma(i \alpha+1)}du\\&+\int_0^{t_{k-j}} \frac{(t_{k-j}-u)^{\alpha-1}}{\Gamma(\alpha)} \sum_{i=1}^\infty |\kappa_2|^{i-1}  \frac{u^{i\alpha}-\eta(u)^{i\alpha}}{\Gamma(i \alpha+1)}du.
\end{align*} 
We separate the term $i=1$ from $z(u):=\sum_{i=2}^\infty |\kappa_2|^{i-1}\frac{u^{i\alpha}}{\Gamma(i \alpha+1)}$ which is a $\mathcal{C}^1([0,T])$ function and apply respectively \cite[Theorem 2.4 (b) and (a)]{DFF} to get $|\cS_{j,k}|\le \frac{T}{n}\left( C T^{2\alpha-1}+\frac{\|z'\|_{\infty}}{\alpha}T^\alpha \right)$ and thus~\eqref{bound_Sjk}.

For the last inequality, we set $I_{j,k}=\int_0^{t_j} |\D_s\xcn_{t_k}| ds$ and we obtain from~\eqref{eq_dxcn}
$$I_{j,k}\le |\kappa_2|\sum_{i=1}^{k-1}c_{i,k}I_{j,i} + \frac{\sigma T^\alpha}{\Gamma(\alpha+1)}, $$
in a similar way as above. We conclude by the same  Gr\"onwall type inequality~\cite[Lemma 3.4]{LiZe}.
\end{proof}

\begin{proof}[Proof of Lemma~\ref{lem_err_Malliavin_int}]
Let $k \in \{1,\dots,n\}$ be fixed. We set $\widetilde{I \Delta}_{\ell,k}=\int_{0}^{t_k}|\D_s X_{t_\ell}-\D_s\xcn_{t_\ell}| ds$ and get from~\eqref{rec_delta_Malliavin} and the triangle inequality
\begin{equation}\label{ID_tilde}\widetilde{I \Delta}_{\ell,k} \le |\kappa_2| \sum_{i=1}^{\ell-1}c_{i,\ell} \widetilde{I \Delta}_{i,k} +\widetilde{\mathcal{S}}_{\ell,k},  \end{equation}
with 
$$\widetilde{\mathcal{S}}_{\ell,k}= |\kappa_2| \int_0^{t_\ell} \left(\int_{0}^{t_k} |\D_s X_{u}-\D_s X_{\eta(u)}|   ds \right) \frac{(t_\ell -u)^{\alpha-1}}{\Gamma(\alpha)} du.$$
We now bound uniformly $\widetilde{\mathcal{S}}_{\ell,k}$ and we write 
\begin{align*}
|\D_s X_{u}-\D_s X_{\eta(u)}|\le &|\mathbf{1}_{s<u}\frac{(u-s)^{\alpha-1}}{\Gamma(\alpha)}-\mathbf{1}_{s<\eta(u)}\frac{(\eta(u)-s)^{\alpha-1}}{\Gamma(\alpha)}|\\
&+\sum_{i=2}^\infty |\kappa_2|^{i-1}\frac{(u-s)_+^{i\alpha-1} -(\eta(u)-s)_+^{i\alpha-1}}{\Gamma(i\alpha)}.
\end{align*}
 We then obtain 
\begin{equation}\label{bound_integ}\int_0^{t}  |\D_s X_{u}-\D_s X_{\eta(u)}|   ds  \le \frac{C}{n^\alpha},\end{equation}
since the integral of the sum can be calculated exactly and is a $O(1/n)$ using that $2\alpha-1>0$, and \cite[Lemma 4.5 with $\beta=0$]{FSW} gives the bound for the first term. 
This gives $|\widetilde{\mathcal{S}}_{\ell,k}|\le C|\kappa_2|\frac{T^{\alpha}}{\Gamma(\alpha+1)} \frac{1}{n^\alpha}$. This bound is thus uniform in $\ell$, $k$ and $n$. We finally apply the Gronwall type estimate given by~\cite[Lemma 3.4]{LiZe} to~\eqref{ID_tilde} and get the claim. 
\end{proof}

\begin{proof}[Proof of Proposition~\ref{prop_E22}]  
Let $I=\sup\{i>j: \alpha(k)=k-1 \text{ for } j<k\le i  \}$, with $I=j$ if  $\{i>j: \alpha(k)=k-1 \text{ for } j<k\le i  \}=\emptyset$ i.e. $\alpha(j+1)<j$. 

{\bf Case 1: $I=j$}. We can differentiate directly and get 
\begin{align*}
  \partial_j  H_j(r_1,\dots,r_j)= & \int_0^{r_{j}}\int_0^{r_{j+2}}\dots\int_0^{r_{m-1}}    \phi({\bf \bar{r}}) \times \\
  & \phantom{*********} (r_{\alpha(j+1)}-r_j)^{\ell_{j+1} \alpha-1}  \prod_{i>j+1} (\bar{r}_{\alpha(i)}-r_i)^{\ell_i \alpha-1}  dr_m\dots dr_{j+2} \\
  &+\int_0^{r_j}\int_0^{r_{j+1}}\dots\int_0^{r_{m-1}}  \partial_j \phi({\bf r}) \prod_{i>j} (r_{\alpha(i)}-r_i)^{\ell_i \alpha-1} dr_m\dots dr_{j+1} \\
  &+\int_0^{r_j}\int_0^{r_{j+1}}\dots\int_0^{r_{m-1}}  \phi({\bf r})   \sum_{i>j:\alpha(i)=j} (\ell_{i} \alpha-1) (r_j-r_i)^{\ell_{i} \alpha-2} \times \\
  &  \phantom{*********}\prod_{i'>j, i'\not = i} (r_{\alpha(i')}-r_{i'})^{\ell_{i'} \alpha-1}
  dr_m\dots dr_{j+1}\\
  &=:I_1+I_2+I_3,
\end{align*}
with  ${\bf \bar{r}}=(r_1,\dots,r_j,r_j,r_{j+2},\dots,r_m)$.
We use that $\phi$ is bounded and integrate from $dr_m$ to $dr_{j+2}$ to get 
$$|I_1| \le  \|\phi\|_\infty  \frac{T^{ \sum_{i=j+2}^m \alpha \ell_i}}{\prod_{i=j+2}^m \alpha \ell_i} (r_{\alpha(j+1)}-r_j)^{\ell_{j+1} \alpha-1}. $$
Note that  $\alpha(j+1)<j$ since we are in the case $I=j$. We get that 
$$\int_0^{r_{j-1}}  (r_{\alpha(j+1)}-r_j)^{\ell_{j+1} \alpha-1}   dr_j  \le  \frac{1}{\alpha \ell_j} T^{\ell_{j+1} \alpha}. $$
Thus, $I_1$ is integrable  with respect to $r_j$.

For the second integral, we use Proposition~\ref{prop_phi}~(2) that gives 
$$|\partial_j\phi ({\bf r})|\le C((r_j-r_{j+1})^{2\alpha-2}+(r_{j-1}-r_{j})^{2\alpha-2}).$$
Integrating again from $dr_m$ to $dr_{j+2}$, we obtain
$$|I_2|\le C \frac{T^{ \sum_{i=j+2}^m \alpha \ell_i}}{\prod_{i=j+2}^m \alpha \ell_i}  \int_0^{r_j}  ((r_j-r_{j+1})^{2\alpha-2}+(r_{j-1}-r_{j})^{2\alpha-2}) (r_{\alpha(j+1)}-r_{j+1})^{\ell_{j+1}\alpha -1}  dr_{j+1}.$$
If $\ell_{j+1}\ge 2$, we get 
$$|I_2| \le C \frac{T^{ \sum_{i=j+2}^m \alpha \ell_i}}{ \prod_{i=j+2}^m \alpha \ell_i} T^{\ell_{j+1}\alpha -1} \left(  \frac{r_j^{2\alpha-1}}{2\alpha-1} + (r_{j-1}-r_{j})^{2\alpha-2} \right),$$
and $\int_0^{r_{j-1}}   \left(  \frac{r_j^{2\alpha-1}}{2\alpha-1} + (r_{j-1}-r_{j})^{2\alpha-2} \right) dr_j \le C_T$. Otherwise, $\ell_{j+1}=1$ and we use $(r_{\alpha(j+1)}-r_{j+1})^{\alpha -1}\le (r_{j-1}-r_{j+1})^{\alpha -1}$ since  $\alpha(j+1)<j$. This gives 
 \begin{align*}
 |I_2|&\le C \frac{T^{ \sum_{i=j+2}^m \alpha \ell_i}}{\prod_{i=j+2}^m \alpha \ell_i}  \int_0^{r_j}  ((r_{j-1}-r_{j})^{2\alpha-2}+(r_j-r_{j+1})^{2\alpha-2})   (r_{j-1}-r_{j+1})^{\alpha -1} dr_{j+1}\\
 &\le C \frac{T^{ \sum_{i=j+2}^m \alpha \ell_i}}{\prod_{i=j+2}^m \alpha \ell_i}   \left( \frac{T^\alpha}{\alpha} (r_{j-1}-r_j)^{2\alpha-2}+ \int_0^{r_j}  (r_{j}-r_{j+1})^{2\alpha-2}   (r_{j-1}-r_{j+1})^{\alpha -1} dr_{j+1}\right).
 \end{align*}

If $\alpha>2/3$, we have $ (r_{j-1}-r_{j+1})^{\alpha -1}\le  (r_{j}-r_{j+1})^{\alpha -1}$ and 
$\int_0^{r_j}  (r_{j}-r_{j+1})^{3\alpha-3}   dr_{j+1}\le  \frac{T^{3\alpha-2}}{3\alpha-2}$. We can thus conclude as for $\ell_{j+1}\ge 2$.  Otherwise, we have $\alpha \in (1/2,2/3]$ and the  integral is equal, up to a positive multiplicative constant to
$$r_j^{2\alpha-1}r_{j-1}^{\alpha-1} {_2}F_1(1-\alpha,1;2\alpha,\frac{r_j}{r_{j-1}}), $$
using the change of variable $r_{j+1}=ur_j$, $u\in [0,1]$. By~\cite[Eq.~15.8.1]{Daalhuis}, this term is also equal, up  to a positive multiplicative constant to
$$\left( \frac{r_j}{r_{j-1}} \right)^{2\alpha-1} (r_{j-1}-r_j)^{3\alpha-2} {_2}F_1(3\alpha-1,2\alpha-1;2\alpha,\frac{r_j}{r_{j-1}}). $$
For $\alpha<2/3$, we use~\eqref{max_confluent} to get that this term is dominated by $(r_{j-1}-r_j)^{3\alpha-2}$ and is thus integrable  with respect to $r_j\in [0,r_{j-1}]$. For $\alpha=2/3$, we use~\cite[Eq.~15.4.21]{Daalhuis} to get the integrability. 

We now upper bound $I_3$.  We use that $\phi$ is bounded, so that we have to upper bound 
$$\int_0^{r_j}\int_0^{r_{j+1}}\dots\int_0^{r_{m-1}} |\ell_{i} \alpha-1| (r_j-r_i)^{\ell_{i} \alpha-2} \prod_{i'>j, i'\not = i} (r_{\alpha(i')}-r_{i'})^{\ell_{i'} \alpha-1}
   dr_m\dots dr_{j+1}.$$
When $\ell_i\ge 2$, we can upper bound this integral in the same way as $I_1$. We now consider $\ell_i= 1$ and  integrate from $dr_m$ to $dr_{i+1}$ to get the upper bound
 \begin{align*}
 &  \frac{T^{ \sum_{k=i+1}^m \alpha \ell_k}}{\prod_{k=i+1}^m \alpha \ell_k}   \int_0^{r_j}\dots\int_0^{r_{i-2}}\left( \int_0^{r_{i-1}} (1-\alpha) (r_j-r_i)^{\alpha-2}  dr_i \right) \prod_{i>i'>j} (r_{\alpha(i')}-r_{i'})^{\ell_{i'}\alpha-1} dr_{i-1} \dots dr_{j+1}\\
&\le  \frac{T^{ \sum_{k=i+1}^m \alpha \ell_k}}{\prod_{k=i+1}^m \alpha \ell_k}   \int_0^{r_j}\int_0^{r_{j+1}}\dots\int_0^{r_{i-2}} (r_j-r_{i-1})^{\alpha-1} \prod_{i>i'>j} (r_{\alpha(i')}-r_{i'})^{\ell_{i'}\alpha-1} dr_{i-1} \dots dr_{j+1}.
 \end{align*} 
Using that $2xy\le x^2+y^2$,  we get $$\int_0^{r_{i-2}} (r_j-r_{i-1})^{\alpha-1} (r_{\alpha(i-1)}-r_{i-1})^{\ell_{i-1}\alpha-1} dr_{i-1}\le \frac 12 \left(\frac{T^{2\alpha-1}}{2\alpha-1}+ \frac{T^{2 \ell_{i-1}\alpha-1}}{2\alpha\ell_{i-1}-1} \right)$$
and then integrating from $dr_{i-2}$ to $dr_j$, we get the result. 

{\bf Case 2: $I>j$}. We make a change of variable in the integral to differentiate with respect to $r_j$. Namely, we set $(r_{j+1},r_{j+2},\dots,r_I)=(u_{j+1}r_j,u_{j+2}u_{j+1}r_j,\dots,\prod_{k=j+1}^{I}u_k r_j )$ with $u_{j+1},\dots,u_I \in (0,1)$. We first observe that (using the definition of $I$)
\begin{align*}
  \prod_{i>j} (r_{\alpha(i)}-r_i)^{\ell_i \alpha-1}&=\prod_{I\ge i>j} (r_{i-1}-r_i)^{\ell_i \alpha-1}\prod_{i>I} (r_{\alpha(i)}-r_i)^{\ell_i \alpha-1} \\
&= r_j^{\sum_{i=j+1}^I (\ell_i \alpha-1) } \prod_{I\ge i>j} u_i^{\sum_{k=i+1}^I (\ell_k \alpha-1)} (1-u_i)^{\ell_i \alpha-1}       \prod_{i>I} (r_{\alpha(i)}-r_i)^{\ell_i \alpha-1},
\end{align*}
and noting that the Jacobian is equal to $r_j^{I-j}u_{j+1}^{I-j-1}\dots u_{I-2}^2 u_{I-1}$, we obtain
\begin{align*}
  H_j(r_1,\dots,r_j)&=r_j^{\sum_{i=j+1}^I \ell_i \alpha } \int_0^{1}\dots \int_0^1  \int_0^{ \prod_{k=j+1}^{I}u_k r_j  } \int_0^{r_{I+1}} \dots\int_0^{r_{m-1}}  \phi({\bf ru}) \times \\&  \prod_{I\ge i>j} u_i^{\sum_{k=i+1}^I \ell_k \alpha} (1-u_i)^{\ell_i \alpha-1}       \prod_{i>I} ({\bf ru}_{\alpha(i)}-r_i)^{\ell_i \alpha-1} dr_m\dots dr_{I+1}du_I\dots  du_{j+1}    
\end{align*}
with ${\bf ru} = (r_1,\dots,r_j, u_{j+1}r_j, \dots,\prod_{k=j+1}^{I}u_k r_j,r_{I+1},\dots,r_m)$. Note that all the integrals between $0$ and $1$ are well defined and finite since they have a density of beta type. 
We then set \\ $\widetilde{\bf ru } = (r_1,\dots,r_j, u_{j+1}r_j, \dots,\prod_{k=j+1}^{I}u_k r_j,\prod_{k=j+1}^{I}u_k r_j,r_{I+2},\dots,r_m)$ and  get 
\begin{align*}
&\partial_j  H_j(r_1,\dots,r_j)\\
&=\left(\sum_{i=j+1}^I \ell_i \alpha \right) r_j^{\sum_{i=j+1}^I \ell_i \alpha -1} \int_0^{1}\dots \int_0^1  \int_0^{ \prod_{k=j+1}^{I}u_k r_j  }  \int_0^{r_{I+1}}\dots\int_0^{r_{m-1}}  \phi({\bf ru}) \times \\&  \prod_{I\ge i>j} u_i^{\sum_{k=i+1}^I \ell_k \alpha} (1-u_i)^{\ell_i \alpha-1}       \prod_{i>I} ({\bf ru}_{\alpha(i)}-r_i)^{\ell_i \alpha-1}  dr_m\dots dr_{I+1}du_I\dots  du_{j+1}\\
&+ r_j^{\sum_{i=j+1}^I \ell_i \alpha } \int_0^{1}\dots \int_0^1  \int_0^{ \prod_{k=j+1}^{I}u_k r_j  } \int_0^{r_{I+2}} \dots\int_0^{r_{m-1}}  \phi(\widetilde{\bf ru}) \times \\&  \prod_{I\ge i>j} u_i^{1+\sum_{k=i+1}^I \ell_k \alpha} (1-u_i)^{\ell_i \alpha-1}       \prod_{i>I} (\widetilde{\bf ru}_{\alpha(i)}-r_i)^{\ell_i \alpha-1} dr_m \dots dr_{I+2} du_I \dots du_{j+1}  
\\
&+r_j^{\sum_{i=j+1}^I \ell_i \alpha } \int_0^{1}\dots \int_0^1  \int_0^{ \prod_{k=j+1}^{I}u_k r_j  }  \int_0^{r_{I+1}}\dots\int_0^{r_{m-1}}  \left[ \sum_{p=j}^I  \left(\prod_{k=j+1}^{p} u_k\right) (\partial_p \phi)({\bf ru}) \right] \times \\&  \prod_{I\ge i>j} u_i^{\sum_{k=i+1}^I \ell_k \alpha} (1-u_i)^{\ell_i \alpha-1}       \prod_{i>I} ({\bf ru}_{\alpha(i)}-r_i)^{\ell_i \alpha-1}  dr_m\dots dr_{I+1}du_I\dots  du_{j+1} \\
&+ r_j^{\sum_{i=j+1}^I \ell_i \alpha } \int_0^{1}\dots \int_0^1   \int_0^{ \prod_{k=j+1}^{I}u_k r_j  }  \int_0^{r_{I+1}}\dots\int_0^{r_{m-1}}  \phi({\bf ru}) \times \prod_{I\ge i>j} u_i^{\sum_{k=i+1}^I \ell_k \alpha} (1-u_i)^{\ell_i \alpha-1}   
 \\&  \times  \sum_{i>I: j\le \alpha(i)\le I }  (\ell_i \alpha-1) \frac{{\bf ru}_{\alpha(i)}}{r_j} ({\bf ru}_{\alpha(i)}-r_i)^{-1}\prod_{i'>I} ({\bf ru}_{\alpha(i')}-r_{i'})^{\ell_{i'} \alpha-1}  dr_m\dots dr_{I+1}du_I\dots  du_{j+1} \\
 &=: J_1+J'_1+ J_2+ J_3.
\end{align*}
We can then upper bound $J_1+J'_1$ and $J_3$ with similar arguments as, respectively, $I_1$ and $I_3$. We now focus on $J_2$ which is similar to $I_2$, but we prefer to give more details since its analysis is more involved. We use Proposition~\ref{prop_phi} to get
$$ \left| \sum_{p=j}^I  \left(\prod_{k=j+1}^{p} u_k\right) (\partial_p \phi)({\bf ru}) \right|\le C \left( (r_{j-1}-r_j)^{2\alpha-2} +r_j^{2\alpha-2} {\bf A}(u) + \left(\prod_{k=j+1}^{I}u_k r_j -r_{I+1}\right)^{2\alpha-2} \right),$$
with ${\bf A}(u)=(1-u_{j+1})^{2\alpha-2}+\dots +\left(\prod_{k=j+1}^{I}u_k\right)^{2\alpha-2}(1-u_{I})^{2\alpha-2}$.
As for $I_2$, we first integrate from $dr_m$ to $dr_{I+2}$ and get 
\begin{align*}
  |J_2|&\le C \frac{T^{ \sum_{k=I+2}^m \alpha \ell_k}}{\prod_{k=I+2}^m \alpha \ell_k}  r_j^{\sum_{i=j+1}^I \ell_i \alpha } \int_0^{1}\dots \int_0^1  \int_0^{ \prod_{k=j+1}^{I}u_k r_j  }    \prod_{I\ge i>j} u_i^{\sum_{k=i+1}^I \ell_k \alpha} (1-u_i)^{\ell_i \alpha-1}  \times \\&   \left( (r_{j-1}-r_j)^{2\alpha-2} +r_j^{2\alpha-2} {\bf A}(u) + \left(\prod_{k=j+1}^{I}u_k r_j -r_{I+1}\right)^{2\alpha-2} \right)  ({\bf ru}_{\alpha(I+1)}-r_{I+1})^{\ell_{I+1} \alpha-1}   \\
  &\phantom{*****************************************} dr_{I+1}du_I  \dots du_{j+1}.
\end{align*}
If $\ell_{I+1}\ge 2$, we bound $ ({\bf ru}_{\alpha(I+1)}-r_{I+1})^{\ell_{I+1} \alpha-1}\le T^{\ell_{I+1} \alpha-1}$ and the integral with respect to $r_I+1$ is upper bounded by 
 $C((r_{j-1}-r_j)^{2\alpha-2} +r_j^{2\alpha-2}+  \left(\prod_{k=j+1}^{I}u_k r_j\right)^{2\alpha-1})$, which is integrable with respect to $r_j \in  [0,r_{j-1}]$. In the case $\ell_{I+1}= 1$, since the terms $(r_{j-1}-r_j)^{2\alpha-2}$ and $r_j^{2\alpha-2} {\bf A}(u)$ do not depend on $r_{I+1}$,  their corresponding integrals are easily upper bounded by $C((r_{j-1}-r_j)^{2\alpha-2} +r_j^{2\alpha-2})$, which is integrable with respect to $r_j \in  [0,r_{j-1}]$. We thus focus on 
$$\int_0^{r_I}  (r_I -r_{I+1})^{2\alpha-2} ({\bf ru}_{\alpha(I+1)}-r_{I+1})^{ \alpha-1} dr_{I+1}.$$
If $\alpha>2/3$, this integral is upper bounded by $\frac{T^{3\alpha-2}}{3\alpha-2}$, and we conclude easily. Otherwise, by a change of variable this integral is equal, up to multiplicative constant, to
$$ \left(\frac{\prod_{k=j+1}^{I}u_k r_j}{{\bf ru}_{\alpha(I+1)}}\right)^{2\alpha-1}  \left( {\bf ru}_{\alpha(I+1)} - \prod_{k=j+1}^{I}u_k r_j \right)^{3\alpha-2} {_2}F_1\left(3\alpha-1,2\alpha-1;2\alpha,\frac{\prod_{k=j+1}^{I}u_k r_j}{{\bf ru}_{\alpha(I+1)}}\right).$$ 
 We conclude as for $J_2$.
\end{proof}

\bibliographystyle{abbrv}
\bibliography{Rough_CIR}
\end{document}